\documentclass[12pt]{article}
 \usepackage{amsmath,amsfonts,amsthm,amssymb}

\def\fullpage
{
\addtolength{\topmargin}{-1.5 cm}
\addtolength{\oddsidemargin}{-1.5 cm}
\addtolength{\textwidth}{+3 cm}
\addtolength{\textheight}{+3 cm}
}
\fullpage

\newtheorem{theorem}{Theorem}[section]
\newtheorem{lemma}[theorem]{Lemma} 
\newtheorem{corollary}[theorem]{Corollary}
 
\theoremstyle{definition}
\newtheorem*{remark}{Remark}
\newtheorem*{acknowledgements}{Acknowledgements}
\numberwithin{equation}{section}

\newcommand{\ex}{\E}
\newcommand{\eqn}[1]{\eqref{#1}}
\newcommand{\aas}{\whp} 
\newcommand{\remove}[1]{}

\newcommand{\eq}{\begin{equation}}
\newcommand{\en}{\end{equation}}

\newcommand\eps{\epsilon}

\renewcommand\epsilon{\varepsilon}
\newcommand\Var{\operatorname{Var}}
\newcommand\E{\operatorname{\mathbb E{}}}
\newcommand\pr{\operatorname{\mathbb P{}}}
\newcommand\bb[1]{\bigl(#1\bigr)}

\newcommand\whp{whp}
\newcommand\Po{\operatorname{Po}}
\newcommand\diam{\operatorname{diam}}
\renewcommand\Pr{{\mathbb P}}
\newcommand\Ga{\Gamma}
\newcommand\la{\lambda}
\newcommand\las{\lambda_\star}
\newcommand\bp{{\mathfrak X}}
\newcommand\ceil[1]{\lceil#1\rceil}

\newcommand\floor[1]{\lfloor#1\rfloor}
\newcommand\Gat{\Gl{t}}

\newcommand\Gatm{\Glm{t}}

\newcommand\Gl[1]{G_{\le #1}}
\newcommand\Glm[1]{G_{\le #1}^0}
\newcommand\isom{\cong}

\newcommand\dd{\,{\mathrm d}}
\newcommand\Bi{\mathrm{Bi}}
\newcommand\amax{a_{\mathrm{max}}}
\newcommand\XP{{\tilde X}^+}
\newcommand\YP{{\tilde Y}^+}
\newcommand\tY{{\tilde Y}}
\newcommand\bpP{{\tilde \bp}^+}
\newcommand\Op{O_{\mathrm p}} 
\newcommand\littleop{o_{\mathrm p}} 
\newcommand\downto{\searrow}
\newcommand\noproof{\hfill$\Box$}
\newcommand\RP{{\mathbb R}^+}
\newcommand\Y{\mathcal{Y}}
\newcommand\cc{{\mathrm{c}}}
\newcommand\cs{{\tilde s}}
\newcommand\Ts{{T^\star}}
\newcommand\bx{{\bar x}}
\newcommand\by{{\bar y}}
\newcommand\rx{{\bf x}}
\newcommand\ry{{\bf y}}
\newcommand\Tht{{\widetilde\Theta}}
\newcommand\Ot{{\widetilde O}}
\newcommand\La{\Lambda}
\newcommand\surv{{\mathcal S}}
\newcommand\tdef{{\mathcal S}_\omega}
\newcommand\ind[1]{{{\bf 1}_{#1}}}
\newcommand\refSS[1]{Subsection~\ref{#1}}
\newcommand\refSSs[1]{Subsections~\ref{#1}}
\newcommand\gba{{\mathcal B}_1}
\newcommand\gbb{{\mathcal B}_2}
\newcommand\tA{{\tilde A}}
\newcommand\tB{{\tilde B}}

\newcommand\tN{{\tilde N}}
\newcommand\bi{{\bf i}}
\newcommand\PP{{\mathcal P}}
\newcommand\R{{\mathbb R}}
\newcommand\FF{F}
\newcommand\tF{{\tilde F}}
\newcommand\bceil[1]{\left\lceil #1 \right\rceil}
\newcommand\bfloor[1]{\left\lfloor #1 \right\rfloor}
\newcommand\bone{B_1}
\newcommand\boneb{B_1^0}
\newcommand\btwo{B_2}
\newcommand\mut{\widetilde\mu}
\newcommand\muh{\widehat\mu}
\newcommand\mun{\mu}
\newcommand\kmax{k_{\max}}
\newcommand\wtE{\widetilde E}
\newcommand\fp[1]{\left\{ #1\right\}}
\newcommand\eventand[2]{#1,\,#2}
\newcommand\Eri{E_{\rho_i}}
\newcommand\wN{w}

\begin{document}

\title{The diameter of sparse random graphs}

\author{Oliver Riordan%
\thanks{Mathematical Institute, University of Oxford, 24--29 St Giles', Oxford OX1 3LB, UK.
E-mail: {\tt riordan@maths.ox.ac.uk}.}
\and  Nicholas Wormald%
\thanks{Department of Combinatorics and Optimization, University of Waterloo,  Waterloo ON, Canada.
E-mail: {\tt  nwormald@uwaterloo.ca}.    Supported by the  Canada Research Chairs Program and NSERC.}
}
\date{5th September 2010
}

\maketitle

\begin{abstract}
In this paper we study the diameter of the random graph
$G(n,p)$, i.e., the
largest finite distance between two vertices, for
a wide range of functions $p=p(n)$. For $p=\la/n$ with $\la>1$ constant
we give a simple proof of an essentially best possible result,
with an $\Op(1)$ additive correction term.
Using similar techniques, we establish two-point concentration
in the case that $np\to\infty$.
For $p=(1+\eps)/n$ with $\eps\to 0$,
we obtain a corresponding result that applies
 all the way down to the scaling
window of the phase transition,
with an $\Op(1/\eps)$ additive correction term whose 
(appropriately scaled) limiting distribution we describe.
Combined with earlier results, our new results complete the determination of the diameter of the random graph $G(n,p)$ to an accuracy of the order of its standard deviation (or better), for all functions $p=p(n)$.
Throughout we use branching process methods, rather than the more common
approach of separate analysis of the 2-core and the
trees attached to it.
\end{abstract}

\section{Introduction and main results}\label{sec_intro}

Throughout, we write $\diam(G)$ for the {\em diameter} of a graph $G$, meaning the
largest graph distance $d(x,y)$ between two vertices $x$ and $y$
in the same component of $G$:
\[
 \diam(G)=\max\{d(x,y): x,y\in V(G),\, d(x,y)<\infty\},
\]
where, as usual, $V(G)$ denotes the vertex set of $G$.
In this paper we shall study the diameter of the random graph $G(n,p)$
with vertex set $[n]=\{1,2,\ldots,n\}$, where each possible edge is present
with probability $p=p(n)$, independently of the others.
For certain functions $p=p(n)$, tight bounds on the diameter of $G(n,p)$ are known; our main
aim is to prove such bounds for all remaining functions.
In particular, in the special case $p=\la/n$ with $\la>1$ constant we shall
determine the diameter up to an additive error term that is bounded in probability,
where earlier results achieved only $o(\log n)$. A secondary aim is to 
present a particularly simple proof in this case.
All our results apply just as well to $G(n,m)$;
in the range of parameters
we consider there is essentially no difference between the models. More precisely, although the
results for one model do not obviously transfer to the other, the proofs for $G(n,m)$ are 
essentially the same.

We treat three ranges of $p=p(n)$ separately: $p=\la/n$ with $\la>1$ constant,
$np\to\infty$ but with an upper bound on the growth rate that extends
well into the range covered by classical results, and finally $p=(1+\eps)/n$ with $\eps(n)\to 0$
but $\eps^3n\to\infty$.
In each case, our
analysis investigates the neighbourhoods of vertices, and has three
components or phases: `early growth' --- we study the distribution of the
number of vertices at distance $t$ from a given vertex $v$ when $t$ is small; `regular
growth' in the middle --- we show that the number of vertices at
distance $t$ is very likely to grow regularly
once the neighbourhoods have become `moderately
large'; `meeting up' --- we show that the distance between two
vertices is almost determined by the times their respective
neighbourhoods take to become `large'. This is eventually translated
into a result on the diameter.
 
Our overall plan is made possible by the very accurate information we obtain on the first phase (early growth). The main approach for this is to compare the neighbourhoods
of a vertex of
$G(n,\la/n)$ with the 
standard Poisson Galton--Watson branching process $\bp_\la=(X_t)_{t\ge 0}$;
this starts with a
single particle in generation 0, and each particle in generation
$t$ has a Poisson $\Po(\la)$ number of children in the next
generation, independently of the other particles and of the history.

A particle in the process $\bp_\la$ {\em survives} if it has descendants in all
later generations; the process {\em survives} if the initial particle
survives. If $\la>1$, then the survival probability $s=\Pr(\forall t: |X_t|>0)$
is the unique positive solution to
\begin{equation}\label{sdef}
 1-s=e^{-\la s}.
\end{equation}
Since particles in generation 1 survive independently of each other, the
number of such particles that survive has a $\Po(s\la)$ distribution,
the number that die has a $\Po((1-s)\la)$ distribution, and these
numbers are independent.
It follows that conditioning on the process dying, we obtain
again a Poisson Galton--Watson process $\bp_{\las}=(X_t^-)_{t\ge 0}$, with the `dual' parameter
\begin{equation}\label{lasdef}
 \las=\la(1-s),
\end{equation}
which may also be characterized as
the solution $\las<1$ to
\begin{equation}
 \las e^{-\las}=\la e^{-\la}.
\end{equation}
This parameter
is crucial to understanding the diameter of $G(n,\la/n)$.
For this and other basic branching process results, see, for example,
Athreya and Ney~\cite{AN}.

Throughout the paper we use standard notation for probabilistic asymptotics
as in~\cite{JLR}. In particular, $X_n=\littleop(f(n))$ means $X_n/f(n)$ converges
to $0$ in probability, and $X_n=\Op(f(n))$
means $X_n/f(n)$ is bounded in probability.

Our first aim is to give a proof of a tight estimate for the diameter
of $G(n,\la/n)$ when $\la>1$ is constant as $n\to \infty$ that is simpler than
our result for the general case, and also compares favourably with the existing proofs of much weaker bounds
for the more general models discussed below.
\begin{theorem}\label{th_const}
Let $\la>1$ be fixed, and let $\las<1$ satisfy $\las e^{-\las}=\la e^{-\la}$.
Then
\begin{equation}\label{dform}
 \diam(G(n,\la/n)) = \frac{\log n}{\log\la} + 2\frac{\log n}{\log(1/\las)} +\Op(1).
\end{equation}
\end{theorem}
As usual, we say that an event holds {\em with high probability}, or {\em whp}, if
its probability tends to 1 as $n\to\infty$. Theorem~\ref{th_const} simply says
that, for any $K=K(n)\to\infty$, the diameter is whp within $K$ of the sum of
the first two terms on the right of \eqref{dform}.

The proof of Theorem~\ref{th_const} is fairly simple, and will be given in Section~\ref{sec_const}.

Turning to the case $\la=\la(n)\to\infty$, we obtain the following result, 
proved in Section~\ref{sec_inf} using essentially the same method,
although there are various additional complications.
\begin{theorem}\label{th_inf}
Let $\la=\la(n)$ satisfy $\la\to\infty$ and $\la\le n^{1/1000}$,
and let $\las<1$ satisfy $\las e^{-\las}=\la e^{-\la}$.
Then $\diam(G(n,\la/n))$ is two-point concentrated:  
there exists a function $f(n,\la)$ satisfying
\[
 f(n,\la) = \frac{\log n}{\log\la} + 2\frac{\log n}{\log(1/\las)} +O(1)
\]
such that \aas\ $ \diam(G(n,\la/n))  \in\{f(n,\la),f(n,\la)+1\}$.
Furthermore, for any $\eps>0$ and any function $\la$ such that,
for large $n$,
neither
$\log n/\log(1/\las)$ nor $\log n/\log\la$ is within $\eps$ of an integer, we have
\begin{equation}\label{normalform}
 \diam(G(n,\la/n)) = \bceil{\frac{\log n}{\log\la}} + 2\bfloor{\frac{\log n}{\log(1/\las)}}+1
\end{equation}
whp.
\end{theorem}
Bruce Reed has independently announced a related result, in joint work with Nikolaos Fountoulakis;
the details are still to appear. We believe that the methods used are quite different.

The main interest of Theorem~\ref{th_inf} is when $\la$ tends to infinity fairly slowly;
if $\la$ grows significantly faster than $\log n$, then the situation
is much simpler, and much more precise results are known.
Indeed, when $\la/(\log n)^3\to\infty$, Bollob\'as~\cite{B}
showed concentration of the diameter on at most two values, and found
the asymptotic probability of each value.
In the light of this result we would lose nothing by assuming that
$\la\le (\log n)^4$, say; however, the bound $\la\le n^{1/1000}$ turns out to be
enough for our arguments.

The bulk of the paper
is devoted to the case of expected degree tending to $1$,
where we prove the following result.

\begin{theorem}\label{th_eps}
Let $\eps=\eps(n)$ satisfy  $0<\eps<1/10$ and $\eps^3n\to\infty$.
Set $\la=\la(n)=1+\eps$, and let
$\las<1$ satisfy $\las e^{-\las}=\la e^{-\la}$.
Then
\begin{equation}\label{deps}
 \diam(G(n,\la/n)) = \frac{\log(\eps^3 n)}{\log\la} + 2\frac{\log(\eps^3 n)}{\log(1/\las)} +\Op(1/\eps).
\end{equation}
\end{theorem}
Our method in fact gives a description of the limiting distribution of
the final correction term (after rescaling); see Theorem~\ref{th:dist1}.

A weaker form of Theorem~\ref{th_eps} has been obtained independently by
Ding, Kim, Lubetzky and Peres~\cite{DKLP_anat,DKLP_diam}; see the Remark below.

In the rest of this section we briefly discuss the results above and their relationship to earlier
work.

Theorem~\ref{th_const} is best possible in the following sense: it is not hard to see
that the diameter cannot be concentrated on a
set of values with bounded size as $n\to\infty$.
Indeed, given any (labelled) graph $G$ with diameter $d$ and at least two isolated
vertices, let $G'$ be constructed from $G$ by taking a
 path $P$ of length
$d$ joining two vertices at maximal distance in $G$,
and adding an edge joining each end of $P$
to an isolated vertex. Each graph $G'$ constructed in this way contains
a unique pair of vertices at maximal distance $d+2$, and $G$ may be recovered uniquely from $G'$ by deleting
the (unique) edges incident with these vertices. Restricting our attention
to graphs $G$ with $\Theta(n)$ isolated vertices, the relation $(G,G')$
is thus $1$ to $\Theta(n^2)$. Since the probability of $G'$ in the model
$G(n,p)$, $p=\la/n$, is equal to the probability of $G$ multiplied by $p^2/(1-p)^2=\Theta(1/n^2)$, it follows easily
that for any $d$ we have
\[
 \Pr\bb{\diam(G(n,\la/n))=d+2} \ge \Theta(1)\Pr\bb{\diam(G(n,\la/n))=d} -o(1);
\]
the $o(1)$ term comes from the possibility that $G(n,\la/n)$ has 
fewer than $\Theta(n)$ isolated vertices. It follows that $\diam(G(n,\la/n))$
cannot be concentrated on a finite set of values.
In fact, our methods allow us to obtain the limiting distribution
of the $\Op(1)$ correction term in~\eqref{dform},
although this is rather complicated
to describe; we return to this briefly in Section~\ref{sec_dist}.

A much weaker form of Theorem~\ref{th_const}, with a $o(\log n)$
correction term, is a special case of a result of
Fernholz and Ramachandran~\cite{FR:diam} for random graphs with a given
degree sequence, and also of a result of Bollob\'as, Janson and Riordan~\cite[Section 14.2]{BJR}
for inhomogeneous random graphs with a finite number of vertex types.
We shall follow the ideas of~\cite{BJR} to some extent, although the present
simpler context allows us to take things much further, obtaining a much more
precise result. Earlier, Chung and Lu~\cite{ChungLu:diam} also studied $\diam(G(n,\la/n))$,
$\la>1$ constant, but their results were not strong enough to give the correct
asymptotic form. Indeed, they conjectured that, under suitable
conditions, the diameter is approximately $\log n/\log\la$, as one might initially expect.

For the subcritical case, which is much simpler,
{\L}uczak~\cite{Luczak:diam} proved very precise
results: he showed, for example, that if $\eps\to 0$ and $\eps^3n\to\infty$, then
the subcritical random graph $G=G(n,(1-\eps)/n)$ satisfies
\begin{equation}\label{Luczak}
 \diam(G) = \frac{ \log(2\eps^3n) +\Op(1)}{-\log(1-\eps)};
\end{equation}
see his Theorem 11(iii), and note that the exponent $2$ instead of $3$ appearing there
is a typographical error. (He also proved a simple formula for the limiting distribution
of the $\Op(1)$ term -- the probability that it exceeds a constant $\rho$ tends to
$1-\exp(-e^{-\rho})$ as $n\to\infty$; the limiting
distribution in the present supercritical case turns out to be much more complicated.)
{\L}uczak's results are effectively the last word on the subcritical case, which we shall
not discuss further.

Returning to constant $\la>1$, the lack of concentration on a finite number 
of points
contrasts with the case of random $d$-regular graphs studied by Bollob\'as and
Fernandez de la Vega~\cite{BFdlV}, 
who established concentration on a small set of values
in this case. Sanwalani and Wormald~\cite{SW} have recently shown two-point concentration.
(More precisely, they prove one-point concentration for almost all $n$, and
for the remaining $n$ find the probabilities of the two likely values within $o(1)$.)
Note that the diameter in this case
is simply $\log (n\log n)/\log(d-1) +\Op(1)$ for $d\ge 3$;  as we shall see, the behaviour
of the two models for this question is very different. Usually, $G(n,\la/n)$
is much simpler to study than a random regular graph, but here
there are additional complications corresponding to the $2\log n/\log(1/\las)$ term in \eqref{dform}.

Let us briefly mention a few related results for other
random graph models. Perhaps the earliest results
in this area are those of Burtin~\cite{Bu1,Bu2} and Bollob\'as~\cite{B}.
Turning to results determining the asymptotic diameter when
the average degree is constant, one of the first is the result of
Bollob\'as and Fernandez de la Vega~\cite{BFdlV} for $d$-regular random graphs mentioned
above; another is that of Bollob\'as and Chung~\cite{BC},
finding the asymptotic diameter of a cycle plus a random matching, which
is again logarithmic.
Later it was shown by `small subgraph conditioning' (see~\cite{W})
that
for such graphs any whp statements are essentially the same as for the
uniform model of random 3-regular graphs.  The same goes for a
variety of other random regular graphs constructed by superposing
random regular graphs of various types.
For a rather different model, namely
a precise version of the Barab\'asi--Albert `growth with preferential
attachment' model, Bollob\'as and Riordan~\cite{BRdiam} obtained a (slightly)
\emph{sub}logarithmic diameter, contradicting the logarithmic diameter
suggested by Barab\'asi, Albert and
Jeong~\cite{N401,BAJp2} (on the basis of computer experiments)
and Newman, Strogatz and Watts~\cite{nswpp} (on the basis of heuristics).

More recently, related results, often concerning the `typical' distance between
vertices, rather than the diameter, have been proved
by many people, for various models. A few examples are the results of
Chung and Lu~\cite{ChungLu:exp_dist,ChungLu:exp_dist2}, and van den
Esker, van der Hofstad, Hooghiemstra, van Mieghem and
Znamenski~\cite{EHHZ,HHM_fvar,HHZ_infvar}; for a discussion of
related work see~\cite{HHM_fvar}, for example.

The formula \eqref{dform} is easy to understand intuitively: 
typically, the size of the $d$-neighbourhood of a vertex (the set
of vertices at distance $d$) grows by a factor of $\la$ at
each step (i.e., as $d$ is increased by one).
Starting from two typical vertices, taking
$\log(\sqrt{n})/\log\la$ steps from each, the neighbourhoods reach size
about $\sqrt{n}$; at around this point the neighbourhoods are likely
to overlap, so the typical distance between vertices is $\log n/\log \la$.
The second term in \eqref{dform} comes from exceptional 
vertices whose neighbourhoods take some time to start expanding, or, equivalently,
from the few very longest trees attached to (typical vertices of) the \emph{2-core}
of $G(n,\la/n)$,
the maximal subgraph with no vertices of degree 0 or 1.
It is well known that the trees hanging off the 2-core of $G(n,\la/n)$
have roughly the distribution of the branching process
$\bp_{\las}$; hence, some of these
trees will have height roughly $\log n/\log(1/\las)$, and it turns out that
the diameter arises by considering two trees of (almost) maximal height
attached to vertices in the 2-core at (almost) typical distance.

Although we shall use the 2-core viewpoint later, its use has an intrinsic difficulty caused by the significant variation in the distances between
vertices in the 2-core. One can view the variation in the distance between two
random vertices of $G=G(n,\la/n)$
as coming from three sources: (i) variation in the distances to the 2-core,
(ii) variation in the times the neighbourhoods in the 2-core take to start expanding, and (iii)
variation in the time the neighbourhoods of the two vertices take to join up
once they have reached a certain size. 
An advantage of our approach is that it seamlessly integrates (i) and (ii),
by looking simply at neighbourhood growth in the whole graph $G$. 
Taking this viewpoint, the dual parameter $\las$ arises as follows:
let $X_t^+\subset X_t$
be the set of particles of $\bp_\la$ that survive (have descendants in all future
generations). Then $X_0^+$ contains the initial
particle with probability $s$, and is empty otherwise. Moreover, conditioning
on a particle being in $X_t^+$ is exactly the same as conditioning
on at least one its children surviving, so the number of surviving children
then has the distribution $Z=Z_\la$  
of a $\Po(s\la)$ random variable conditioned
to be at least $1$. Hence, $(X_t^+)$ is again a Galton--Watson branching
process, but now with offspring distribution $Z$, and $X_0^+$ either empty
or, with probability $s$, consisting of a single particle.
Note that 
\begin{equation}\label{plus}
 \Pr(Z=1) = \frac{s\la e^{-s\la}}{1-e^{-s\la}} = \frac{s\la(1-s)}s =\las,
\end{equation}
using \eqref{sdef} and \eqref{lasdef}.
Hence, the probability that $X_t^+$ consists of a single particle,
given that the whole process survives, is exactly $\las^t$.
Roughly speaking, this event corresponds to the branching process staying
`thin' for $t$ generations, i.e., the neighbourhood growth process taking
time $t$ to `get going'.
In the next section we shall prove a more precise version of this statement.

Turning to Theorem~\ref{th_eps},
the form of the diameter here  differs from what one might expect in the presence of the factors
$\eps^3$ inside the logarithms in the numerators.
Very loosely speaking, these factors turn out to be related to the fact that the branching
process survives with probability $\Theta(\eps)$, and then usually has size $\Theta(1/\eps)$
larger than its unconditional expected size, as well as to the fact
that, roughly speaking, it takes on the order of $1/\eps$ generations for anything much to happen;
we shall return to this at various points.
An alternative way of thinking about these factors is that the `interesting' structure
of $G(n,p)$ is captured by the {\em kernel}, the graph obtained from the 2-core by
suppressing vertices of degree 2.
The results of {\L}uczak~\cite{Lcyc} or alternatively 
Pittel and Wormald~\cite{PWio}
imply in particular that the number of vertices
in the kernel is asymptotically $8\eps^3 n/6$. 
\begin{remark}
In the first draft of this paper, we obtained a slightly weaker form of Theorem~\ref{th_eps},
giving the same conclusion but requiring an additional assumption,
that $\eps^3n$ grows at least as fast as an explicit {\em extremely} slowly growing function
of $n$ (essentially $\log^* n$, i.e., the minimum $k$ such that the $k$th iterated
logarithm of $n$ is less than $1$).  
This is a less restrictive assumption than what is common in related contexts, that $\eps^3 n$ is at least some power of $\log n$.
Since then,
Ding, Kim, Lubetzky and Peres~\cite{DKLP_anat,DKLP_diam}
have obtained a form of Theorem~\ref{th_eps} with a larger error
term (a multiplicative factor of $1+o(1)$), valid whenever $\eps^3n\to\infty$
and $\eps\to 0$; under these assumptions, $\log\la\sim \eps$
and $\log(1/\las)\sim\eps$, so the diameter is $(3+\littleop(1))\log(\eps^3n)/\eps$.
 Their approach, based around the 2-core and kernel,
is very different to ours.
Seeing this paper stimulated us to remove the unnecessary restriction on $\eps$; it turned
out that one simple observation (Lemma~\ref{noshort} below) was the main
missing ingredient.
Using this lemma, it became possible to simplify some
of our original arguments and, with a little further technical work, to extend them
to the entire weakly supercritical range.
\end{remark}

We remark also that {\L}uczak and Seierstad~\cite{LuczakSeierstad} have obtained a `process
version' of Theorem~\ref{th_eps}. This gives much weaker bounds on the diameter,
differing by a constant factor, but can be applied to the random graph process
to show (roughly speaking) that whp these bounds hold simultaneously for the entire
range of densities considered in Theorem~\ref{th_eps}.

In Theorem~\ref{th_eps}, the condition $\eps\le 1/10$ is imposed simply for convenience;
this may be weakened to $\eps=O(1)$ without problems. However, for $\eps=O(1)$
bounded away from zero one can instead apply Theorem~\ref{th_const}: it
is not hard to check that the constants implicit in the correction term vary smoothly with $\la$,
and so are bounded over any compact set of $\eps>0$.
For this reason, in proving Theorem~\ref{th_eps} we may assume that $\eps\to 0$
as $n\to\infty$; we shall do this whenever it is convenient.

On the other hand,
the condition $\eps n^{1/3}\to\infty$ is almost certainly necessary
for our method to give nontrivial information. If $\eps n^{1/3}$ is bounded,
then we are inside the `window' of the
phase transition, so $G(n,(1+\eps)/n)$ is qualitatively
similar in behaviour to $G(n,1/n)$, and the behaviour of the diameter
is much more complicated than outside the window. For one thing,
there is no longer a unique `giant' component that is much larger
than all other components. Also, the 2-core of each non-tree component
contains only a bounded number of cycles; to study the distribution
of the diameter accurately, one needs to study the distribution of the lengths
of these cycles, which is very different from the situation 
with supercritical graphs. 
Nachmias and Peres~\cite{NP} showed that
inside the window the diameter of the largest component
is $\Op(n^{1/3})$, with a corresponding lower
bound; more recently, Addario-Berry, Broutin and Goldschmidt~\cite{ABG}
have established convergence in distribution of the rescaled diameter, 
and given a (rather complicated) description of the limit in terms
of continuum random trees.

Finally, as in Theorem~\ref{th_const}, the $\Op(1/\eps)$ correction term in Theorem~\ref{th_eps}
is in some sense best possible. As noted earlier, our method gives a description
of the limiting distribution of this correction term;
see Section~\ref{sec_dist}.

In summary, the results of {\L}uczak~\cite{Luczak:diam} (below the critical window),
Addario-Berry, Broutin and Goldschmidt~\cite{ABG} (inside the window),
Theorem~\ref{th_eps} (above but average degree tending to $1$),
Theorem~\ref{th_const} (constant average degree),
Theorem~\ref{th_inf} (average degree tending to infinity slowly)
and Bollob\'as~\cite{B} (average degree tending to infinity quickly)
together establish tight bounds on the diameter
of $G(n,p)$ throughout the entire range of the parameters.

\section{The case $p=\la/n$, $\la>1$ constant}\label{sec_const}

In this section we shall prove Theorem~\ref{th_const}.
We start by recalling a basic fact about branching processes.

{From} standard branching process results (see, for example,
Athreya and Ney~\cite{AN}),
the martingale $|X_t|/\la^t$ converges almost surely to a random variable $Y=Y_\la$
whose distribution (which depends on $\la$)
is continuous except for mass $1-s$ at $0$, with strictly
positive density on ${\mathbb R}^+$.
Furthermore, $Y=0$ coincides (except possibly on a set of measure zero)
with the event that the branching process dies out.
Since almost sure convergence implies convergence in probability, a trivial
consequence of this is that, for $\la>1$ and $0<c_1<c_2$ all fixed,
\begin{equation}\label{ar}
 \inf_t \Pr\bb{c_1\la^t \le |X_t|\le c_2\la^t} > 0,
\end{equation}
where the infimum is over all $t\ge 1$ such that the interval $[c_1\la^t,c_2\la^t]$
contains an integer.
The following result indicates that unusually small populations in a given generation are typically due (at least, with a significant probability) to a branching process that stays essentially nonbranching (with only small `side branches') until a point where it branches at a typical rate.
\begin{lemma}\label{l1}
Let $\la>1$ be fixed. There are constants $c,C>0$ such that
for every $\omega\ge 2$ and $t\ge 1$ we have
\begin{equation}\label{l1eq}
 c\min\{\las^{t-t_1},1\} \le \Pr\bb{0<|X_t| < \omega} \le C\las^{t-t_1},
\end{equation}
where $t_1=\lfloor\log\omega/\log\la\rfloor$.
\end{lemma}
\begin{proof}
The lemma is essentially a statement about the asymptotics of $Y$
near $0$; this statement
follows, for example, from a result of Harris~\cite{Harris48}.
However, translating back to a statement about $X_t$ rather
than $Y$ would introduce an extra error term, corresponding
to the probability that $X_t/\la^t$ still differs from $Y$
by more than a constant factor when $X_t$ first exceeds
$\omega$, so we shall
give a direct proof.

We start by proving the upper bound. 
Conditioned on $\bp_\la=(X_t)_{t\ge 0}$ dying out, an event
of probability $1-s$, this process
has the distribution of the subcritical process $\bp_{\las}=(X_t^-)_{t\ge 0}$.
Hence,
\[
 \Pr\bb{|X_t|>0,\, \exists t':X_{t'}=\emptyset } 
 = (1-s) \Pr(|X_t^-|>0) \le (1-s)\E(|X_t^-|) = (1-s)\las^t.
\]
Let $p_t=\Pr(|X_t|>0)$. Then
\begin{equation}\label{pts}
 p_t= s+ \Pr\bb{|X_t|>0,\, \exists t':X_{t'}=\emptyset } =s+O(\las^t).
\end{equation}
Let us note for later that the implicit constant is independent of $\la$;
indeed, it may be taken to be $1$.

We may partition $X_1$, the set of children of the initial particle, into two sets:
the set $S$ consisting of those that have descendants $t-1$ generations
later (i.e., in $X_t$), and the set $X_1\setminus S$ of those that do not.
Since the probability that a particle in $X_1$ has one or more descendants
in $X_t$ is $p_{t-1}$, the size of $S$ has a Poisson distribution
with mean $\la p_{t-1}$.
Let us condition on $|X_t|>0$. Then the conditional distribution
of $|S|$ is that of a Poisson distribution with mean $\la p_{t-1}$
conditioned on being at least $1$, and we have 
\[
 \Pr\bb{|S|=1 \bigm| |X_t|>0} = \frac{\la p_{t-1} e^{-\la p_{t-1}}}{1-e^{-\la p_{t-1}}}
 = \frac{\la s e^{-\la s}}{1-e^{-\la s}}(1+O(\la\las^{t-1})) = \las+O(\la\las^t),
\]
using \eqref{pts} and \eqref{plus}. Note for later use that the implicit constant is independent of $\la$ provided $\la>1$ and $\la$ is bounded away from 1.

Let $r_t=\Pr(|X_t| < \omega \mid |X_t|>0)$. If $|X_t|<\omega$,
then every particle in $S$ has fewer than $\omega$ descendants in $X_t$.
Hence,
\begin{eqnarray}
 r_t &\le& \Pr\bb{|S|=1\bigm| |X_t|>0} r_{t-1} + \Pr\bb{|S|>1\bigm| |X_t|>0} r_{t-1}^2 \nonumber\\
 &=& (\las+O(\la\las^t)) r_{t-1}+(1-\las+O(\la\las^t)) r_{t-1}^2\nonumber\\
 &=& r_{t-1}(\las+(1-\las) r_{t-1}) +O(\la\las^t r_{t-1}).\label{rb1}
\end{eqnarray}
Setting $r_t'=r_t/\las^t$ and recalling that $\la$ is constant, we thus have
\begin{equation}\label{rb2}
 r_t'\le r_{t-1}' + \frac{1-\las}{\las}r_{t-1}r_{t-1}'+O(r_{t-1}).
\end{equation}

Using only the trivial inequality
$\Pr(0<|X_t|<\omega_1)\le \Pr(0<|X_t|<\omega)$ for $\omega_1<\omega$,
the upper bound in \eqref{l1eq} for $\omega$ at least some constant $\omega_0$
implies the same bound, with a different constant, for all $\omega\ge 2$.
Thus we may assume that $\omega$ is at least some large constant $\omega_0$,
and hence that $t_1$ is large.
We may also assume $t\ge t_1$.
By~\eqref{ar} we have $\Pr(|X_{t_1}|> \omega)\ge c_0$ for some constant $c_0>0$.
Hence $r_{t_1}\le 1-c_0$ is bounded away from $1$.
Choosing $\omega_0$ large enough, so $\las^t\le \las^{t_1}$ is small, the error term in~\eqn{rb1} can be assumed arbitrarily small relative
to $r_{t-1}$.
Using \eqref{rb1},
and noting that for $t>t_1$ we have $\las+(1-\las)r_{t-1}<\las+(1-\las)(1-c_0)<1$,
it then follows that $r_t$ decreases 
exponentially as $t$ increases from $t_1$, i.e., that there is a $c_1>0$ 
(depending only on $\la$, not on $\omega$) such that $r_{t_1+t}\le e^{-c_1t}$.
Hence, $\sum_{t\ge t_1} r_t$ is bounded (independently of $\omega$).
Using \eqref{rb2}, it follows that there is a constant $C_0$
such that for $t\ge t_1$ we have
$r_t'\le C_0 (r_{t_1}'+1)$. In other words,
\[
 r_t \le C_0\las^{t-t_1} r_{t_1} +C_0\las^t \le C_0(1+\las^{t_1})\las^{t-t_1}
\le 2C_0\las^{t-t_1} = O(\las^{t-t_1}).
\]
Since
\[
 \Pr\bb{0<|X_t|<\omega} \le \Pr\bb{|X_t|<\omega\bigm| |X_t|>0} =r_t,
\]
this completes the proof of the upper bound.

Turning to the lower bound, this is essentially trivial if $t\le t_1$:
in this case, $\Pr(0<|X_t|<\omega)$ is bounded away from 0 by~\eqref{ar}.
We may thus assume that $t> t_1$.
We shall prove the lower bound by considering the following much
more specific event $E$, the event that $|X_{t-t_1}^+|=1$, that
the unique particle $v$ of $X_{t-t_1}^+$ has between $1$ and $\omega-1$
descendants in $X_t$, and that no other particles of $X_{t-t_1}$ have
descendants in $X_t$. Clearly, if $E$ holds then $0<|X_t|<\omega$.

Recalling that $\bp^+=(X_t^+)$ is the set of particles
whose descendants survive forever, any such particle always
has at least one child by definition, and, by \eqref{plus},
has exactly one child with probability $\las$. Thus
\begin{equation}\label{onetrack}
 \Pr(|X_{t-t_1}^+|=1) = s\las^{t-t_1}.
\end{equation}
Given that $|X_{t-t_1}^+|=1$, the number $N_v$ of descendants in $X_t$ of the unique
particle $v$ in $X_{t-t_1}^+$ has the distribution of $|X_{t_1}|$ conditioned
on the whole process surviving.  From \eqref{ar}, 
the (unconditional) probability that $|X_{t_1}|$ is between $\omega/2$ and $\omega-1$, say,
is bounded away from zero, and the conditional probability that $\bp_\la$ survives given this
event is at least $s$. Thus 
\begin{multline*}
 \Pr\bb{N_v<\omega \bigm| |X_{t-t_1}^+|=1} =
\Pr\bb{|X_{t_1}|<\omega \bigm| \forall t: |X_t|>0} \\
 \ge 
\Pr\bb{|X_{t_1}|<\omega,\, \forall t: |X_t|>0}\ge c_2,
\end{multline*}
for some positive constant $c_2$.

It remains to exclude descendants in $X_t$
of other particles in $X_{t-t_1}$. By definition, these particles do not survive.
We may construct $\bp_\la$ as follows: first construct $\bp^+ = (X_t^+)_{t\ge 0}$. Then add
in the particles that die: for each particle in each set $X_r^+$, we must
add an independent copy of $\bp_{\las}$ rooted at this particle.

Given that $|X_{t-t_1}^+|=1$, we have $|X_r^+|=1$ for all $r\le t-t_1$.
The probability that the copy of $\bp_{\las}$ started at time $r$
survives to time $t$ is $\Pr(|X_{t-r}^-|>0) \le \las^{t-r}$.
Since the different copies are independent, the probability that
all die before time $t$ is at least $\prod_{r\le t-t_1} (1-\las^{t-r})$.
Now $\las<1$, so 
$\sum_{r\le t-t_1} \las^{t-r}=O(\las^{t_1}) =O(1)$, and 
$\prod_{r\le t-t_1}(1-\las^{t-r})\ge c_3$, for some $c_3>0$
depending only on $\la$.
Hence, 
\[
 \Pr\bb{0<|X_t|<\omega} \ge s\las^{t-t_1} c_2 c_3 = \Omega(\las^{t-t_1}),
\]
completing the proof of the lemma.
\end{proof}

The above lemma tells us virtually all we need to know about the branching process for the `early growth' part of the proof of
Theorem~\ref{th_const}. The next ingredient for this phase is a lemma
connecting the growth of neighbourhoods in the graph to the branching process.  The branching process model is most relevant if the growing neighbourhood of a vertex remains a tree. To be sure, almost all vertices do not lie on or near a short cycle. However, we cannot simply ignore the exceptional vertices, since a result about the diameter makes a statement
about {\em all} vertices, not just almost all. So we must be a little careful.

We deal with the problem of non-tree neighbourhoods as follows. Given a vertex $x$ of a graph $G$, let $\Ga_t(x)$
be the set of vertices at graph distance $t$ from $x$.
Let $\Gat(x)$
be the subgraph of $G$ induced by $\bigcup_{t'\le t} \Ga_{t'}(x)$, regarded
as a rooted graph with root $x$. We shall explore
the neighbourhoods $\Ga_t(x)$ in the following essentially standard way.
Fix once and for all an order on $V(G)$.
Having found $\Ga_t(x)$ (starting with $t=0$), go through the vertices
of $\Ga_t(x)$ one by one in the predetermined order.
For each vertex $v$ we expose
all edges from $v$ to vertices not yet reached in the exploration; this means we test each potential edge
{\em to an as yet unreached vertex} for its presence; any edges detected are called `uncovered.'
If we uncover an edge $vw$, we add $w$ to $\Ga_{t+1}(x)$.
Of course this process correctly identifies the sets $\Ga_t(x)$.
However,
it only uncovers certain edges: let $\Gatm(x)$ denote
the graph formed by the edges uncovered in our tests exploring
up to $\Ga_t(x)$. Then $\Gatm(x)$ is a tree: it is a spanning
tree in the graph $\Gat(x)$.

In the following results, $X_{\le t}$ denotes the union of generations $0$ to $t$
of the branching process $\bp_\la$, regarded as a rooted tree with root the initial
particle, and $\isom$ denotes isomorphism of rooted trees.

\begin{lemma}\label{cpl}
Let $\la>0$ be fixed. For any rooted tree $T$ with $|T|\le n/2$ we have
\[
 \Pr\bb{\Gatm(x)\isom T} = e^{O(|T|^2/n)}\Pr\bb{X_{\le t}\isom T}
\]
and
\[
 \Pr\bb{\Gat(x)\isom T}  = e^{O(|T|^2/n)}\Pr\bb{X_{\le t}\isom T},
\]
where the implicit constants depend only on $\la$.
\end{lemma}
\begin{proof}
This is well known and easy to prove. The first statement follows from the natural
step-by-step coupling between $\Gatm(x)$ and the branching process,
where each step investigates the children (of a vertex or a particle,
respectively). Suppose we have reached $r-a$ vertices in total
so far. Then the probabilities of finding $a$ vertices in the next step
are $p_1=\binom{n-r+a}{a}(\la/n)^a(1-\la/n)^{n-r}$
and $p_2=e^{-\la}\la^a/a!$
in the two models. The ratio of these probabilities
is
\[
 p_1/p_2 = (n-r+a)_{(a)}n^{-a} (1-\la/n)^{-r} e^\la(1-\la/n)^n
 = e^{O(ar/n+r\la/n + \la^2/n)},
\]
where $x_{(a)}=x(x-1)\cdots(x-a+1)$. The sum of $ar$ or $r$ over all vertices
in the tree is trivially at most $|T|^2$, so it follows that
\begin{equation}\label{cpl1}
 \frac{ \Pr\bb{\Gatm(x)\isom T} }{ \Pr\bb{X_{\le t}\isom T} }
 = \exp\bb{ O(|T|^2/n + \la|T|^2/n + \la^2|T|/n) }.
\end{equation}
Since $\la$ is fixed, this proves the first statement.
 
If $\Gat(x)\isom T$, then $\Gat(x)$ is a tree, so
$\Gatm(x)=\Gat(x)$.  Hence $\Gat(x)\isom T$ implies $\Gatm(x)\isom
T$. Given that $\Gatm\isom T$, the probability that none of the
untested edges between the $|T|$ vertices found is also present is
again $e^{O(|T|^2/n)}$. So the second statement follows from the
first.
\end{proof}

Using Lemmas~\ref{l1} and~\ref{cpl}, we can study the initial rate of
growth of the neighbourhoods of the vertices of $G(n,\la/n)$.  The
first step is to show that these neighbourhoods cannot stay small but
non-empty for too long.
The basic picture is that after about
\begin{equation}\label{t1def}
 t_1=\floor{\log\omega/\log\la}
\end{equation}
steps, we expect a typical vertex neighbourhood to expand to size
approximately $\omega$. It is very unlikely that there are any vertices in the graph whose
neighbourhoods expand to some reasonable size, say around $\log n$, and then fail to expand to
size $\omega$ in roughly the expected time from that point.
However, some unusual vertices take up to
\begin{equation}\label{t0def}
 t_0=\floor{\log n/\log(1/\las)}
\end{equation}
steps before their neighbourhoods expand significantly,
and so take this many more steps than usual
to reach size roughly $\omega$.

We argue this more precisely as follows.

Set $\omega=(\log n)^6$, say, and define $t_0$ and $t_1$ as above.
Let $K=K(n)$ tend to infinity slowly (for instance, slower than $\log \log n$).

For each vertex $x$, let $B_1(x)$ be the `bad' event that $1\le |\Ga_{t'}(x)|<\omega$
holds for all $0\le t'\le t=t_0+t_1+K$.
The event $B_1(x)$ is a disjoint union of events of the form
$\Gatm\isom T$, where each tree $T$ has size
at most $t\omega=o(\sqrt n)$.
Also, the corresponding union of the events $X_{\le t}\isom T$
is the event that $0<|X_{t'}|<\omega$ holds for all $t'\le t$.
Hence, by Lemma~\ref{cpl},
\begin{equation}\label{B1x}
 \Pr(B_1(x)) \sim \Pr\bb{ \forall t'\le t:0<|X_{t'}|<\omega}
 \le \Pr\bb{0<|X_t|<\omega} =O(\las^{t_0+K}),
\end{equation}
where the last step is from Lemma~\ref{l1}.

Let $B_1$ be the event that $B_1(x)$ holds for some $x$. Then
\begin{equation}\label{noB1}
 \Pr(B_1)\le n\Pr(B_1(x))=O(n\las^{t_0+K}) = O(\las^K)=o(1).
\end{equation}

We now move on to the `regular growth' part of the proof. That is, our next aim is to show that once the neighbourhoods of a vertex $x$ reach size $\omega$,
with very high probability they then grow at a predictable
rate until they reach size comparable with $n$.
We shall use the following convenient form of the Chernoff bounds on the binomial
distribution; see~\cite{Janson_conc}, for example.
\begin{lemma}\label{Chern}
Let $Y$ have a binomial distribution with parameters $n$ and $p$.
If $0\le\delta\le 1$ then
\[
 \Pr\bb{ |Y-np| \ge \delta np } \le 2e^{-\delta^2 np/3}.
\]
\noproof
\end{lemma}

Let $0<\delta<1/1000$ be an arbitrary (small) constant.
Let us say that a vertex $x$ has {\em regular large neighbourhoods}
if one of the following holds: either $|\Ga_t(x)|< \omega$ for all $t$,
or, setting $t^-=\min\{t: |\Ga_t(x)|\ge \omega\}$ and
$t^+=t^- + \log(n^{3/4}/\omega)/\log\la$, we have
\[
 (1-\delta) \la^{t-t^-+1}|\Ga_{t^--1}(x)|  \le  |\Ga_t(x)|  \le 
 (1+\delta) \la^{t-t^-+1}|\Ga_{t^--1}(x)|
\]
for $t^-\le t\le t^+$.
In other words, the neighbourhoods grow by almost exactly a factor of $\la$
at each step from just before the first time they reach size $\omega$
until they reach size around $n^{3/4}$.
Note that since we start from the last `small' neighbourhood $\Ga_{t^--1}(x)$, 
the growth condition above certainly implies that
\begin{equation}\label{cl}
 \frac{1-\delta}{1+\delta} \  \le \ \frac{ |\Ga_t(x)|}{\omega \la^{t-t^-}}
 \ \le\  \la(1+\delta) 
\end{equation}
holds for $t^-\le t\le t^+$.

Let $B_2(x)$ be the `bad' event that a given vertex $x$ of $G(n,\la/n)$
fails to have regular large neighbourhoods,
and $B_2=\bigcup_x B_2(x)$ the global bad event that not all vertices
have regular large neighbourhoods.
\begin{lemma}\label{reg}
For each fixed vertex $x$ of $G(n,\la/n)$ we have
$\Pr(B_2(x))=o(n^{-1})$. Thus $\Pr(B_2)=o(1)$.
\end{lemma}
\begin{proof}
This is well known (c.f.\ Janson, {\L}uczak and Ruci\'nski~\cite[Section 5.2]{JLR}),
and essentially trivial from the Chernoff bounds (or Hoeffding's inequality);
we nevertheless give the details.
We explore the successive neighbourhoods of $x$ in $G(n,\la/n)$
in the usual way, writing $a_t$ for $|\Ga_t(x)|$.
Conditional on $a_0,a_1,\ldots,a_t$, the distribution
of $a_{t+1}$ is binomial with parameters $n-m$ and $p=1-(1-\la/n)^{a_t}$,
where $m=\sum_{t'\le t}a_{t'}$ and $p$ is 
the probability that one of the undiscovered vertices is adjacent to at least
one member of $\Ga_t(x)$.
Assuming that $m=O(n^{3/4})$, say,
we have $\E(a_{t+1}\mid a_0,\ldots,a_t) =\la a_t(1+O(n^{-1/4}))$.
It then follows from Lemma~\ref{Chern} that, conditional on $a_0,\ldots,a_t$,
if $a_t\ge \omega/(100\la)$ then we have
\[
 \Pr\left(\left| \frac{a_{t+1}}{a_t}-\la\right| \ge \frac{1}{(\log n)^2}\right)
 = e^{-\Omega( (\log n)^{-4}a_t)} = o(n^{-100}),
\]
using $\omega= (\log n)^6$. 
Similarly, if $a_t<\omega/(100\la)$ then $\Pr(a_{t+1} \ge \omega )\le n^{-100}$.

Let $t^-$ be the first $t$ with $a_t\ge \omega$, if such a $t$ exists.
We have already shown above that the probability that $0<a_t<\omega$
holds for all $t$ up to $t_0+t_1+K=O(\log n)$ is $o(n^{-1})$,
so with probability $1-o(n^{-1})$ either $t^-$ is undefined, in which
case there is nothing to prove, or $t^-=O(\log n)$, in which
case we have so far uncovered $O(t^-\omega)=o(n^{3/4})$ vertices.
{From} the estimates above, with very high probability $a_{t^--1}\ge \omega/(100\la)$,
and, from this point on, the ratios $a_{t+1}/a_t$ are within a factor $1+O((\log n)^{-2})$
of $\la$ until $a_t$ first exceeds $n^{3/4}$. It follows
that $x$ has regular large neighbourhoods with probability $1-o(n^{-1})$,
as claimed.
\end{proof}

Note that we took $\omega$ as large as $(\log n)^6$ just to simplify the estimates.
If we are a little more careful, a large constant times $\log n$ will in fact do:
significant deviations in the ratio $a_{t+1}/a_t$ are only likely
near the beginning, so we can bound these ratios above and below by sequences approaching
$1$ geometrically with high enough probability.

We now move onto the third phase of the proof, where we consider the   meeting up of neighbourhoods of different vertices and hence the distance between them. This still involves a careful look at the early development of neighbourhoods, since, from the second phase of the proof, we know that vertices with large close neighbourhoods will have large distant neighbourhoods. We treat the upper and lower bounds in the Theorem~\ref{th_const} separately.

\subsection{Upper bound}\label{upper fixed}
As above, set $\omega=(\log n)^6$, say,
and let $K=K(n)$ tend to infinity slowly.

For $x\in V(G)$ let $t_\omega(x)=\min\{t: |\Ga_t(x)| \ge \omega\}$,  
if this minimum exists; otherwise
$t_\omega(x)$ is undefined.
Note that if the event $B_1$ defined above
does not hold, then whenever $t_\omega(x)$ is defined,
we have $t_\omega(x)\le t_0+t_1+K$.

Set 
\[
 t_2=\lfloor \log(n/\omega^2)/\log \la\rfloor,
\]
and, for $x,y\in V(G)$,
let $E_{x,y,i,j}$ be the event that $t_\omega(x)= t_0+t_1-i$, $t_\omega(y)=t_0+t_1-j$,
and $d(x,y)\ge t_\omega(x)+i+t_\omega(y)+j+t_2+3K+c_0$ all hold, where $c_0>2$ is some constant. Our next aim is to bound the probability
of the event $E_{x,y,i,j}$ for given vertices $x$ and $y$ and given $i,j\ge -K$.

Recall that $B_2$ is the event that not all vertices
have regular large neighbourhoods. We claim that there is some $c>0$ (depending only on $\la$) such that
\begin{eqnarray}
 \Pr(E_{x,y,i,j}\setminus B_2)
 &\le& \las^{t_0-i}\las^{t_0-j} e^{-c \la^{3K+i+j}} +o(n^{-100}) \nonumber \\
 &=& O\bb{n^{-2}\las^{-i-j}e^{- c\la^{3K+i+j}}} +o(n^{-100}).\label{xyij}
\end{eqnarray}
First, arguing as in the proof of \eqref{B1x},
using Lemma~\ref{l1} and a version of Lemma~\ref{cpl} where we
start with two vertices and compare with two
copies of the branching process, we see that
\[
 \Pr\bb{ t_\omega(x)=t_0+t_1-i,\, t_\omega(y)=t_0+t_1-j,\,
  d(x,y)>t_\omega(x)+t_\omega(y) } =O(\las^{t_0-i}\las^{t_0-j}).
\]
Exploring the neighbourhoods of $x$ and $y$ in the obvious way, 
suppose we find that $t_\omega(x)=t_0+t_1-i$, $t_\omega(y)=t_0+t_1-j$,
and $d(x,y)>t_\omega(x)+t_\omega(y)$, i.e., our explorations have not yet met.
Set $\ell=\floor{(t_2+3K+i+j)/2}$. Suppose for the moment that
$\ell\le \ell_0=\log(n^{3/4}/\omega)/\log\la$.
Continuing the exploration of the two neighbourhoods a further $\ell$ steps in each case,
we may assume that with respect to the neighbourhoods of $x$ and $y$
that have been revealed so far, the regular large neighbourhood
condition has not yet been violated. (If it has been, the event $B_2$
must hold, and we are bounding the probability of an event contained
in the complement of $B_2$.)
Then
\[
 \min\{|\Ga_{t_\omega(x)+\ell}(x)|,\, |\Ga_{t_\omega(y)+\ell}(y)|\} \ge 0.99 \omega \la^\ell= \Omega\bb{\sqrt{n\la^{3K+i+j}}}.
\]
It may be that $d(x,y)\le t_\omega(x)+\ell+t_\omega(y)+\ell$, in which case we are done.
Otherwise, the edges between $\Ga_{t_\omega(x)+\ell}(x)$ and $\Ga_{t_\omega(y)+\ell}(y)$
have not yet been tested, so the chance that no such edge is present
is
\[
 (1-\la/n)^{ |\Ga_{t_\omega(x)+\ell}(x)| |\Ga_{t_\omega(y)+\ell}(y)| }
  \le e^{-(\la/n) \Omega(n\la^{3K+i+j})} \le e^{-c\la^{3K+i+j}},
\]
for some constant $c>0$. Multiplying by the $O(\las^{t_0-i}\las^{t_0-j})$ bound obtained above
gives \eqref{xyij} in this case.

If $\ell>\ell_0$, the argument is similar; this time, assuming $B_2$ does not hold only allows
us to control the sizes of the neighbourhoods for $\ell_0<\ell$ steps beyond $t_\omega(x)$
and $t_\omega(y)$. But by this time they reach size at least $n^{3/4}/2$, and the probability
that they do not join is at most $e^{-(\la/n)n^{3/2}/4}=o(n^{-100})$.

Let $B$ be the event that  
\begin{eqnarray*}
 \diam(G)\ge 2t_0+2t_1+t_2+3K +10 &\ge& 2\frac{\log\omega}{\log\la}+2\frac{\log n}{\log(1/\las)}
 + \frac{\log n-2\log\omega}{\log\la}+3K \\
 &=& 
 \frac{\log n}{\log\la} + 2\frac{\log n}{\log(1/\las)} +3K.
\end{eqnarray*}
Our aim is to prove that with $K\to\infty$ arbitrarily slowly, we have
$\Pr(B)=o(1)$; in order to do so, it suffices to show that $\Pr(B\setminus (B_1\cup B_2))=o(1)$.

Suppose that $B$ holds but $B_1\cup B_2$ does not, and let $x$ and $y$ be vertices at maximum distance.
Since $B_1$ does not hold, and $d(x,y)$ is so large, exploring successive neighbourhoods of $x$ and $y$,
these neighbourhoods both reach size at least $\omega$ before they meet. Hence
$E_{x,y,i,j}$ holds for some $i$ and $j$. Since
$B_1$ does not hold, $E_{x,y,i,j}$ can only hold if $i,j\ge -K$.
Hence, using \eqref{xyij},
\begin{multline*}
 \Pr(B\setminus (B_1\cup B_2)) 
\le \sum_{i,j\ge -K} \sum_{x,y\in V(G)} \Pr(E_{x,y,i,j}\setminus B_2) \\
 \le o(n^{-90})+ n^2\sum_{i,j\ge -K} n^{-2}O\left(\las^{-i-j}e^{-c\la^{3K+i+j}}\right),  \end{multline*}
which is $o(1)$ since
\[
\sum_{r\ge -2K} (r+2K+1)\las^{-r} e^{-c\la^{3K+r}} 
 =O(e^{-c\la^K})=o(1).
\]
This completes the proof of the upper bound.

\begin{remark}\label{rno1}
Note that one cannot prove the upper bound directly by the first moment method; a separate
argument excluding very long thin neighbourhoods (bounding the probability of $B_1$)
is needed. Indeed, it is not too hard to show that the estimates above are essentially tight.
Thus, if for some $r\ge K$ 
there happens to be a vertex $x$ with $t_\omega(x)=t_0+t_1+r$, say, an 
event of probability around $\las^r$, then
$x$ will be at distance roughly $d=2t_0+2t_1+t_2+K$ from many of the
roughly $(1/\las)^{r-K}$ vertices $y$ with $t_\omega(y)=t_0+t_1-r+K$. Since there are $\Theta(\log n)$
possible values of $r$, the expected number of pairs of vertices at distance $d$
will tend to infinity if $K\to\infty$ slowly enough.
\end{remark}

\subsection{Lower bound}\label{lower_fixed}

The idea of the lower bound is simple. Let $S$ be the set
of vertices $x$ with $t_\omega(x)\ge t_0+t_1-K$. Then,
from the arguments in the previous section, the expected
size of $S$ is roughly $(1/\las)^K$, which tends to infinity.
We would like to show that $|S|$ is large with high probability
using the second moment method. Since two vertices in $S$ are likely
to be far apart, the result will follow. There are two problems.
A minor one is that
the events that different vertices lie in $S$ are not
that close to independent: vertices in $S$ will usually be located in trees attached to the 2-core, and
$S$ roughly corresponds to the set of vertices at least a certain
distance from the 2-core. Although most trees
attached to the 2-core will contain no such vertices,
it turns out that, on average, each tree contributing
one or more such vertices contributes some constant number larger than $1$,
so $|S|$ is not well approximated
by a Poisson distribution. A more serious, related, problem is that to find vertices
at large distance we need to find vertices in $S$
whose short-range neighbourhoods do not overlap, i.e., 
vertices coming from different trees. We solve both these problems
by looking for vertices $x\in S$ satisfying an additional
condition, the {\em strong wedge condition}, that usually corresponds  to $x$
being the unique vertex in its tree at maximal distance from the 2-core.

Note that as we are now looking for a lower bound on the diameter, we do not need
to consider all promising pairs of vertices for our candidate vertices at large distance. We may thus impose additional conditions
as convenient, and our result will still be sharp enough as long as these
conditions are likely enough to be satisfied.
One such condition is that the neighbourhoods are trees up to a suitable
distance. 

Let $x\in V(G)$, and suppose that $\Gat(x)$ is a tree for some $t>0$.
The {\em weak/strong wedge condition}
holds from $x$ to $x_2\in \Gamma_t(x)$ if for every $z\ne x$
in the graph $\Gat(x)$,
the distance from $z$ to the closest vertex $y$
on the unique path from $x$ to $x_2$ is at most/strictly less than the
distance from $x$ to $y$. Note that either condition implies that 
the degree of $x$ in $G$ must be 1. In this section we shall always
work with the strong wedge condition;
the weak wedge condition will play a role in Section~\ref{sec_eps}.

Let $t_K$ denote $t_0-K$, where $K=K(n)\to\infty$ arbitrarily slowly, in particular with $K\le \log \log n$,
and let $W_x^0$ be the event that $\Gl{t_K}(x)$ is a tree with the following
properties:
there is a unique vertex $x_2$ at distance $t_K$ from $x$, and
the strong wedge condition holds from $x$ to $x_2$.
Let $W_x$ be the event that $W_x^0$ holds and $\Gl{t_K}(x)$
contains fewer than $\omega/2$ vertices,
where $\omega =(\log n)^6$ as before.

Note for later that if $W_x$ (or $W_x^0$) holds, then the tree  $\Gl{t_K}(x)$  consists of an $x$-$x_2$ path $P_x$ of length $t_K$ with a (possibly empty) set of trees attached
to each interior vertex, the height of each tree being strictly
less than the distance to the nearest endvertex of $P_x$.
(Thus $W_x^0$ is a sort of `diamond' condition. We will use a precise version of this terminology in the next section.) It follows
that the diameter of $\Gl{t_K}(x)$ is $t_K$, and that $x$ and $x_2$
are the unique pair of vertices of $\Gl{t_K}(x)$ at this distance.

Let $W^0$ and $W$ be the branching process events corresponding to $W_x^0$ and $W_x$, so $W^0$ is the branching process version of our diamond condition.
The event that $W$ holds is a disjoint union of events that
$X_{\le t_K}$ is one of certain trees with at most $\omega/2=o(n^{1/2})$
vertices, so by Lemma~\ref{cpl} we have
$\Pr(W_x)\sim\Pr(W)$.

Once the branching process reaches size $(\log n)^4$, it is very unlikely
ever to shrink down to size $1$, and in fact the probability that $W^0$ holds but one of the first
$t_K=t_0-K$ generations
has size at least $(\log n)^4$ is $o(n^{-100})$. (This follows from the proof of Lemma~\ref{reg},
but is much simpler.) 
Assuming this does not happen, the sum of sizes of the first $t_0-K$ generations is at most $t_0(\log n)^4=  O(\log^5 n)$. It follows that $\Pr(W^0\setminus W)=o(n^{-100})$,
so
\begin{equation}\label{WW0}
 \Pr(W)=\Pr(W^0)+o(n^{-100}).
\end{equation}

To calculate $\Pr(W^0)$, consider the event $W'$, that $W^0$ holds and the unique
particle in generation $t_0-K$ survives. 
Note that
\begin{equation}\label{WW'}
 \Pr(W') = s\Pr(W^0).
\end{equation}
If $W'$ holds, then $|X_t^+|=1$ for $t=t_0-K$ and hence for $t=0,1,\ldots,t_0-K$,
an event of probability $s\las^{t_0-K}$.
Conversely, constructing $\bp_\la$ as before by starting from $\bp^+$ and adding
in independent copies of $\bp_{\las}=(X_r^-)$ started at each particle,
$W'$ holds if and only if $|X_{t_0-K}^+|=1$ and, for $0\le t< t_0-K$,
the copy of $\bp_{\las}$ started at the unique particle of $X_t^+$
dies within $\min\{\max\{t,1\},t_0-K-t\}$ generations:
dying within $t_0-K-t$ generations 
ensures that $|X_{t_0-K}|=|X_{t_0-K}^+|=1$, and, for $t>0$, dying within
$t$ generations ensures that the strong wedge condition holds.
Let $d_t=\Pr(|X_t^-|=0)$ be the probability that the subcritical
process $\bp_{\las}$ dies within $t$ generations.
Then we have
\begin{multline}\label{PW'}
 \Pr(W')  = s\las^{t_0-K} d_1\prod_{t=1}^{t_0-K-1} d_{\min\{t,t_0-K-t\}} \\
 = s\las^{t_0-K}d_1 d_1^2d_2^2d_3^2\cdots d_{\floor{(t_0-K)/2}-1}^2 d_{\floor{(t_0-K)/2}}^\theta,
\end{multline}
where the exponent $\theta$ of the last factor is $1$ or $2$  depending on the parity of $t_0-K$.
As we shall see, the later factors in the product are essentially irrelevant.
Indeed,
\begin{equation}\label{dt}
 1-d_t = \Pr(|X_t^-|>0) \le \E(|X_t^-|) = \las^t,
\end{equation}
so $1-\las^t\le d_t\le 1$, and $-\log d_t=O(\las^t)$.
Since $\sum_t \las^t$ is convergent, we thus have $\prod_t d_t=\Theta(1)$,
so $\Pr(W')=\Theta(\las^{t_0-K})$ and, using \eqref{WW0} and \eqref{WW'},
\[
 \Pr(W)=\Pr(W^0)+o(n^{-100}) = s^{-1}\Pr(W') +o(n^{-100})= \Theta(\las^{t_0-K}).
\]
Since $\las^{t_0}$ is of order $1/n$ and $K\to\infty$, it follows that $n\Pr(W)\to\infty$,
and hence that $n\Pr(W_x)\to\infty$.

Recalling that $t_1=\floor{\log \omega/\log \la}$,
set $t=t_K=t_0-K$,
let $W_x^+$ be the event that
$W_x$ holds, $|\Gamma_{t+t_1}(x)|\ge \omega$, and
$|V(\Gl{t+t_1}(x))|<\omega^2$.
If $W_x$ holds, then 
exploring the neighbourhoods of $x$ to distance $t$ we have
by definition reached at most $\omega/2$ vertices.
Using~\eqn{ar}, it is easy to show that $\Pr(W_x^+\mid W_x)=\Theta(1)$, so
$\Pr(W_x^+)=\Theta(\Pr(W_x))$.

Let $N$ be the number of vertices $x$ for which $W_x^+$ holds,
so $\E(N)=n\Pr(W_x^+)=\Theta(n\Pr(W_x))\to \infty$. We shall use the second
moment method to show that $N$ is concentrated about its mean.
The argument is slightly more complicated than one might expect
(or hope for); while one can give simpler arguments that are very
plausible, we have so far failed to turn such an argument into
a rigorous proof. In fact, the argument we do present deals
with all issues of possible dependence with very little calculation.

Suppose that $x$ and $y$ are distinct vertices and that $W_x^+$
and $W_y^+$ both hold. Then $W_x$ and $W_y$ also hold. Our
immediate aim is to show that the subgraphs
$\Gat(x)$ and $\Gat(y)$ must be edge disjoint,
i.e., they can meet only if $x_2=y_2$, and then only at this one vertex.
We shall write $W_x \star W_y$ for the event
that $W_x$ and $W_y$ hold, and $\Gat(x)$ and $\Gat(y)$ 
are edge disjoint. In other words
$W_x\star W_y = W_x\cap W_y\cap \{d(x,y)\ge 2t\}$.
We define $W_x^+\star W_y^+$ similarly,
so
$W_x^+\star W_y^+ = W_x^+\cap W_y^+\cap \{d(x,y)\ge 2t+2t_1\}$.
(One must be careful here: when $W_x$ holds, $\Gat(x)$ is {\em not} a certificate
for this event in the sense of the van den Berg--Kesten box product~\cite{vdBK},
i.e., specifying that this particular subgraph is present as an induced subgraph
does not guarantee that $W_x$ holds. To guarantee $W_x$, one must
also certify that various edges are absent, from $\Gl{t-1}$ to vertices
outside $\Gat$; such certificates for $W_x$ and $W_y$ can never be disjoint,
so we cannot simply apply Reimer's Theorem~\cite{Reimer} to bound $\Pr(W_x\star W_y)$.)

Still assuming that $x$ and $y$ are distinct vertices such that $W_x^+$ and $W_y^+$ both
hold,
suppose first that $y$ lies strictly inside $\Gat(x)$, i.e.,
that $y\in V(\Gat(x))\setminus \{x_2\}$.
As noted earlier, since $W_x$ holds, $\Gat(x)$ has diameter $t$, and this diameter is realized
uniquely by $x$ and $x_2$. Thus the vertex $y_2$, which is at distance
$t$ from $y$, must lie outside $\Gat(x)$. But then the unique
$y$-$y_2$ path $P_y$ passes through $x_2$. Considering the vertex
$z$ where $P_y$ first meets $P_x$, the strong wedge condition for $x$
gives $d(y,z)<d(x,z)$. But the strong wedge condition for $y$
gives $d(x,z)<d(y,z)$, a contradiction.

We may thus assume that $y$ lies outside $V(\Gat(x))\setminus \{x_2\}$.
Suppose now that $y_2$ also lies outside this set, and that $y_2\ne x_2$.
Since $x_2$ is a cutvertex, it follows that all of $P_y$ is outside
$V(\Gat(x))\setminus \{x_2\}$.
If $\Gat(x)$ and $\Gat(y)$ meet, then, since $x_2$ is a cutvertex,
$x_2$ must be a vertex of $\Gat(y)$. Furthermore, since $\Gat(y)$
consists of $y_2$ plus a component of $G\setminus \{y_2\}$,
all of $\Gat(x)$ lies in
$\Gat(y)\setminus\{y_2\}$. In particular $x\in \Gat(y)\setminus\{y_2\}$
and we obtain a contradiction as above.

If $x_2=y_2$, then each of $\Gat(x)$ and $\Gat(y)$ is formed by $x_2$ together with a tree
component of $G-x_2$. Since each of $x$ and $y$ is the unique vertex at maximal distance
from $x_2=y_2$ within its tree, and $x\ne y$, these components are different, and so disjoint,
so $W_x\star W_y$ holds.

We may thus assume that, if $W_x^+\cap W_y^+$ holds but $W_x\star W_y$ fails,
then $y$ lies outside $V(\Gat(x))\setminus\{x_2\}$
but $y_2$ is inside. It is easy to check that in this
case $\Gat(x)\cup \Gat(y)$ forms a component of $G$ (and actually $y_2$ must lie on the path from $x$ to $x_2$).
Since $W_x^+$ holds,
this component (the component of $G$ containing $x$) has size at least $|\Ga_{t+t_1}(x)|\ge\omega$;
however, $W_x\cap W_y$ also holds, so 
it has size less than $2\omega/2$, a contradiction.

We have just shown that if $W_x^+\cap W_y^+$
holds, then so does $W_x\star W_y$.  
It follows that either $W_x^+\star W_y^+$ holds, or
$d(x_2,y_2)\le 2t_1$, implying $d(x,y)\le 2(t_0-K+t_1)$.
Thus, for the second moment,
\begin{eqnarray*}
 \E N^2-\E N 
&=& \sum_x \sum_{y\ne x} \Pr(W_x^+\cap W_y^+ ) \\
&=& \sum_x \sum_{y\ne x} \Pr(W_x^+\cap W_y^+\cap (W_x\star W_y)) \\
&\le& \sum_x \sum_{y\ne x} \Pr(W_x^+ \star W_y^+) +\Pr\bb{W_x\star W_y\cap \{d(x,y)\le 2(t_0-K+t_1)\}}.
\end{eqnarray*}
For the first term, we have $\Pr(W_x^+\star W_y^+)\sim \Pr(W_x^+)\Pr(W_y^+)$,
since testing the event $W_x^+$ uses up at most $\omega^2$ vertices, which does not affect the probability
of  $W_y^+$ significantly. (Alternatively, 
as before we may use Lemma~\ref{cpl} and a version
of this lemma where we start at two vertices.)

To handle the second term,
we use the following inequality, which we shall prove in a moment:
\begin{equation}\label{fewclose}
 \Pr\bb{W_x\star W_y\cap \{d(x,y)\le 2(t_0-K)+t_3-K'\}} =o(n^{-2}),
\end{equation}
where $K'=3K\log(1/\las)/\log \la$ and $t_3=\log n/\log\la$.
Assuming this, using the fact that
$2t_1\le t_3-K'$ for large $n$ if $K$ tends to infinity sufficiently slowly,
we have $\E N^2 \le \E N+(1+o(1))(\E N)^2 +n^2o(n^{-2})$.
Since $\E N\to\infty$, it follows that $\E N^2\sim (\E N)^2$,
so by Chebyshev's inequality $N$ is concentrated about its mean,
and in particular, $N\ge 2$ whp.

Set $d=2(t_0-K)+t_3-K'$, so
$d=\log n/\log \la+2\log n/\log(1/\las) -O(K)$.
With $K$ tending to infinity arbitrarily slowly, our aim
in this subsection is to prove that $\diam(G)\ge d$ holds whp.

Let $M$ be the number of pairs
of distinct vertices $x$, $y$ for which $W_x\star W_y\cap \{d(x,y)\le d\}$
holds. Using \eqref{fewclose} again, we have $\E M=o(1)$, so $M=0$ \whp.
Thus, whp, we have $N\ge 2$ and $M=0$. Then there are distinct vertices $x$, $y$ for which $W_x^+$
and $W_y^+$ hold. As shown above, it then follows
that $W_x\star W_y$ holds. Since $M=0$, we have $d(x,y)>d$.
From the classical results of Erd\H os and R\'enyi~\cite{ER60},
there is some constant $A>0$ such that whp exactly one component
of $G$, the `giant' component, contains more than $A\log n$ vertices.
Since (for $n$ large) $W_x^+$ implies that $x$ is in a component with at
least $\omega>A\log n$ vertices, whp any pair $x$, $y$ satisfying
the conditions above lies in the giant component, so $d<d(x,y)<\infty$,
and $\diam(G)>d$, as required.

It remains only to prove \eqref{fewclose}.
To do so, we explore the neighbourhoods of a given pair $x$, $y$ 
of vertices as usual, to test whether $W_x\star W_y$ holds.
If so, the possible edges between the
remaining vertices, including $x_2$ and $y_2$, have not yet been tested,
so each is present with
its original unconditional probability. Hence, given $W_x\star W_y$,
summing over all possible paths we see that
the probability that $d(x_2,y_2)\le \ell$
is at most
\[
 \sum_{k\le \ell} n^{k-1} (\la/n)^k = \sum_{k\le \ell}\la^k/n = O(\la^\ell/n)
\]
and
\begin{eqnarray*}
 \Pr\bb{W_x\star W_y \cap \{d(x,y)\le 2(t_0-K)+t_3-K')\}}
 &=& \Pr(W_x)\Pr(W_y) O(\la^{t_3-K'}/n)\\
 &=& O\bb{\las^{t_0-K}\las^{t_0-K} \la^{-K'}} \\
 &=& O(1/n^2)(1/\las)^{2K}\la^{-K'} \\
 &=& O(1/n^2)(1/\las)^{2K-3K} =o(n^{-2}),
\end{eqnarray*}
as required.

Combining the lower bound on the diameter we have just proved, and the upper
bound proved in \refSS{upper fixed}, we obtain
Theorem~\ref{th_const}.

\section{Average degree tending to infinity}\label{sec_inf}

In this section we shall prove Theorem~\ref{th_inf}. Throughout,
when we consider $G(n,\la/n)$ we assume that $\la=\la(n)\to\infty$
with $\la\le n^{1/1000}$.
For convenience, we always assume that $\la$ is larger
than some absolute constant $\la_0$, chosen so that
the various statements `provided $\la$ is large enough'
in what follows hold for $\la\ge \la_0$.
With $\la$ tending to infinity,
some aspects of the proof become easier than the $\la$ constant case,
whilst some become more difficult.

We retain the same basic plan of attack as for the case of $\la$ constant. One of the main problems is
that we cannot simply work with the time that the neighbourhoods of a
vertex take to reach a certain size $\omega$, since the first
neighbourhood larger than this may have size
anywhere from $\omega$ to around $\la\omega$; this
difference is too big for our later arguments.  Instead we will look
at the size of the neighbourhoods at a specific time. We could
consider sizes in certain ranges, but it turns out that we can simply
consider individual sizes, bounding the probability that a certain
neighbourhood has exactly a certain size $r$.  Roughly speaking,
as in the previous section, the
probability that the neighbourhoods of a vertex take $a$ generations
longer than usual to reach (or exceed) some given size turns out to be around $\las^a$,
where $\las<1$ is the dual branching process parameter, defined
by $\las e^{-\las}=\la e^{-\la}$.
This event corresponds to the (later)
neighbourhoods being a factor of $\la^a$ smaller than usual. So we study
for {\em real} parameters $a$ the probability that the neighbourhoods
are $\la^a$ smaller than usual, expressing this probability as a power
of $\las$.

Throughout this section it will be useful to bear in mind the asymptotic
formula
\begin{equation}\label{lasinf}
 \las =  \la e^{-\la} +O(\la^2 e^{-2\la}),
\end{equation}
which follows easily from $\las e^{-\las}=\la e^{-\la}$ and $\las<1$.
Note in particular that $\las$ is asymptotically smaller than any constant
negative power of $\la$.

\subsection{Branching process preliminaries}
We first give some lemmas describing the growth behaviour
of the branching process. 
\begin{lemma}\label{smoothlate}
Suppose $\la\ge 10$ and $0<\delta\le 1/2$. 
Given that $|X_r|=k\ge 1$,
with probability at least $1-e^{-c\delta^2 \la k}$ we have
$|X_t|/(\la^{t-r}k)\in [1-\delta,1+\delta]$ for all $t\ge r$,
where $c>0$ is an absolute constant.
\end{lemma}
\begin{proof}
We may assume without loss of generality that $r=0$. For $t\ge 0$, 
let $\rho_t=|X_{t+1}|/(\la |X_t|)$, and let $E_t$ be the event
that $|\rho_t-1| > \delta/3^{t+1}$; it suffices to prove
that $\Pr(\bigcup_t E_t)\le e^{-c\delta^2\la k}$.
Let $F_t$ be the event that $E_t$ holds but no $E_s$ holds, $s<t$,
so $\Pr(\bigcup_t E_t) = \sum \Pr(F_t)$. If no $E_s$ holds for $s<t$,
then $|X_t|\ge k\la^t\prod_{s<t} (1-\delta/3^{s+1})\ge k\la^t/10$.
Turning to $|X_{t+1}|$, conditional on $|X_t|$, by Lemma~\ref{Chern}
the probability that $\rho_t$ lies outside $[1-\delta/3^{t+1},1+\delta/3^{t+1}]$
is at most $2\exp(-c_0 \delta^2 9^{-t} \la |X_t|)$, for some $c_0>0$.
Hence 
$\Pr(F_t) \le 2\exp(-c_0 \delta^2 9^{-t} \la^{t+1} k/10)$,
and the result follows by summing this rapidly decreasing sequence.
\end{proof}


For $0\le a < 1$ define $g(a)=g(\la,a)$ by  $\las^{g(a)}=\Pr(Z\le \la^{1-a})$,
where $Z$ has a Poisson distribution with mean $\la$. Thus $\las^{g(a)}$ is the probability
that $Z$ is smaller than its mean by a factor of $\la^a$ or more.
Note that $g(a)$ is 
(weakly) increasing in $a$. Also, as $\la\to\infty$ we have
$\Pr(Z \le \la)\to 1/2$ and $\las\to 0$, so $g(0)=o(1)$.
%
A simple calculation shows that
$g(a)=1-o(1)$ for any fixed $0<a<1$.
Also, using \eqref{lasinf} we have 
$\Pr(Z\le 1)=(1+\la)e^{-\la}=\las e^{-\las}(1+1/\la)>\las$
for large enough $\la$,  
and so
\begin{equation}\label{g01}
 0\le g(a)<1
\end{equation}
for all $0\le a<1$.

Extend $g$ to the real line
by defining $g(x)=\floor{x}+g(x-\floor{x})$;
this gives an increasing function which, from \eqref{g01}, satisfies
\begin{equation}\label{g approx}
\floor{x}\le g(x) \le \floor{x}+1
\end{equation} 
for all $x$.
It is straightforward to check that for
any constant $b\ge 3$, say, if $n$ is large enough then
\begin{equation}\label{addb}
\las^{g(a-\log b/\log \la)}\ge \la^{b/4}\las^{g(a)}
\end{equation}
holds for all $a$.
Indeed, if $m\le a-\log b/\log\la,\,a<m+1$ for some integer $m$, then \eqref{addb}
decodes to a statement of the form $\Pr(Z\le bk)\ge \la^{b/4}\Pr(Z\le k)$,
where $1\le k\le \la/b$; the inequality is easily verified by considering,
for example, the ranges $k\ge \la/(10b)$, $\sqrt{\la}\le k\le \la/(10b)$,
and $1\le k\le \sqrt{\la}$.
On the other hand, if $a-\log b/\log\la<m\le a$ then it decodes
to $\las^{-1}\Pr(Z\le kb/\la) \ge \la^{b/4}\Pr(Z\le k)$, with $k<\la$ and $bk>\la$;
this is easily verified by considering the cases $k\ge 0.9\la$ and $k<0.9\la$, say.

We next give an analogue of the upper bound in Lemma~\ref{l1}; note that we do not
round $t_1$ to an integer.
\begin{lemma}\label{ia}
Suppose that $\omega\ge \la$ and that $t\ge 0$ is an integer.
Then for $\la$ at least some absolute constant, setting $t_1=\log\omega/\log\la$ we have
\[
 \Pr\bb{ 0<|X_t| < \omega/2 } \le 3 \las^{g(t-t_1)}.
\]
\end{lemma}
\begin{proof}
Note first that if $t<t_1$,
then $g(t-t_1)\le 0$ by \eqref{g approx}, so the result
holds trivially. We may thus assume that $t\ge t_1$, so $t\ge\ceil{t_1}$.
\medskip

\noindent
{\em Case 1:\,} $t\ge\ceil{t_1}+1$.
\smallskip

Similar to the proof of Lemma~\ref{l1}, set 
$r_t=\Pr(|X_t|<\omega/2 \mid |X_t|>0)$.
Then  it suffices to show
that $r_t\le 3\las^{g(t-t_1)}$. 
We shall show in a moment that if $t=\ceil{t_1}+1$, then
\begin{equation}\label{12}
  \Pr\bb{ 0<|X_t| < \omega/2 } \le 1.1 \las^{g(t-t_1)}.
\end{equation}
Suppose for the moment that this holds.
Then by monotonicity of $g$ and the fact that $g(1)\ge 1$,
and since $\pr(|X_t|>0)\sim 1$, for such $t$ we have $r_t\le 1.2\las$ if $\la$
is at least some (absolute) constant.

As noted in the proof of Lemma~\ref{l1}, the implicit constant in all $O(\cdot)$
notation leading to \eqref{rb1}
may be taken to be absolute when $\la>1$ is bounded away from $1$,
so this bound applies with $\la$ growing as a function of $n$. In particular,
from \eqref{rb1}, we have for arbitrary $t\ge 1$ 
\begin{equation}\label{rtlbig}
 r_t \le r_{t-1}(\las + r_{t-1} +O(\la\las^t))
 = r_{t-1}(\las + r_{t-1} +o(\las^{t-1})),
\end{equation}
using $\la\las=o(1)$ for the last step, which follows from \eqref{lasinf}.

We may iterate \eqref{rtlbig}, with $\las$ sufficiently small in the following (as $\la$ can be assumed large). Beginning with   $t=\ceil{t_1}+1$, when $r_t\le 1.2\las$ from \eqref{12} and
 hence $r_{t+1}\le 2.7\las^2$ from \eqref{rtlbig}, we see that that $r_t$ decreases extremely rapidly: $r_t\le 3\las r_{t-1}$ for $t>\ceil{t_1}+1$.
Feeding the resulting bound $r_{\ceil{t_1}+k}\le (3\las)^k$, $k>1$, back into \eqref{rtlbig},
it follows that for $t>\ceil{t_1}+1$ we have
$r_t\le r_{t-1}\las(1+\eps_t)$ where the first error term $\eps_{\ceil{t_1}+2}$ is at most $1.3$
and later ones decrease extremely rapidly.
Since $\prod_t(1+\eps_t)\le 2.4$ for $\la$ large enough,
the result for Case~1 now follows from \eqref{12}.

It remains to prove \eqref{12}.
Assuming now that $t=\ceil{t_1}+1$, put $a=t-t_1-1$, so that $0\le a<1$. We claim that 
\begin{equation}\label{clm1}
\pr\bb{0<|X_2|\le \la^{1-a}}\sim \las^{1+g(a)}
\end{equation}
and that
\begin{equation}\label{clm2}
\pr\big(\eventand{|X_2|> \la^{1-a}}{|X_t| < \omega/2}\big) =o(\las^{1+g(a)}).
\end{equation}
Since $1+g(a)=g(t-t_1)$, these imply~\eqn{12}.

Note that $\pr(|X_1|=1) =\la e^{-\la}\sim \las$, and the probability that subsequently  $|X_2|\le \la^{1-a}$ is $\las^{g(a)}$ by definition of $g$. Thus,   
$$
\pr\big(\eventand{|X_1|=1}{|X_2|\le \la^{1-a}}\big) \sim \las^{1+g(a)}.
$$ 
On the other hand, conditioning on $|X_1|=k\ge 2$, the conditional
distribution of $|X_2|$ is Poisson $\Po(k\la)$. Since $a\ge 0$,
we may assume that $\la^{1-a}\le k\la/2$, and it follows that
there is an absolute constant $c_2>0$ such that for all $k\ge 2$,
$\pr(|X_2|\le \la^{1-a}\mid |X_1|=k) <e^{-c_2k\la}$.
Fixing $k_0>  3/c_2$, we have
\begin{multline*}
 \Pr\bb{|X_1|> k_0,\,|X_2|\le \la^{1-a}} \le
 \pr\big(|X_2|\le \la^{1-a} \bigm| |X_1|>k_0 \big)\\ \le e^{-3\la} = o(\las^2) = o\big(  \las^{1+g(a)}\big).
\end{multline*}

Turning to $2\le k\le k_0$, we have $\pr(|X_1|=k ) = O(\la^{k-1}\las)$.  
Suppose firstly that $g(a)<c_2$ as defined above.
Then $\pr\big((2\le |X_1|\le k_0) \wedge |X_2|\le \la^{1-a}\big)< O(\la^{k_0-1}\las)e^{-2c_2\la}=\las^{1+2c_2+o(1)}=o( \las^{1+g(a)})$. So we may assume that  $g(a)\ge c_2$. Then, noting that for $|X_2|\le \la^{1-a}$
to hold each particle in $X_1$ must have at most $\la^{1-a}$ children, we have
$$
\Pr \bb{|X_1|=k,\,|X_2|\le \la^{1-a}}=O(\la^{k-1}\las\las^{kg(a)})=o( \las^{1+g(a)}),
$$
since $\la\las^{g(a)}=\las^{g(a)-o(1)}=o(1)$.
Putting the pieces together, we have established~\eqn{clm1}.

The proof of~\eqn{clm2} is similar. Condition on $|X_2|=k$, where
$k>\la^{1-a}$. In the event that $|X_t| < \omega/2$, the average
number of descendants in $X_t$ of a particle in $X_2$ is less than
$\omega/(2\la^{1-a})$. However, we know that any
one such particle expects
$\la^{t-2}=\la^{t_1+a-1}=\omega/\la^{1-a}$ such descendants, and
applying Lemma~\ref{smoothlate}, we see that
$\Pr(|X_t|<\omega/2 \mid |X_2|=k)\le e^{-c_3k\la}$
for some $c_3>0$. Arguing as for the proof of~\eqn{clm1}, there exists
$k_1$ such that $\pr( |X_t|<\omega/2 \mid |X_2|>k_1 ) = o\big( \las^{1+g(a)}\big)$.

We are left with showing~\eqn{clm2} in the case that $\la^{1-a}<|X_2|\le k_1$,
which requires $|X_2|\ge 2$ since $a<1$. It is easy to
see that $\pr( |X_2|\le k_1)=\Theta(\la^{k_1-1}\las^2)=\las^{2-o(1)}$. Conditional upon this,
for $|X_t|<\omega/2$ to hold at least one particle in $X_2$ must have at most
half its expected number of descendants in $X_t$.
By Lemma~\ref{smoothlate} and the union bound, the conditional
probability of this is at most
$k_1e^{-c_3\la}=o(\las^{c_3/2})$. Hence
$$
 \pr( |X_2|\le k_1) \pr\bb{|X_t| < \omega/2 \bigm| \la^{1-a}<  |X_2|\le k_1}=o(\las^{2-o(1)+c_3/2})= o(\las^{2})  
$$
and we have~\eqn{clm2} since $1+g(a)\le 2$.
\medskip

\noindent
{\em Case 2:\,} $t=\ceil{t_1}$. 
\smallskip

In this case, setting $a=t-t_1\in [0,1)$,
we have $\Pr(0<|X_1|\le \la^{1-a})<\las^{g(a)}$ by definition of $g$.
Using this in place of \eqref{clm1}, it suffices to show
that $\Pr\big(\eventand{|X_1|> \la^{1-a}}{|X_t|<\omega/2}\big)=o(\las^{g(a)})$;
the proof is identical to that of \eqref{clm2}, apart from the notation.
\end{proof}


We next turn to the analogue of the lower bound in Lemma~\ref{l1}; as
there, we bound the probability of a rather specific event involving
extra conditions that will be needed in our lower bound
on the diameter.

We say  
that the branching process $(X_t)$ satisfies the {\em diamond
condition to generation $r$} if $|X_1|=1$, there is a unique particle $x_r$ in $X_r$,
and the chain $x_0x_1\cdots x_r$ of ancestors of $x_r$ is such that
any `side branches' starting from $x_i$ die within $\min\{i,r-i\}$ further
generations. For $r=0$ we interpret the diamond condition to hold vacuously.
\begin{lemma}\label{ib}
Let $t'\ge 0$ be an integer, and $0\le a<1$ a real number. Let
$F_0$ be the event that $|X_{t'}|=1$ and the diamond condition holds to generation $t'$,
and let $F_1$ be the event that $|X_{t'+1}| \le \la^{1-a}$.
Then as $\la\to\infty$ we have
\[
 \Pr(F_0\cap F_1) \sim \las^{g(t'+a)},
\]
uniformly in $t'$ and $a$. Furthermore, provided $\la$ is at least
some absolute constant, then for any $\omega\ge \la$ 
and $t\ge t_1=\log\omega/\log\la$ there is a $\rho$
with $\omega/3\le\rho\le 2\omega$ such that 
\[
 \Pr\bb{ F_0\cap F_1\cap\{ |X_t|=\rho\} } \ge \las^{g(t-t_1)}/(3\la\omega),
\]
where $F_0$ and $F_1$ are defined as above with $t'$ and $a$ the
integer and fractional parts of $t-t_1$, respectively.
\end{lemma} 
Note that $t_1$ is {\em not} rounded to an integer.
Essentially, the lemma says that the probability that $\bp_\la$ survives
but (after some time) is a factor $\la^x$ smaller than it should be
is around $\las^{g(x)}$. The second statement shows that
there is some specific size in a suitable range such that the probability
of hitting exactly this size is not much smaller.

\begin{proof}
The event $F_0$ is exactly the event $W^0$ referred to
in~\eqn{WW0}, but with $t_0-K$ replaced by $t'$. Using \eqref{WW'}
to translate \eqref{PW'} back in terms of  $W^0$, we have
$
 \las^{t'}\ge \Pr(F_0) \ge \las^{t'}d_1 d_1^2d_2^2d_3^2\cdots,
$
where $d_t=\Pr(|X_t^-|=0)$ is at least $1-\las^t$ from \eqref{dt}.
Since $\las\to 0$, it follows that
$\Pr(F_0)\sim \las^{t'}$.

Conditioning on $F_0$ says nothing about the descendants of the
unique particle $z$ in $X_{t'}$, so if $Z$ is Poisson with mean $\la$ then
\[
 \Pr(F_1\mid F_0) = \Pr(Z\le \la^{1-a}) = \las^{g(a)},
\]
where the last step is the definition of $g(a)$.
Since $\las^{t'}\las^{g(a)} = \las^{g(t'+a)}$, this proves the first statement.

Turning to the second statement, suppose that $\omega\ge\la$
and $t\ge t_1=\log\omega/\log\la$. Let $t'=\floor{t-t_1}$ and $a=t-t_1-t'$.
Let $F_1'$ be the event that $|X_{t'+1}|=\floor{\la^{1-a}}$. 
Noting that $x=\floor{\la^{1-a}}$ is the most likely value $x$ of $Z$ with
$x\le \la^{1-a}$, arguing as above we have
$\Pr(F_1'\mid F_0)\ge \las^{g(a)}/\la$, and hence
$
 \Pr(F_0\cap F_1') \ge \las^{g(t'+a)}/(2\la),
$
provided $\la$ is large enough.

Noting that $t-t'\ge t_1\ge 1$,
let $F_2$ be the event that the ratio $|X_t|/(\la^{t-t'-1}|X_{t'+1}|)$ is between
$9/10$ and $11/10$. Then by Lemma~\ref{smoothlate}
we have $\Pr(F_2\mid F_0\cap F_1')\to 1$, so
\[
 \Pr(F_0\cap F_1'\cap F_2) \ge \las^{g(t'+a)}/(3\la).
\]
Noting that $\la^{1-a}\la^{t-t'-1}=\la^{t-(t'+a)}=\la^{t-(t-t_1)}=\la^{t_1}=\omega$,
and that $\floor{\la^{1-a}}\ge \la^{1-a}/2$,
if $F_0\cap F_1'\cap F_2$ holds then so does the event $E_\rho=F_0\cap F_1\cap \{|X_t|=\rho\}$
for some $\rho$ between $9\omega/20$ and $11\omega/10$. So
there is some $\rho$ in this range for which $\Pr(E_\rho)\ge \las^{g(t-t_1)}/(3\la\omega)$,
as required.
\end{proof}

We also need an analogue of Lemma~\ref{cpl} without the assumption that $\la$ is fixed.
\begin{lemma}\label{cpl inf}
Let $\la=\la(n)$ satisfy $\la\le n^{1/10}$.
Then the estimates
\[
 \Pr\bb{\Gat(x)\isom T}  \sim \Pr\bb{\Gatm(x)\isom T} \sim \Pr\bb{X_{\le t}\isom T}
\]
hold uniformly over rooted trees $T$ with $|T|\le n^{2/5}$, where $t$ is the height of $T$.
\end{lemma}
\begin{proof}
The proof is essentially identical to that of Lemma~\ref{cpl}. Indeed, the estimate
\eqref{cpl1} is valid assuming only that $|T|$, $\la \le n/2$, say; under
our present assumptions this estimate is $\exp(O(n^{-1/5}+n^{-1/10}+n^{-2/5}))=1+o(1)$.
As before, the result for $\Gat(x)$ follows, now noting that the expected
number of untested edges present is $O(\la|T|^2/n)=o(1)$.
\end{proof}
  
\subsection{Neighbourhoods in the graph and how they meet}
Our immediate  plan is to examine those vertices for which the breadth first search procedure takes an unusually long time to reach a `large' number of vertices. For convenience we choose `large' to mean around $\la^{10}$; since we assume $\la<n^{1/1000}$, say, $\la^{10}$ is much less than $n^{1/4}$. We do not attempt to optimise the power of $n$ giving the upper bound on $\la$.
We first work towards   a lemma that gives asymptotically the probability that two neighbourhoods  of size at least $\la^{9}/4$ have a certain distance between them. This will be needed in particular later when we make variance calculations in using the second moment method.

As in Section~\ref{sec_const}, set  
\[
 t_0=\floor{\log n/\log(1/\las)}.
\]
For $r\ge 1$, let $S_r$ be the set of vertices $x$ in the random graph with $|\Gamma_{t_0+10}(x)|=r$.

Lemmas~\ref{ib} and~\ref{ia}, in conjunction with Lemma~\ref{cpl inf}, give some information on the expected size of $S_r$, or, more precisely, on the size of unions of such sets
over $r$ in suitable ranges, though (as will be apparent in the argument below) the upper and lower bounds given by the lemmas can differ by a factor of $\la$ or more. 

We first consider the branching process.
For $r\ge \la$, 
setting $\omega=3r>2r$ in Lemma~\ref{ia} gives
\begin{equation}\label{ub}
\pr (0<|X_t|\le r)<3\las^{g(t-\log(3r)/\log \lambda)}.
\end{equation} 
Although we shall not use it, let us note that
in the other direction, with $\alpha$ constant and $\la$ large enough,
applying Lemma~\ref{ib} with $\omega = \frac12 \alpha r$
and then Lemma~\ref{smoothlate} gives
\begin{equation}\label{lb}
 \pr (0<|X_{t}| < \alpha r) \ge \frac12 \las^{g(t-\log (\alpha r / 2)/ \log \la)},
\end{equation}
provided the argument of $g$ is greater than $0$.

We will transfer the bounds above to the random graph using Lemma~\ref{cpl inf},
which shows that the corresponding random graph and branching process
events have asymptotically the same probability, provided there
are not too many vertices close to $x$, so that the trees used in
applying Lemma~\ref{cpl inf} are not too large. First, define
$\Gamma_{\le i}(x)=\bigcup_{j=0}^i \Gamma_j(x)$.

Let $\bone$ be the (`bad') set of vertices $x$ such that $|\Gamma_{\le t_0+10}(x)|>n^{1/4}$.
From \eqref{lasinf} we have $\log(1/\las)\sim \la$, which is much larger than $\log\la$,
so $\la^{t_0+10}=n^{o(1)}\la^{10}\le n^{1/8}$ if $n$ is large enough.
For fixed $k\ge 1$, the number of unlabelled rooted trees of height $t$
with exactly $k$ (non-root) leaves, all at distance $t$ from the root,
can be estimated by adding paths to leaves one at a time,
giving the crude upper bound $O(1)(t+1)^{k-1}$.
It is thus easily seen that 
for fixed $k$ we have $\E |\Gamma_{\le t}(x)|^k \le O(1)(t+1)^{k-1}\la^{tk}$.
With $t=t_0+10=O(\log n)$ and $k=20$, this gives 
$\E |\Gamma_{\le t_0+10}(x)|^{20} \le(\log n)^{O(1)}n^{2.5}=o(n^3)$. Thus Markov's inequality
gives
\begin{equation}\label{pbs}
 \E|\bone| \le n\Pr\bb{|\Gamma_{\le t_0+10}(x)|^{20} \ge n^5} = o(n^{-1}).
\end{equation}
A similar calculation shows that
\begin{equation}\label{pbb}
 \Pr(\boneb) = o(n^{-2}),
\end{equation}
where $\boneb$ is the branching process event corresponding to $\bone$.

Define $\mut_r$ to be the expected number of vertices $x$ in $S_r\setminus \bone$.
Applying Lemma~\ref{cpl inf} to each relevant tree,
which has at most $n^{1/4}$ vertices
by definition, and summing over $x$,
we have $\mut_r\le (1+o(1))n\Pr(|X_{t_0+10}|=r)$, so by \eqref{ub} we have
\begin{eqnarray}
\mut_r&<& n(3+o(1))\las^{g(t_0+10-\log(3r)/\log \lambda)} \nonumber\\
&<& (3+o(1))\las^{8-\log(3r)/\log \lambda } \label{Nub}
\end{eqnarray}
using~\eqn{g approx}.

We can similarly see easily that the union of the sets $S_r\setminus \bone$ over all $r<\la^{9}/4$ is \aas\ empty: setting $\omega=\la^9/2$ in Lemma~\ref{ia} gives
\begin{eqnarray*}
 \Pr\bb{ 0<|X_{t_0+10}| < \la^9/4} &\le& 3 \las^{g(t_0+1+\log 2/\log \la)}\\
 &\le& \frac{3}{n} \las^{g(\log 2/\log \la)}
\end{eqnarray*}
which is $3n^{-1}\Pr(\Po(\la)\le \la/2)=o(1/n)$.
Hence, using Lemma~\ref{cpl inf} again, $\sum_{r<\la^9/4} \mut_r =o(1)$.
Since $\E|\bone|=o(1)$, it follows that
\begin{equation}\label{small r}
 \bigcup_{1\le r\le \la^9/4} S_r = \emptyset \mbox{ whp.}
\end{equation}
Thus, we are interested in $S_r$ for $ r\ge \la^9/4$.

An annoying feature of the present situation is that with some small
probability, the size of $\Gamma_{i}(x)$ can `misbehave' for
$i>t_0+10$. Although there are \aas\ no vertices for which this
happens to a significant extent, we need to treat these vertices
separately.
Define $\ell(r)=\max\{0,\ceil{\log\big(2(\log^6 n)/r\big) / \log \la}\}$,
so $\ell(r)\ge 0$ is minimal subject to $r\la^{\ell(r)} \ge 2\log^6 n$.
Let $\btwo$ be the set of `bad' vertices $x$ with
the property that $|\Gamma_{t_0+10}(x)|\ge \la^9/4$, and
the `ratio error'
\begin{equation}\label{B2rat}
 \frac{|\Gamma_{t_0+10+i}(x)|}{\la^i |\Gamma_{t_0+10}(x)|} -1
\end{equation}
has absolute value at least $\la^{-2}$ for some $0\le i\le \ell(r)$.
We write $V_0=V\setminus (\bone\cup\btwo)$ for the set of `good' vertices.
\begin{lemma}\label{extended regular} 
\begin{description}
\item{(a)} $\ex|\btwo|= o(1)$.
\item{(b)} Conditional on two vertices $x$ and $y$ being in $S_{r_1}\cap V_0$ and $S_{r_2}\cap V_0$ respectively, where $ \la^{9}/4\le r_i\le n^{1/4}$ ($i=1$ and 2), and additionally conditional on $d(x,y)>2t_0+20+\ell(r_1)+\ell(r_2)$, we have
$$
\pr\big(d(x,y)>2t_0+20+k\big)   = \exp\bigg(-\frac{ r_1r_2}{n}\big(1+O(\la^{-2})\big)\sum_{i=1}^k\la^i \bigg) +o(n^{-3}) 
$$
for all $k>\ell(r_1)+\ell(r_2)$, where the constant implicit in the $O(\cdot)$ terms is uniform over all such $r_1$, $r_2$ and $k$.
\end{description}
\end{lemma}
\begin{proof}
As in the proof of Lemma~\ref{reg}, we explore the successive neighbourhoods of a vertex. If $a_t$ denotes $|\Gamma_t(x)|$, then conditional on the part of the graph explored up to this point, and assuming that it contains at most $n^{2/3}$ vertices, $a_{t+1}$ is distributed as binomial with parameters $n-O(n^{2/3})$ and  $p=(\la a_t/n)(1+O(\la a_t/n))$.
The mean is $\la a_t(1+O(n^{-1/4}))$, so by Lemma~\ref{Chern} (a Chernoff bound) we have
\begin{equation}\label{chern}
\pr( |a_{t+1}- \la a_t|\le  (\la a_t)^{3/4}) = 1-e^{-\Omega(\sqrt{\la a_t})}.
\end{equation}

To prove (a), in view of \eqref{pbs} we only need to show that $\ex|\btwo\setminus\bone|
=o(1)$.  First explore the successive
neighbourhoods of any vertex $x$ up to $\Gamma_{t_0+10}(x)$. If the
cardinality of this set, $a_{t_0+10}$, is less than $\la^9/4$ or 
greater than $2\log^6 n$, or if $|\Gamma_{\le t_0+10}(x)|>n^{1/4}$, then $x$ is
certainly not in $\btwo\setminus\bone$.
Condition on the exploration so far assuming that none of these events hold, and that $a_{t_0+10}=r$, so
$\la^9/4\le r \le 2\log^6 n$. Next, continue exploring
a further $\ell(r)$ steps.  Provided the event in the left
side of~\eqn{chern} holds at each exploration step, the `relative
error' $| a_{t+1}/\la a_t -1|$ is at most $ (\la a_t)^{-1/4}$. In this
case,
$$
\left| \frac{a_{t_0+10+i}}{\la^i a_{t_0+10}} -1\right|< 2( \la a_{t_0+10})^{-1/4} ,
$$ which is less than $\la^{-2}$ since $a_{t_0+10}\ge
\la^9/4$. This implies $x\notin\btwo$. On the other hand, the probability
that the event in the left side of~\eqn{chern} fails to hold for at
least one of the relevant $t$ is at most $e^{-\Omega\big(\sqrt{ \la r}\big) }$,
 which is $\las^{ \Omega\left(\sqrt{ r/\la}\right)}$
since $\las>e^{-\la}$. The expected number of vertices $x\notin \bone$ with
$a_{t_0+10}=r$ is $O(\las^{8-\log(3r)/\log \la})$ by~\eqn{Nub}.
Multiplying these bounds together and summing over $r\ge \la^9/4$ gives $o(1)$, that is, $\ex|\btwo\setminus\bone|=o(1)$, as required.

We turn to (b). Let $ \la^{9}/4\le r_i\le n^{1/4}$ ($i=1$ and 2). Take any vertices $x$ and $y$, and explore the successive neighbourhoods  of each  up to distance $t_0+10+\ell(r_1)$ and $t_0+10+\ell(r_2)$ respectively. At this point, it is revealed whether these neighbourhoods are all disjoint, which is equivalent to $d(x,y)>2t_0+20+\ell(r_1)+\ell(r_2)$, and also (recalling that $V_0=V\setminus(\bone\cup\btwo)$) whether
$x \in V_0$ and $y \in V_0$. Condition on the event that all three of these hold. It follows from $x\notin\btwo$ that $|\Gamma_{ t_0+10+i}(x)|=r_1(1+O(\la^{-2}))\la^i$ for $0\le i \le \ell(r_1)$, and similarly for $y$.  

We next explore the further neighbourhoods of $x$ and $y$, each time choosing the smaller of the two for further exposure, until one of them has reached cardinality at least $n^{3/5}$, or until they meet, whichever happens first. Note that for all $r$ we have $r  \la^{\ell(r)}\ge 2\log^6 n$ by the definition of $\ell$.
Since $x\notin\btwo$, using the `error ratio' property in the definition of $B_2$ (see \eqref{B2rat}) it follows that $|\Gamma_{t_0+10+\ell(r_1)}(x)|\ge \log^6 n$,
and similarly for $y$.
So, by applying~\eqn{chern} and conditioning on non-failure at each step, we conclude that with probability at least $1-o(n^{-3})$, at each step
$$
|\Gamma_{ t_0+10+k}(x)|=r_1 \la^k (1+O(\la^{-2})),
$$ 
and similarly for $y$.
So we may assume this is the case each time.
{From} this, when the sum of the two distances is  $2t_0+20+k-1$, the product of the sizes of the neighbourhoods is $ r_1r_2 \la^{k-1}(1+O(\la^{-2}))$, and hence the probability of not joining in the next step is
$$ 
\exp\big(-r_1r_2 \la^{k-1}(1+O(\la^{-2}))\la/n\big).
$$
The result follows, as long as the probability that they do not meet by the time that one of the neighbourhoods has reached size $n^{3/5}$ is bounded above by $o(n^{-3})$. This must be the case since on the previous step, the neighbourhood that was extended must have had size at least $n^{3/5}/\la(1+o(1))$, so the product of sizes on the previous step must have been at least $n^{6/5}/\la^2(1+o(1))$ which is at least $n^{11/10}$ as $\la<n^{1/1000}$. Thus the probability of not joining on the last step was at most $\exp (-\la n^{1/10}(1+o(1))=o(n^{-3})$.
\end{proof}

We now turn to the proof of Theorem~\ref{th_inf}.
\begin{proof}[Proof of Theorem~\ref{th_inf}.]
Recall that $\la=\la(n)$ is some given function of $n$ satisfying $\la\to\infty$
and $\la\le n^{1/1000}$. All limits are as $n\to\infty$, or, equivalently, as $\la\to\infty$.
As usual, all inequalities we claim are required to hold only if $n$ (or $\la$) is
sufficiently large.

Our first aim is to estimate the probability of the event conditioned
on in Lemma~\ref{extended regular}(b).
Let  $\widehat P_r$ denote the probability that a given vertex is in $V_0\cap S_r$, and $\widehat P_{r_1,r_2}$ the probability that a given pair of distinct vertices $x$ and $y$ satisfy   $x\in V_0\cap S_{r_1}$, $y\in V_0\cap S_{r_2}$, and $d(x,y)>2t_0+20+\ell(r_1)+\ell(r_2)$.
Note that $x\in V_0\cap S_{r_1}$ iff the set of vertices at distance at
most $t_0+10+ \ell(r_1)$ from $x$ forms one of a specific set of graphs
with less than $n^{1/4}+O(\la(\log n)^6)=o(n^{1/3})$ vertices,
and $\widehat P_{r_1,r_2}$ counts configurations in which the explorations
from $x$ and $y$ are disjoint. Since each exploration `uses up'
$o(n^{1/3})$ vertices, it is easy to see (for example using
a version of Lemma~\ref{cpl inf} starting with two vertices) that
\begin{equation}\label{prodisgood}
\widehat P_{r_1,r_2} \sim \widehat P_{r_1}\widehat P_{r_2}.
\end{equation}

For any $r$, let $\muh_r$ denote $n\widehat P_r$, the expected size of $V_0\cap S_r$;
recall that $\mut_r=\E |S_r\setminus \bone|$, so
$\muh_r=\mut_r+o(1)$ by Lemma~\ref{extended regular}(a). Also, for integer
$k\ge 1$ define $\muh(r_1,r_2,k)$ to be  the expected   number of ordered pairs $(x,y)$ of vertices with $x\in V_0\cap S_{r_1}$, $y\in V_0\cap S_{r_2}$, and
$d(x,y)>2(t_0+10)+k$.
Since $V_0=V$ whp, the number of such pairs
essentially determines the diameter. From the above observations and Lemma~\ref{extended regular}(b),
\begin{equation}\label{Ehat}
\muh(r_1,r_2,k) \sim\muh_{r_1}\muh_{r_2}\left(\exp\Big(-r_1r_2 (1+O(\la^{-2})) \sum_{i=1}^k\la^i/n\Big)+o( n^{-3})\right)
\end{equation}
provided $k >\ell(r_1)+\ell(r_2)$ and $r_1$ and $r_2$ satisfy the constraints of Lemma~\ref{extended regular}.  
Note that we shall consider values of $k$ that are at least $\log n/\log\la-30$, which is larger than $2\ell(r)$ for any $r>0$.

Define $\mun_r=n\pr(|X_{t_0+10}|=r)$, which we shall analyse using Lemmas~\ref{ib} and~\ref{ia}. 
We claim that
\begin{equation}\label{P1}
\muh_r \le \mut_r\le \mun_r(1+o(1));
\end{equation}
indeed, the first inequality holds by definition. If $r>n^{1/4}$
then $\mut_r=0$; otherwise the second inequality follows from Lemma~\ref{cpl inf},
summing over the possible neighbourhoods of $x$.
In the other direction, although we shall not these bounds,
note that for $r\le n^{1/4}$ we have
$$
\muh_r \ge \mun_r(1+o(1))+o(1),
$$
since $\muh_r\ge \mut_r+o(1)$ by Lemma~\ref{extended regular}(a), and
$\mut_r\ge \mun_r(1+o(1))+o(1)$ from Lemma~\ref{cpl inf}, together with \eqref{pbb}.
 
We next show that
vertices in sets $V_0\cap S_r$ with $r> \la^{13.1}$ will not determine the diameter of the graph, for the reason that they join too quickly to all vertices under consideration: we claim that \aas\ all such vertices have distance at most $\log n/\log \la +2t_0-0.05$ from all other vertices; we will see later that {\aas} the diameter is greater than this. To establish this claim, without loss of generality consider only   
$r_1 \ge \la^{13.1}$ and $r_2\ge \la^9/4$. Note that the conditions on the $r_i$ in  Lemma~\ref{extended regular}(b) are so restrictive because it aims for a fairly accurate asymptotic estimate. In this case we only need to observe that if  $x\in V_0\cap S_{r}$ for $r=r_1$ or $r_2$,  by definition of $\btwo$,  $|\Gamma_{t_0+10+i}(x)|\sim \la^i r$ until the neighbourhoods
reach size at least $(\log n)^6$ (which they may do at $i=0$), and for larger neighbourhoods up to size $n^{2/3}$,~\eqn{chern} provides the same relation with probability at least $1-e^{-\Omega(\log ^3 n)}= 1-o(n^{-5})$.
Summing over all $O(n^2)$ pairs of vertices $x$ and $y$ gives
\begin{equation}\label{mrk}
\muh(r_1,r_2,k) \le \muh_{r_1}\muh_{r_2}\exp\Big(-(1+o(1))r_1r_2  \la^k/n\Big)  +o(n^{-3}),
\end{equation}
which is similar to~\eqn{Ehat} but does not have the same restrictions on $r_1$ and $r_2$.
For $k=\floor{ \log n/\log \la -20.05}$ we have 
 $r_1r_2\la^k/n>(r_1/\la^{13.1})(4r_2/\la^9) \la^{1.04}$. 
Now~\eqn{P1} and~\eqn{Nub}, together with $\las=e^{-\la +o(\la)}$ (see \eqref{lasinf}), give
$$
\muh_{r}=O(1)\exp\big( (1+o(1))\la (\log(3r)/\log \la-8)\big).
$$
Summing the resulting bound on $\muh(r_1,r_2,k)$ over all $r_1 \ge \la^{13.1}$ and $r_2\ge \la^9/4$
gives $o(1)$, as required to establish the claim.
(The key observation
is that when $r_1$ and $r_2$ take their minimum values, we have $\muh_{r_1}\muh_{r_2}=
\exp(O(\la))$, while the exponential factor in \eqref{mrk} is
at most $\exp(-\la^{1.04})$. When $r_1$ and $r_2$ increase, so does $\muh_{r_1}\muh_{r_2}$,
but the exponential factor decreases more than fast enough to compensate.)

Recalling~\eqn{small r}, let $R$ be the set of indices $r$, $\la^{9}/4\le r\le  \la^{13.1}$, for which $\muh_r>\la^{-14}$. Then, by the union bound, the expected number of vertices in all sets $V_0\cap S_r$ with $r$ in this range but not in $R$ is $o(1)$, i.e., there are \aas\ no such vertices. Since $V_0=V$ whp,
using the observation above about sets $S_r$ with $r>\la^{13.1}$ and~\eqref{small r}, we have
shown that
\begin{equation}\label{dmain}
 \diam(G) = \max_{(r_1,r_2)\in R^2} \max \{ d(x,y): x\in V_0\cap S_{r_1},\,y\in V_0\cap S_{r_2}\} \mbox{ whp.}
\end{equation}

It only remains to examine $r_1$ and $r_2$ in $R$.
Note that if $r\in R$ then $r\le \la^{13.1}$, so from~\eqn{Nub} and~\eqn{P1} we have
\begin{equation}\label{mribd}
 \muh_r<\las^{-6}<e^{6\la}.
\end{equation}
Let $k_0(r_1,r_2)$ denote the maximum $k$ such that $\muh(r_1,r_2,k)>\la^{-27}$.  (This number $k_0$ depends on $n$.)
Then  $\muh(r_1,r_2,k_0(r_1,r_2)+1)\le \la^{-27}$. Let $k_{\max}$ be the maximum value of $k_0$ over all pairs $(r_1,r_2)$ in $R^2$.
{From} \eqref{mrk}, \eqref{mribd} and the definition of $R$, it is easy to check that $\kmax=\log n/\log\la+O(1)$.
Setting $f(n,\la)=2(t_0+10)+\kmax$, to prove the first part of Theorem~\ref{th_inf}
we shall show that the diameter is \whp\ either $f(n,\la)$ or $f(n,\la)+1$.
Since $|R^2|=O(\la^{26.2})$, by the union bound, the expected number of pairs of vertices $x$ and $y$ counted in \eqref{dmain} at distance greater than $f(n,\la)+1$ is $o(1)$. Thus $\diam(G)\le f(n,\la)+1$ holds whp.

To see that the diameter is \aas\ at least $f(n,\la) = 2(t_0+10)+k_{\max}$
we shall look for vertices at this distance in suitable sets $S_{r_i}$.
Choose $(r_1,r_2)$ in $R^2$ with   $k_0(r_1,r_2) = k_{\max}$.  Note that $\muh(r_1,r_2,k_{\max})>\la^{-27}$. That is, from~\eqn{Ehat}, 
$$
\muh_{r_1}\muh_{r_2}\exp\Big(-r_1r_2 (1+O(\la^{-2})) \sum_{i=1}^{k_{\max}}\la^i/n\Big) + o(\muh_{r_1}\muh_{r_2}/n^3)>(1+o(1))\la^{-27}.
$$

By definition   $\muh_{r_i}\le n$, and $n^{-1}=o(\la^{-27})$, so
$$
\muh_{r_1}\muh_{r_2}\exp\Big(-r_1r_2 (1+O(\la^{-2})) \sum_{i=1}^{k_{\max}}\la^i/n\Big) >\la^{-27}(1+o(1)).
$$
Using \eqref{mribd} for $r=r_1$ and $r=r_2$,
it follows that
$$
\exp\Big(-r_1r_2 (1+O(\la^{-2})) \sum_{i=1}^{k_{\max}}\la^i/n\Big) > \la^{-28}e^{-12\la}
>e^{-13\la},
$$ 
if $n$ is large enough.
Taking logs and stopping the sum one step earlier, this gives
\begin{equation}\label{nottosmall}
-r_1r_2 (1+O(\la^{-2})) \sum_{i=1}^{k_{\max}-1}\la^i/n >  -13.
\end{equation}
Hence, by Lemma~\ref{extended regular}(b),
vertices $x$ and $y$ whose $(t_0+10)$-neighbourhoods have sizes $r_1$ 
and $r_2$ respectively have a significant (at least $e^{-13}+o(1)$) probability
of being at distance at least $2t_0+20+k_{\max}$. 
Although by design we expect a large number of pairs of such vertices $x$ and $y$, it is still possible that the expected number of possibilities for either $x$ or $y$ goes to 0! Our strategy is to consider vertices
with $|\Ga_{t_0+10}(\cdot)|$ around $2000r_i$, say, and show that this gives
us many vertices $x$ and $y$ to work with. We also impose certain extra conditions
on their neighbourhoods needed later.

For $i=1,2$, since $r_i$ is in $R$, we have $\muh_{r_i}>\la^{-14}$.
Now~\eqn{P1} shows that $\Pr(|X_{t_0+10}|=r) = \mun_{r_i}/n>(1+o(1))\la^{-14}/n$.
By \eqref{ub} it follows that
\begin{equation}\label{nlpre}
 \las^{g(t_0+10-\log(3r_i)/\log \lambda)}>(1/3+o(1))\la^{-14}/n.
\end{equation}
Let $\omega_i=1000r_i \le\la^{14}$.
By \eqref{addb} and \eqref{nlpre} we have
\begin{equation}\label{nl}
 \las^{g(t_0+10-\log \omega_i/\log \lambda)}>(1/3+o(1))\la^{250/3}\la^{-14}/n \ge \la^{40}/n,
\end{equation}
if $\la$ is large enough.
For $i=1,2$, applying Lemma~\ref{ib} with $\omega=\omega_i$ and $t=t_0+10$,
there is some $\rho_i$ with $\omega_i/3\le \rho_i\le 2\omega_i$
such that the event $F_0\cap F_1\cap\{|X_{t_0+10}|=\rho_i\}$ described
in Lemma~\ref{ib} has probability $\pi_i$ satisfying
\begin{equation}\label{pibig}
 \pi_i \ge \las^{g(t_0+10-\log \omega_i/\log \lambda)}/(3\la\omega_i) \ge \la^{39}/(3n\omega_i) \ge \la^{25}/n,
\end{equation}
using \eqref{nl}.
Let $\wtE_{\rho_i}(x)$ denote the event that $x\notin \bone$ and the neighbourhoods of $x$ up to distance
$t_0+10$ form a tree that, when viewed as a branching process, satisfies
the conditions $F_0\cap F_1\cap \{|X_{t_0+10}|=\rho_i\}$.
By \eqref{pbb} and Lemma~\ref{cpl inf}, we have $\Pr(\wtE_{\rho_i}(x)) \sim \pi_i+o(n^{-2})$.
Since $\pi_i$ is much larger than $n^{-2}$, it follows that
$
 \Pr(\wtE_{\rho_i}(x)) \sim \pi_i \ge \la^{25}/n.
$

Let $\Eri(x)$
be the event that $\wtE_{\rho_i}(x)$ holds and $x\in V_0$, so
the only additional condition is that $x\notin \btwo$.
Let $P_i=\Pr(\Eri(x))$. Since $\Pr(x\in\btwo)=o(1/n)$, we have
\begin{equation}\label{newerup}
 P_i\sim \Pr(\wtE_{\rho_i}(x)) \sim \pi_i \ge \la^{25}/n.
\end{equation}
 
Note also for later that, writing $t_i$ and $a_i$ for the integer and fractional
parts of $t_0+10-\log\omega_i/\log\la$, and writing
$F_0(x)$ for the event that the neighbourhoods
of $x$ satisfy the diamond condition to distance $t_i$ (corresponding to $F_0$ in Lemma~\ref{ib}),
then starting from the first statement of Lemma~\ref{ib} and arguing as above we have
\[
 \Pr\bb{F_0(x)\cap \{ |\Ga_{t_i+1}(x)| \le \la^{1-a_i} \}} \sim\las^{g(t_0+10-\log\omega_i/\log\la)}.
\]
Using the first inequality in \eqref{pibig} it follows that
\begin{equation}\label{nn}
 P_i \ge \la^{-15} \Pr\bb{F_0(x)\cap \{|\Ga_{t_i+1}(x)| \le \la^{1-a_i} \}}.
\end{equation}
In other words, once we have explored the neighbourhoods to the `branching vertex'
$x_0$, and found few neighbours in the next step, it is not that unlikely that $\Eri(x)$ holds.

Given distinct vertices $x$ and $y$, as in~\eqn{prodisgood} the probability that $E_{\rho_1}(x)$  
and $E_{\rho_2}(y)$
hold and $d(x,y)>2(t_0+10)+\ell(\rho_1)+\ell(\rho_2)$ is $(1+o(1))P_1P_2$.
Furthermore, conditional on this holding,
then by a variant of Lemma~\ref{extended regular} that simply
includes extra conditions on the neighbourhoods of a vertex up to distance $t_0+10$,
the conditional probability $P$ that $d(x,y)\ge 2(t_0+10)+\kmax$ satisfies
\begin{equation}\label{Pform}
 P = \exp\bigg(-\frac{\rho_1\rho_2}{n}(1+o(1))\sum_{i=1}^{\kmax-1}\la^i \bigg) +o(n^{-3}).
\end{equation}
Since $\rho_i\le 2\omega_i=2000 r_i$, using \eqref{nottosmall} shows
that $P\ge \exp(-O(1))$, so $P=\Theta(1)$.

Let us call an ordered pair $(x,y)$ a {\em regular far pair} if $E_{\rho_1}(x)$
and $E_{\rho_2}(y)$ hold, and $d(x,y)\ge 2(t_0+10)+\kmax$, and let
$N$ denote the number of regular far pairs; our aim is to show that $N\ge 1$ holds whp.
{From} \eqref{newerup} we have
$nP_1, nP_2\ge (1+o(1))\la^{25}\to\infty$, so
$$
 \E N \sim n^2P_1P_2P \to\infty.
$$
Unfortunately, we cannot use the trick from
Subsection~\ref{lower_fixed} to complete the proof: this trick, which
allowed us to avoid considering the second moment of the number of
{\em pairs} of vertices at large distance, needed $P\sim 1$. This will
in fact hold for almost all values of the parameters in the present setting, but not all. Moreover,  we now have less tolerance in the final estimate
of the diameter, and consequently less flexibility.
Instead we apply the second moment method directly to $N$.
In the arguments that follow we shall avoid using the fact
that $P=\Theta(1)$, using only
\begin{equation}\label{Pb}
 P\ge n^{-1/20},
\end{equation}
say; this will be useful later.

Let $M=\E(N^2)$ denote the expected number of pairs $((x,y),(z,w))$ of
regular far pairs; our aim is to show that $\E M\sim (\E N)^2$.
Note that the number of distinct vertices in $\{x,y,z,w\}$ may be $2$, $3$ or $4$.
The contribution to $M$ from sets with 2 distinct vertices is trivially
at most $2\E N=o((\E N)^2)$ (the factor $2$ arises only if $\rho_1=\rho_2$).
Let us leave aside the case of $3$ vertices, noting only that we expect
the contribution from pairs with $x=z$, say, to be asymptotically
\[
 nP_1 (nP_2)^2 P^2 \sim (\E N)^2 / (nP_1) = o((\E N)^2),
\]
since $nP_1\to\infty$.
The argument for the case of 4 distinct vertices that we shall now give
adapts easily to show this.

Let $M_0$ be the contribution to $M$ arising from sets of 4 distinct
vertices $\{x,y,z,w\}$ whose neighbourhoods up to distance
$t_0+10+\ell(\rho_i)$ are all disjoint, where $i=1$ or $2$ as appropriate.
To estimate $M_0$, explore from four distinct vertices, and
test whether the relevant events $\Eri(\cdot)$ hold with the neighbourhoods disjoint.
As in~\eqn{prodisgood}, this has probability $(1+o(1)) P_1^2P_2^2$.
Our aim is to bound from above the conditional probability that
$d(x,y), d(z,w)\ge 2(t_0+10)+\kmax$, showing that it is at most
$(1+o(1))P^2$. Since none of $x$, $y$, $z$, $w$ is in $\btwo$, the neighbourhoods
have already reached size at least $\log^6 n$.
{From} this point onwards, as before,
we may assume they grow at almost exactly the expected rate.
Note that we may
ignore events of conditional probability $o(n^{-1/10})=o(P^2)$, since we have already
conditioned on an event of probability $(1+o(1))P_1^2P_2^2$.

Since we stop the explorations when the neighbourhoods are no larger than $n^{3/5}$, say,
we may assume that any intersections between neighbourhoods are small,
involving at most a fraction $n^{-1/3}$ of the vertices in a neighbourhood.
Such small intersections do not materially affect the calculations
in Lemma~\ref{extended regular}(b), so the conditional probability
that $d(x,y),d(z,w)\ge 2(t_0+10)+\kmax$ is indeed $(1+o(1))P^2$.

It remains to deal with cases where some of the neighbourhoods meet within
distance $t_0+10+\ell(\rho_i)$ from the respective vertices.
As above just after~\eqn{nl}, let $t_i$ be the relevant parameter $t'$ in Lemma~\ref{ib},
where $i=1$ or $2$ depending on which vertex we consider.
Note that to have the property $\Eri(v)$, all our starting vertices $v$
must have the property that $\Gamma_{t_i}(v)$ contains a unique vertex $v_0$.
Also, within the tree up to this point, $v$ must be the unique vertex at maximal
distance from $v_0$,
 so our `diamond' condition holds. As in Subsection~\ref{lower_fixed}, it follows
that in a quadruple contributing to $M$, the neighbourhoods
cannot meet before the corresponding vertices $v_0$, so the minimum possible
distance between starting vertices is $t_i+t_j$.

Returning to the random graph without conditioning, let us explore the neighbourhoods of our 4 distinct vertices $x$, $y$, $z$, $w$
out to distance $t_i-1$ in each case, assuming these explorations
are disjoint, and that there are no edges between the final
sets (such an edge would give distance $t_i+t_j-1$). Furthermore, let us test for
each of these vertices $v$ {\em how many} neighbours $\Ga_{t_i-1}(v)$ has in the remaining
set $U$ of `unused' vertices, but not \emph{which} neighbours it has. If our quadruple
is to contribute, in each case there must be exactly one neighbour, $v_0$.
Now conditional on the information so far, the probability that $x_0=z_0$, say,
is exactly $1/|U|\sim 1/n$. If this happens, then going forwards, the remaining
calculations are exactly as if we had $x=z$ in the beginning. Summing
the corresponding contributions to $M$, the total from cases with $x\ne z$ but $x_0=z_0$
has an extra factor of $n$ from the choice of $z$ (compared to the
case $x=z$), but also an extra factor that is asymptotic to $1/n$ as noted above. (There
is also the extra factor of at most 1 from the condition on the neighbourhoods
of $z$ up to distance $t_i-1$; we can ignore this). In total,
the contribution here is at most that with $x=z$, which is $o((\E N)^2)$
as noted above. (The argument here is not circular; when considering here
the three-vertex case, a collision of this form reduces to the two-vertex
case.)

So we may assume that $x_0$, $y_0$, $z_0$ and $w_0$ are distinct.
Repeating the trick above, let us first test how many neighbours
each has among the unused vertices (not testing edges such as $x_0z_0$
for now). For our quadruple to contribute, by definition of $\wtE_{\rho_i}(\cdot)$ the numbers must be at
most $\la^{1-a_i}$ with $i=1,2$ as appropriate. Since there are
$n-O(n^{1/4})$ unused vertices, the probability of this happening
is very close to $\Pr(\Po(\la)\le \la^{1-a_i})$.
Using \eqref{nn},
it follows that the probability
that all our tests so far, for the relevant events $\wtE_{\rho_i}(\cdot)$, succeed is at most $\la^{61}P_1^2P_2^2$.
Hence, going forward, we may neglect any event of probability smaller than
$n^{-1/4} = o(\la^{-61})$, say.
So far we revealed the numbers of neighbours, which were all at most $\la$,
but not which vertices
they were. But the probability of a collision is $O(\la^2/n)=o(n^{-1/4})$,
which is negligible. Also, the probability of an edge between $x_0$ and $z_0$,
say, is $O(\la/n)=o(n^{-1/4})$.  Recall that any vertex
in a pair counted in $N$, or a quadruple in $M$ or $M_0$, has the property
$\Eri$ for some $i$ and is hence in $V_0=V(G)\setminus(\bone\cup\btwo)$.
Exploring further
up to distance $10+\ell(\rho_i)$ steps from each vertex $v_0$, where $i=1$ or 2 as appropriate,
assuming typical growth as we may, the probability that two neighbourhoods
meet, starting as they do with at most $\la$ neighbours of $v_0$,
is $O(\la^{20+\ell(\rho_i)+\ell(\rho_j)}/n)=o(n^{-1/4})$. 
So we may assume this does not happen,
and hence $M-M_0$ is negligible compared with $M$.

In summary, it follows that $M=\E(N^2)\sim n^4P_1^2P_2^2P^2 \sim (\E N)^2\to\infty$,
so the second moment method shows that $N\ge 1$ whp.
But then the diameter is at least $2(t_0+10)+\kmax$, completing the proof
of the first half of Theorem~\ref{th_inf}.
 
The second part of the theorem states that for `most' values of $n$
the diameter is almost determined, and gives a
formula. The general
exact formula is a bit complicated if we want to include all values of the parameters,
even restricting to those for which the diameter is almost determined.
In formulating Theorem~\ref{th_inf} we omitted some additional problematic
values of $n$, giving a much simpler formula.
One way to explain the source of the problematic cases is to observe that,
although the difference between the upper and lower bounds \eqn{ub} and \eqn{lb} is usually negligible, when the typical diameter is close to jumping to the next integer, the fact that these bounds do not exactly match becomes important.  

Writing $\fp{x}$ for $x-\floor{x}$, in proving the second part
of the theorem we may assume that
\begin{equation}\label{2c}
\begin{array}{ccccc}
  5 \eps &<&   \fp{\log n/\log\la}       &<& 1-5\eps,\\
  5 \eps &<&   \fp{\log n/\log(1/\las)}  &<& 1-5\eps,
\end{array}
\end{equation}
where $\eps$ is some positive constant, which we may take to be smaller than $1/10$.

Let us first consider some values
of $r$ that, as it will turn out, in many cases (i.e., for many values
of $n$) typically determine the diameter of the random graph.

Define $q_n$ to be the  infimum
of $q$ such that $n^{-1} > \las^{t_0 + g(q)}$.
{From} the definition of $g$, with $\la$ fixed and $q$ varying, $\las^{g(q)}$
jumps by a factor of at most $\la$ at each discontinuity. (With $X\sim \Po(\la)$,
the ratios $\Pr(X\le k+1)/\Pr(X\le k)$ are between $1$ and $\la$,
while the ratio $\las^{-1}\Pr(X\le 1)/\Pr(X<\la)$ is asymptotically $1/\Pr(X<\la)\sim 2$.)
Thus for large $n$
\begin{equation}\label{n-1}
n^{-1} = \las^{t_0 + g(q_n)}/\xi
\end{equation}
for some $\xi=\xi(n)$ between 1 and $\la$.
We call $n$ `normal' if  $ q_n< \eps  $ and $g(q_n)>4\la^{-\eps}$.
Taking logs in \eqref{n-1}, since $\xi=\las^{o(1)}$, while $t_0=\floor{\log n/\log(1/\las)}$,
we have $g(q_n)=\fp{\log n/\log(1/\las)}+o(1)$, so $g(q_n)\ge \eps\ge 4\la^{-\eps}$
if $n$ is large.
Since for any constant $0<a<1$ we have $g(a)\to 1$, while $g(q_n)\le 1-\eps$,
it follows that $q_n=o(1)$, so any (large enough) $n$ satisfying \eqref{2c} is normal.

Putting $t=t_0+10$ and $\omega=\la^{t_1}$ such that $t_1=10-q_n-\log 5/\log \la$ in Lemma~\ref{ia}, we find
$$
\pr(0<|X_{t_0+10}|< \la^{10-q_n}/10)\le 3\las^{g(t_0+q_n +\log 5/\log \la)}=3\las^{t_0+g( q_n +\log 5/\log \la)},
$$
which is at most
$3\las^{t_0+ g( q_n)}/\la^{5/4} =o(n^{-1})$
by~\eqn{addb} and \eqref{n-1}.
Hence, arguing as for \eqref{small r}, we only need to consider vertices
in $S_r$ with $r\ge \la^{10 -q_n}/10$.

Put  $b=\floor{\log n / \log \la +2q_n}$ and 
$\phi = \fp{\log n / \log \la +2q_n}$. Call $n$ `standard' if $3\eps < \phi<1-3\eps$. Since $q_n<\eps$ for normal $n$,
any $n$ satisfying \eqref{2c} is standard.

As noted above, for normal $n$ we only need to consider $r_1$ and $r_2$
at least $\la^{10-q_n-o(1) }$, and for such cases~\eqn{Ehat} gives
\begin{eqnarray*} 
\muh(r_1,r_2,b-18)&\le& (1+o(1))\muh_1\muh_{r_2}
\exp(-\la^{20-2q_n-o(1)+b-18-\log n/\log \la}+o(n^{-3}))\\ &=&
(1+o(1))\muh_{r_1} \muh_{r_2}\exp(-\la^{2-\phi-o(1)}+o(n^{-3})).  
\end{eqnarray*}
For standard $n$ the exponential above is at most
$\exp(-\la^{1+\eps-o(1)})+o(n^{-3})$. Hence for such $n$ the quantity
$\muh(r_1,r_2,b-18)$
goes to 0 quickly unless $\muh_{r_1}$ or $\muh_{r_2}$ is much
bigger than $e^ {100\la}$ say. From arguments as above, we know this
forces $r_1$ and $r_2$ to be much larger
than the typical values of around $\la^{10}$,
at least $\la^{100}$, say, and
then $\muh(r_1,r_2,b-18)$ is much smaller. Using the argument
that earlier permitted us to restrict parameters to the set $R$, such
cases can be neglected. Thus, \aas\ there are no vertices in sets
$S_{r_1}\cap V_0$, $S_{r_2}\cap V_0$ that have
distance greater than $2(t_0+10)+b-18$, for any $r_1$ or $r_2$. Hence
the diameter is at most $2t_0+b+2$ \aas\ for any $n$ satisfying~\eqref{2c}
(or indeed, though we won't need it, for any normal standard $n$).

Continuing with standard normal $n$, let $\omega=\la^{10}$. Then using~\eqn{n-1} and since $g(q_n)\ge 4\la^{-\eps}$, $\las^{-1}=e^{\la+O(\log\la)}$ and $\xi=e^{O(\log\la)}$,
\begin{eqnarray*}
 \las^{g(t_0+10-\log\omega/\log\la)} &=& \las^{t_0} \\
&=&\las^{-g(q_n)}\xi/ n\\
&>& \exp(4 \la^{1-\eps}+O(\log\la))/n\\
&>& \exp(3 \la^{1-\eps})/n.
\end{eqnarray*}
Since the final bound is larger than $\la^{40}/n$ if $\la$ is large enough,
the bound \eqref{nl} holds with $\omega_i=\omega$ for $i=1,2$.
The calculations down to \eqref{nn} go through as before,
now with $\rho_1=\rho_2=\rho$ and $\la^{10}/3\le\rho\le 2\la^{10}$.
This time we have $P_1=P_2\sim \pi_i \ge \exp(3 \la^{1-\eps})\la^{-O(1)}/n$,
using \eqref{pibig} and the bound above.

Writing $N$ for the number of pairs of
vertices with property $E_\rho$ at distance at least
$2(t_0+10)+b-18$, as before we have $\E N\sim (P_1n)^2 P$, with
\[
 P = \exp\bigg(-\frac{\rho^2}{n}(1+o(1))\sum_{i=1}^{b-19}\la^i \bigg) +o(n^{-3})
\]
in place of \eqref{Pform}.
Since $\rho\le 2\la^{10}$ we have
\[
 \log (1/P) \le (4+o(1)) \la^{20+b-19-\log n/\log \la} \sim  4\la^{1+2q_n-\phi} \le 4\la^{1-\eps}
\]
for normal standard $n$. Since $P_1n \ge\exp(3 \la^{1-\eps}-O(\log\la))$, we thus have $\E N\to\infty$.
The second moment argument goes through as before to show that \aas\ $N\ge 1$,
so the diameter is whp at least $2(t_0+10)+b-18$.
(Note that we still have~\eqn{Pb} since 
\eqref{2c} forces $\las$ to be much larger than $1/n$,
and hence $\la=O(\log n)$, so $\log(1/P)=o(\log n)$.)
Hence, from the upper bound shown above, the diameter of the graph is, for  normal standard $n$, \aas\ equal to 
$$
 2t_0+b+2 =  2\floor{\log n/\log(1/\las)}+\floor{\log n / \log \la +2q_n}+2.
$$
Using \eqref{2c} again, and recalling that $q_n<\eps$, this is
$$
 2\floor{\log n/\log(1/\las)}+\floor{\log n / \log \la}+2, 
$$
which is in turn exactly the diameter claimed in \eqref{normalform}, completing the proof of Theorem~\ref{th_inf}.
\end{proof}

\begin{remark}
With hindsight, it is easy to explain intuitively why the diameter
in the last case treated above is given by \eqref{normalform}.
Indeed, with $t_0=\floor{\log n/\log(1/\las)}$, the probability
that a given vertex has $|\Gamma_{t_0}(v)|=1$ is roughly $\las^{t_0}$, which is
significantly larger than $1/n$. On the other hand, typically no vertices
will have $|\Gamma_{t_0+1}(v)|=1$, or indeed $|\Gamma_{t_0+1}(v)|$ much smaller
than $\lambda$. So the diameter is likely to come from two of these `candidate' vertices
with $|\Gamma_{t_0}(v)|=1$.
Each of these has a unique vertex at distance $t_0$. Let us call such vertices {\em active}. Any given pair of active vertices will usually be at distance    $d=\ceil{\log n/\log\la}$ from each other.
However, there are usually many candidate vertices (at least $\las^{-\eps}$, which is roughly
$e^{\eps\la}$), and hence (not necessarily, but usually) about the same number of active vertices.  The expected number of paths of length $d$ joining two given active vertices is roughly $\la^d/n=\la^{1-f}$, so we might expect the probability that a given pair is at distance $d+1$
to be of order $\exp(-\la^{1-f})$, where $f$ is the fractional part of $\log n/\log\la$. 
The probability of no path of length $d+1$ is roughly $\exp(-\la^{2-f})$, which is
much smaller than the reciprocal of the number of pairs of candidate vertices.
So we expect the diameter to be $2t_0+d+1$ whp, as we have shown.
\end{remark}

\section{Just above the critical point}\label{sec_eps}

In this section we shall prove Theorem~\ref{th_eps},
which is the analogue of Theorem~\ref{th_const} for
$G(n,\la/n)$, where now $\la=1+\eps$ with $\eps=\eps(n)$ tending
to zero at a suitable rate.
Roughly speaking, we shall simply repeat
the arguments in Section~\ref{sec_const} more carefully; however,
there are many additional complications that we shall contend with
as we go.
As mentioned in the introduction,
we shall also prove a stronger result, describing the (normalized)
limiting distribution of the correction term; we postpone the somewhat
unpleasant statement of this result until Section~\ref{sec_dist}.

Throughout this section we write $\la$ for $1+\eps$, always assuming that $0<\eps<1/10$,
and often that $\eps=\eps(n)\to 0$. 
As before,
we write $\las$ for the unique solution $\las<1$ to $\las e^{-\las}=\la e^{-\la}$, so
\begin{equation}\label{lasasympt}
 \las = 1-\eps + \frac{2}{3}\eps^2 -\frac{4}{9}\eps^3 +O(\eps^4).
\end{equation}
Sometimes it will be convenient to note that
\begin{equation}\label{las1e}
 \las > 1-\eps
\end{equation}
for all $\eps>0$; this is easily seen using the fact that $\las e^{-\las}$ has positive derivative, and $(1-\eps) e^{-(1-\eps)}<\la e^{-\la}$.
As before we write
$s$ for the survival probability of the branching process $\bp_\la$, so (from \eqref{sdef}),
we have
\begin{equation}\label{sasympt}
 s = 2\eps +O(\eps^2).
\end{equation}
Note that as $\eps\to 0$ we have
\begin{equation}\label{lllas}
 \log(1/\las)\sim \eps \sim\log\la.
\end{equation}

The overall plan of the proof is as for the cases $\la$ constant and $\la\to\infty$.
We shall treat the second phase (regular growth) in \refSS{ss_meet} and the first phase, approximation by the branching process, in \refSS{ss_bptog}. To  be able to carry out the third phase, we still need to study
the distribution of the time the branching process takes to reach a large size.
We do this in \refSS{ss_bpslow},
and prove various other branching process lemmas we shall need in \refSS{ss_bp2}.
In \refSS{ss_2c} we consider the typical distances in the 2-core.
Finally, armed with all these results, we prove the lower bound on the diameter
in \refSS{ss_low}, and the upper bound in \refSS{ss_up}; this turns out
to be not as easy as one might expect, and
both proofs involve considerable re-examination of the first phase, the early growth of the neighbourhoods.

One complication
concerns the wedge condition used in Section~\ref{sec_const};
here this turns out to have probability $\Theta(\eps^3)$,
or $\Theta(\eps^2)$ if we condition on the vertex
being in the giant component. In Section~\ref{sec_const},
we used a much stronger `diamond' condition, that allowed
us to simply avoid dependence between the neighbourhoods
of the vertices we considered. Unfortunately, the diamond condition
corresponds roughly to two wedge conditions, and has
probability $\Theta(\eps^4)$ after conditioning
on being in the giant component. When $\eps\to 0$,
we cannot afford to give up a factor $\eps^2$ 
in the number of vertices we consider to develop neighbourhoods from in the third phase.

Except that \refSSs{ss_bpslow} and \ref{ss_bp2} belong together,
\refSSs{ss_meet} to \ref{ss_bp2} may be read in any order. We have
chosen the present order as the first two subsections are relatively
simple, and may be seen as motivating the extensive branching process analysis that follows.

Throughout we write $\La$ for $\eps^3 n$, and assume
that $\La\to\infty$.  
In particular, we allow ourselves
to assume that $\La$ is `sufficiently large' (i.e., larger than some implicit
constant) whenever this is convenient.
As noted in the introduction, in proving Theorem~\ref{th_eps} we may assume
that $\eps\to 0$; correspondingly, we shall assume without comment
that $\eps$ is `sufficiently small' whenever convenient.

In what follows we shall use standard results about the component
structure of $G(n,p)$ just above the phase transition; let us
recall these here. We write $C_i(G)$ for the number of vertices in the $i$th largest
component of a graph $G$.

\begin{theorem}\label{th_Gnp}
Let $\la=1+\eps$, where $\eps=\eps(n)>0$ satisfies $\eps\to 0$
and $\La=\eps^3 n\to\infty$, and let $s=s(\la)$ denote the survival
probability of $\bp_\la$.
Then
\begin{equation}\label{eq_above}
 C_1(G(n,\la/n)) = s n + \Op(\eps n/\sqrt{\La}),
\end{equation}
and
\[
 C_2(G(n,\la/n)) = \delta^{-1}\left(\log\La -\frac{5\log\log\La}{2} +\Op(1)\right),
\]
where
\[
 \delta = \la-1-\log\la =\eps^2/2 -\eps^3/3+O(\eps^4).
\]
\noproof
\end{theorem}
This result, which extends results
of Bollob\'as~\cite{BB_evol,BB:RG1} by removing
a logarithmic lower bound on $\La$ from the conditions,
is essentially due to \L uczak~\cite{Luczak_near}. Note, however,
that the actual formula for $C_2$ given in~\cite{Luczak_near}
is incorrect; see the discussion in Bollob\'as and Riordan~\cite[Section 3.4]{rgpb},
where a proof of the above result based on branching processes is given.

Formally, by the {\em giant component} of $G=G(n,\la/n)$ we mean the component
$C_1$ with the most vertices (chosen according to any rule if there is a tie).
Recalling from~\eqn{sasympt} that $s\sim 2\eps$, under the conditions of Theorem~\ref{th_Gnp}
we have
\begin{equation}\label{C1}
 |C_1|=(2+\littleop(1))\eps n.
\end{equation}

\subsection{Large neighbourhoods and meeting in the middle}\label{ss_meet}

In this subsection we show that whp once the neighbourhoods of a
vertex become large, they grow at the expected rate until reaching
size $\sqrt{\eps n}\log \La$, say. Showing this is not quite as simple as proving Lemma~\ref{reg},
since when $\eps$ is small, even when the neighbourhoods are fairly large,
the expected increase in size from one step to the next may still be 
smaller than the standard deviation. Hence it may well happen that $\Gamma_t(x)$
is smaller than $\Gamma_{t-1}(x)$ for some $t$. However, this is unlikely
to happen for many consecutive $t$.

We shall start by proving a corresponding growth result for a
Galton--Watson branching process. It may well be that a similar
result exists in the literature, but we have not found it;
the key point is the dependence of the bounds on the parameters
of the branching process. The general theme here and throughout
this section is that the behaviour of the branching process
is only `regular' once it reaches sizes larger than $1/\eps$,
and that it is best seen on time scales on the order of $1/\eps$, the typical time required for a constant factor change in the size of a generation.

Given parameters $\mu=1+\eps$ and $n$,
consider a Galton--Watson branching process $(Z_t)_{t\ge 0}$
starting with a fixed number $N_0$ of particles, 
in which each particle has a binomial number
of children in the next generation, with parameters $n$ and $\mu/n$.
Let $N_t=|Z_t|$ denote the number of particles in generation $t$.
\begin{lemma}\label{bpgrow}
Let $0<\eps,\delta<1$ and $n$ be given, and define $(N_t)$ as above, with $\mu=1+\eps$.
Writing $\omega$ for $\eps N_0$, the probability that 
\begin{equation}\label{abc}
 \left(1-\delta\right) N_0\mu^t \le N_t \le
  \left(1+\delta\right) N_0\mu^t
\end{equation}
holds for all $t\ge 0$ is at least $1-O(e^{-c_0\delta^2\omega})$,
where $c_0>0$ is an absolute constant, and the implicit constant in
the $O(\cdot)$ notation is absolute.
\end{lemma}
\begin{proof}
We may and shall assume that $\delta^2\omega\ge 100$, say;
otherwise, there is nothing to prove.

We may construct $(Z_t)$ in small steps in the following
standard way: let $A_1,A_2,\ldots$
be independent binomial $\Bi(n,\mu/n)$ random variables.
As we construct the process, we number the particles in order
of the time they are born; we start by numbering the particles
of $Z_0$ with $1,2,\ldots,N_0$ in any order. To define $(Z_t)$,
simply take $A_i$ to be the number of children of the $i$th particle.
Writing $S_t$ for $\sum_{t'< t} N_{t'}$, we then have
\begin{equation}\label{Nt}
 N_t = N_0 + \sum_{i\le S_t} (A_i-1) = N_0+B_{S_t}-S_t,
\end{equation}
where $B_i=\sum_{j\le i}A_j$.

For $t\ge -1/\eps$ set
\[
 \delta_t = \frac{\eps\delta}{8}\int_{s=-1/\eps}^t \mu^{-s/4} \dd s,
\]
and set $\delta_t=0$ if $t<-1/\eps$.
Note that $\delta_t$ is an increasing function of $t$, with
\[
 0\le \delta_t \le \frac{\eps\delta}{8} \frac{4}{\log\mu} \mu^{1/(4\eps)} \le
  \delta,
\]
using $(1+\eps)^{1/(4\eps)}<e^{\eps/(4\eps)}=e^{1/4}$ and $\eps/\log\mu = \eps/\log(1+\eps)\le
1/\log 2$.

The key property of $\delta_t$ is that if $t\ge0$ and $r=t-1/\eps$,
then
\begin{equation}\label{tr}
 \delta_t - \delta_r \ge (t-r) \frac{\eps\delta}{8} \mu^{-t/4}
 = \delta \mu^{-t/4}/8.
\end{equation}
For $t\ge 1$, let $E_t$ be the event that $N_t> (1+\delta_t)N_0\mu^t$ holds
but $N_s\le (1+\delta_s)N_0\mu^s$ for all $0\le s<t$. 
Suppose that the upper bound in \eqref{abc} fails for some $t$.
Then $N_t>(1+\delta_t)N_0\mu^t$ for this $t$, and it follows that
one of the events $E_t$ holds. 

Suppose that $E_t$ holds for some $t\ge 0$.
Set $r=t-1/\eps$, and, for convenience, set
$N_s=\mu^sN_0$ for all negative integers $s$, so $\sum_{s<0}N_s=N_0/\eps$.
Then, with all sums starting at $-\infty$ unless otherwise indicated,
\begin{eqnarray*}
 S_t+N_0/\eps &=& \sum_{s<t} N_s \le \sum_{s<t} (1+\delta_s)\mu^s N_0  \\
 &\le& \sum_{s<r} (1+\delta_r) \mu^s N_0 + \sum_{r\le s<t} (1+\delta_t) \mu^s N_0\\
 &=& \sum_{s<t} (1+\delta_t)\mu^s N_0 -
  \sum_{s<r}(\delta_t-\delta_r) \mu^s N_0 \\
 &=& \frac{N_0}{\eps}\left( (1+\delta_t)\mu^t - (\delta_t-\delta_r)\mu^{\ceil{r}}\right) \\ 
 &<& \frac{N_0}{\eps}\left( (1+\delta_t)\mu^t - (\delta_t-\delta_r)\frac{\mu^t}{4} \right),
 \end{eqnarray*}
since $\mu^{t-\ceil{r}}=(1+\eps)^{\floor{1/\eps}}\le (1+\eps)^{1/\eps}<e<4$.

For each fixed $i$, let $f(i)=(1+\eps)i$ denote the expectation of $B_i=\sum_{j=1}^i A_j$.
{From} the above, we have
\[
 f(S_t)-S_t +N_0 =\eps S_t +N_0= \eps (S_t+N_0/\eps) \le  N_0(1+\delta_t) \mu^t
 -N_0 (\delta_t-\delta_r) \mu^t/4.
\]
On the other hand, since $E_t$ holds we have
$N_t> (1+\delta_t) \mu^t N_0$, so from \eqref{Nt} it follows that
\[
 B_{S_t} -S_t+N_0 = N_t > (1+\delta_t)\mu^t N_0.
\]
Combining the two equations above, using \eqref{tr}, and recalling that $N_0=\omega/\eps$,
we see that
\begin{equation}\label{dev}
 B_{S_t} - f(S_t) >  N_0 (\delta_t-\delta_r) \mu^t/4 \ge 
 N_0 (\delta \mu^{-t/4}/8) \mu^t/4 = \delta\omega\eps^{-1}\mu^{3t/4}/32.
\end{equation}
On the other hand, from the bound on $S_t+N_0/\eps$ above we have, very crudely,
\begin{equation}\label{stup}
 S_t \le \frac{N_0}{\eps} (1+\delta_t) \mu^t \le 2\omega\eps^{-2}\mu^t.
\end{equation}
{From} \eqref{dev} and \eqref{stup} it follows that $B_{S_t}-f(S_t)\ge g(S_t)$, where
\[
 g(i) = \max\left\{\delta\omega\eps^{-1}/32,i^{3/4}\delta\omega^{1/4}\eps^{1/2}/60\right\}.
\]
Let $F_i$ be the event that $B_i-f(i) = B_i-\E B_i\ge g(i)$.
We have shown that if one of the events $E_t$ holds, then so does
one of the events $F_i$. At this point we could simply bound
the probability of the union of the $F_i$ by the sum of their probabilities,
but as they are highly dependent, this is rather inefficient.

Let $T=\ceil{\omega/\eps^2}$, noting that $T\ge 100$ (when $\eps$ is
small enough)
and $T< 2\omega/\eps^2$.
For $k=0,1,2,\ldots$, let $G_k$ be the event $\bigcup_{kT<i\le (k+1)T} F_i$,
so
\[
 \Pr\Bigl(\bigcup_{t\ge 1} E_t\Bigr) \le  \Pr\Bigl(\bigcup_{i\ge 1} F_i\Bigr) = 
 \Pr\Bigl(\bigcup_{k\ge 0} G_k\Bigr) \le \sum_{k=0}^\infty \Pr(G_k).
\]
Finally, let $G_k'$ be the event that $B_{(k+2)T}-\E B_{(k+2)T}\ge g(kT)$.
Let us estimate $\Pr(G_k'\mid G_k)$. We test whether $G_k$ holds by examining
each $B_i$ in turn, stopping at the first $i>kT$ for which $F_i$ holds.
Suppose $G_k$ does hold, and that we stop at $i=i'$, so
$kT<i'\le (k+1)T$. Recalling that $B_i=\sum_{j\le i}A_j$, where the $A_j$ are independent
with distribution $\Bi(n,\mu/n)$, we have not yet examined any $A_j$, $j>i'$.
Hence the conditional distribution of $\Delta =B_{(k+2)T}-B_{i'}=\sum_{i'<j\le (k+2)T}A_j$
is just its unconditional distribution, which is binomial with mean $(1+\eps)((k+2)T-i')\ge T\ge 100$.
It is easy to check (for example from the Berry--Ess{\'e}en Theorem) that 
this binomial distribution is well approximated by a normal distribution, and in particular, $\Delta$ exceeds its mean 
with probability at least $1/3$.
But when this happens,
\[
 B_{(k+2)T}-\E B_{(k+2)T} = B_{i'}-\E B_{i'} + \Delta-\E\Delta \ge B_{i'}-\E B_{i'} \ge g(i') \ge g(kT),
\]
since we are assuming $F_{i'}$ holds, and $g(\cdot)$ is non-decreasing.
Thus, given $G_k$, the event $G_k'$ holds with probability at least $1/3$.
Hence $\Pr(G_k')\ge \Pr(G_k)/3$, so $\Pr(G_k)\le 3\Pr(G_k')$.

Now $G_0'$ is the event that $B_{2T}$, a variable with binomial distribution with mean
$\mu_0=(1+\eps)2T \le 4T\le 8\omega\eps^{-2}$,
exceeds its mean by at least $x_0=\delta\omega\eps^{-1}/32$.
Since $x_0\le \mu_0$, Lemma~\ref{Chern} applies, and we see that
$\Pr(G_0')\le 2\exp(-x_0^2/(3\mu_0))\le 2\exp(-\delta^2\omega/24576)$.

For $k\ge 1$, $G_k'$ is the event that $B_{(k+2)T}$, which has a binomial
distribution with mean $\mu_k=(1+\eps)(k+2)T\le 12k\omega\eps^{-2}$,
exceeds
its mean by $x_k=g(kT)\ge g(k\omega\eps^{-2})\ge k^{3/4}\omega\eps^{-1}\delta/60$.
Since $x_k\le \mu_k$, by Lemma~\ref{Chern} we have $\Pr(G_k')\le 2\exp(-c_0k^{1/2}\delta^2\omega)$
for some absolute constant $c_0>0$.
Hence, reducing $c_0$ if necessary,
\[
 \Pr\Bigl(\bigcup_t E_t\Bigr) \le 3\sum_k\Pr(G_k') \le 2e^{-c_0\delta^2\omega} 
  + \sum_{k\ge 1} 2e^{-c_0k^{1/2}\delta^2\omega} = O\bb{e^{-c_0\delta^2\omega}},
\]
recalling that $\delta^2\omega\ge 100$.
As noted above, if the upper bound in \eqref{abc} fails,
then some $E_t$ holds, so we have proved that the upper
bound holds with the required probability.

The argument for the lower bound is almost identical. 
Let $E_t'$ be the event that $N_t< (1-\delta_t)N_0\mu^t$ holds
but $N_s\ge (1-\delta_s)N_0\mu^s$ for all $s<t$.
Changing signs in the argument above, we see that if $E_t'$ holds
then the equivalent of \eqref{dev} holds, namely
\begin{equation}\label{devn}
 B_{S_t} - f(S_t) \le - \delta\omega\eps^{-1}\mu^{3t/4}/32.
\end{equation}
The proof of~\eqref{stup} only assumed $N_s\le (1+\delta_s)N_0\mu^s$ for $s<t$,
which we now know to be true with the required probability.
If \eqref{stup} does hold, then \eqref{devn}
implies that $B_i -f(i)\le -g(i)$ holds for some $i$.
We may bound the probability of this event just as for $F_i$ above,
completing the proof.
\end{proof}

Turning to the graph, Lemma~\ref{bpgrow} enables us to prove
the required growth result.
Our
choice of the parameters here is somewhat arbitrary, but will be useful later. Recall that $V(G)$ denotes the vertex set of a graph $G$, and $\Gamma_r(x)$ the set of
vertices at graph distance $r$ from a vertex $x$.
\begin{lemma}\label{ggrow}
Let $\eps=\eps(n)\le 1$ satisfy $\La=\eps^3 n\to\infty$.
Set $\la=1+\eps$, $\omega=\La^{1/6}$, and
$M=\sqrt{\omega \eps n}$.
For $x\in V(G(n,\la/n))$
and $r\ge 0$, let $E_{x,r}$ be the event that
\[
 (1-2\omega^{-1/3})\la^t|\Gamma_r(x)|
 \le |\Gamma_{r+t}(x)|
 \le   (1+\omega^{-1/3})\la^t|\Gamma_r(x)|
\]
holds for $0\le t\le T=\log(\eps M/\omega)/\log \la$.
Then, for some absolute constant $c_0$,
\[
 \Pr\bb{ E_{x,r} \,\big|\, |\Ga_0(x)|,\ldots,|\Ga_r(x)|} \ge 1-O\bb{\exp(-c_0\omega^{1/3})}
 = 1-o(\La^{-100})
\]
whenever $\omega/\eps \le |\Ga_r(x)|\le 2\omega/\eps$ and
$\sum_{r'\le r}|\Ga_r(x)|\le n^{2/3}$.
\end{lemma}
In other words, once we reach size $\omega/\eps$ in the neighbourhood
exploration, provided we have not so far used up too many vertices,
the neighbourhoods grow at the expected rate until they reach
size approximately $M$. Note that if $\La=\eps^3n\ge (\log n)^{20}$, then the error term in the form
$O(\exp(-c_0\omega^{1/3}))$ is $o(n^{-100})$, i.e., utterly negligible.

\begin{proof}
Condition on the result of the exploration up to step $r$, assuming
that we find between $\omega/\eps$ and $2\omega/\eps$
vertices in the last generation
and have seen at most $n^{2/3}$ vertices so far.
Let $N_t'=|\Ga_{r+t}(x)|$. The (conditional) distribution
of the process $(N_t')_{t\ge 0}$ is very similar to that of $(N_t)$:
the only difference is that each vertex gives rise to a binomial
$\Bi(m,\la/n)$ number of children in the next generation,
where $m$ is the number of vertices not seen so far.

For the upper bound on the neighbourhood sizes, we simply note
that $m\le n$, so $(N_t')$ is stochastically dominated by $(N_t)$.
The result thus follows immediately from Lemma~\ref{bpgrow}.

For the lower bound, set $n'=n-2n^{2/3}$. Note that if the
upper bound holds, which it does with probability $1-O(\exp(-c_0\omega^{1/3}))$,
then by time $T$ we have used at most $n^{2/3}+10M/\eps\le 2n^{2/3}$
vertices, so we still have at least $n'$ left.
For times $t$ by which we have used up at most $2n^{2/3}$ vertices, the process
$(N_t')$ stochastically dominates a process $(N_t'')$ in which each
particle has $\Bi(n',\la/n)$ children. This binomial
has mean $\mu=\la n'/n=(1+\eps)(1-2n^{-1/3})$. Applying Lemma~\ref{bpgrow} again,
it follows that with probability $1-O(\exp(-c_0\omega^{1/3}))$ we have
$
 |\Gamma_{r+t}(x)|\ge (1-\omega^{-1/3})\mu^t|\Gamma_r(x)|
$
for $0\le t\le T$.
Since $T=\log(\eps^{3/2}n^{1/2}\omega^{-1/2})/\log\la \le \log(\La^{1/2})/(\eps/2) = \eps^{-1}\log\La
\le n^{1/3}/\omega$, we have
$\mu^t/\la^t \ge 1-3/\omega$ for $t\le T$, so the lower bound
follows.
\end{proof}

\begin{remark}\label{Rgt}
Let us note that, while the various constants can certainly be improved, 
Lemmas~\ref{bpgrow} and~\ref{ggrow} are tight in several ways. Firstly, since the survival
probability of the branching process $\bp_\la=(X_t)$ is of order $\eps$,
if we start from a neighbourhood $\Ga_r(x)$ of size $a/\eps$,
the neighbourhood exploration process will die quickly with
probability $e^{-\Theta(a)}$. Hence, in order to make it very likely
that the neighbourhoods grow at the right rate, we certainly
need $|\Ga_r(x)|$ to be much larger than $1/\eps$.
In other words, neighbourhoods are only `large' over size $\omega/\eps$,
for some $\omega\to\infty$.
Similarly, it can be seen that the form $\exp(-\Omega(\delta^2\omega))$
of the error bound in Lemma~\ref{bpgrow} is best possible.

Finally, when $\eps$ is close to the lower end of the range we
consider, we cannot extend Lemma~\ref{ggrow} to growth much
beyond size $\sqrt{\eps n}$; shortly beyond this point,
the number of vertices `used up' --- which is larger than $\sqrt{n/\eps}=\sqrt{\eps n}/\eps$ since about $1/\eps$ generations are roughly the same size $\sqrt{\eps n}$ --- is sufficient to slow the growth
appreciably. Fortunately, neighbourhoods of two different vertices
are likely to join up when they have size around $\sqrt{\eps n}$,
as we shall now see. The basic explanation for this is that the probability of the $\sqrt{n/\eps}$ vertices seen near one vertex being distinct from the $\sqrt{n\eps}$ at a given distance from the other vertex becomes small. (It is misleading to consider separately each of the  $\sqrt{n/\eps}$ vertices `near' one vertex being distinct from  $\sqrt{n/\eps}$ vertices `near' the other, since these events do not have the required independence.) The fact that neighbourhoods typically join up when they
each have size $\sqrt{\eps n}$ explains one factor of $\eps$ 
in the first log in \eqref{deps}. 

\end{remark}
 
For $x\in V(G)$ and $a>0$, let $t_a(x)$ denote the smallest
$r$ for which $|\Ga_r(x)|\ge a$, if such an $r$ exists;
otherwise $t_a(x)$ is undefined.
The following simple lemma captures
the observation that, for large $a$, we are unlikely to `overshoot' our
cutoff $a$ by too much.
\begin{lemma}\label{nojg}
Let $\la=1+\eps$ and fix an integer $a$, a vertex $x$ and $\delta>0$. Given that $t_a(x)$ is defined, the probability that $|\Ga_{t_a(x)}(x)|$ exceeds $(1+\delta)(1+\eps)a$
is $e^{-\Omega(\delta^2 a)}$, where the implicit constant is absolute.
\end{lemma}
\begin{proof}
The event that $t_a(x)$ is defined may be written as a disjoint union of events of the form
$E=\{t_a(x)=t,\, |\Ga_{\le t-1}(x)|=m,\, |\Ga_{t-1}(x)|=s\}$, where $0<s<a$.
Let us condition on one such event. Given that $|\Ga_{\le t-1}(x)|=m$ and $|\Ga_{t-1}(x)|=s$,
the distribution of $|\Ga_t(x)|$ is binomial with parameters $n-m$ and $1-(1-\la/n)^s\le s\la/n$.
Hence, given $E$, the conditional distribution of $|\Ga_t(x)|$ is that of a binomial distribution
with mean at most $s\la=(1+\eps)s<(1+\eps)a$ conditioned to be at least $a$. It
is easy to check that the probability that such a distribution exceeds $(1+\delta)(1+\eps)a$
is maximal when $s$ is maximal, and is then (from the Chernoff bounds) of the form $e^{-\Omega(\delta^2 a)}$.
\end{proof}

We now turn to the time neighbourhoods take to meet having reached some `reasonably large' size.

\begin{lemma}\label{meet}
Let $\eps=\eps(n)$ and $\la=1+\eps$ be such that $\eps\to 0$
and $\La=\eps^3 n\to\infty$.
Set $\omega=\La^{1/6}$, and $t_2=\log(\eps^3 n/\omega^2)/\log \la$.
Let $x$ and $y$ be two vertices of $G(n,\la/n)$.
Writing $E$ for the event that
$t_{\omega/\eps}(x)=r_1$, $t_{\omega/\eps}(y)=r_2$, and 
the graphs $\Gl{r_1}(x)$ and $\Gl{r_2}(y)$ each contain
at most $n^{2/3}$ vertices and are disjoint, we have
\begin{equation}\label{jb}
 \Pr\bb{d(x,y)\ge r_1+r_2+t_2+a\bigm| E} = e^{-(1+o(1))\la^a} +O(e^{-c_0\omega^{1/3}})
 =  e^{-(1+o(1))\la^a} +o(\La^{-10})
\end{equation}
for any function $a=a(n)\ge -t_2/2$, and
\[
 \Pr\bb{d(x,y)\le r_1+r_2+t_2-K\bigm| E} =o(1)
\]
whenever $K=K(n)$ is such that $\eps K\to\infty$.
\end{lemma}
\begin{proof}
It suffices to prove the first statement: since $\log\la=\Theta(\eps)$,
if $\eps K\to\infty$ then $K\log\la \to\infty$,
so $\la^{-K}\to 0$, and the second statement follows
immediately from the first.
In proving \eqref{jb}, we may assume that $a \le \amax=\log\omega/(2\log \la)$:
otherwise, $\la^a\ge \omega^{1/2}$, and the additive error term in \eqref{jb},
which is independent of $a$,
dominates the main term. 

We explore the neighbourhoods of $x$ and $y$ in the usual way,
initially stopping each exploration when we first reach a
neighbourhood of size greater than $\omega/\eps$.
At this point, the conditions of the theorem allow us to
assume that we have used up at most $n^{2/3}$ vertices in each exploration,
and that the explorations, having taken $r_1$ steps from $x$ and $r_2$ steps from $y$, have not met.
By Lemma~\ref{nojg}, with conditional probability at least $1-e^{-\Omega(\omega)}$ the last generation in each exploration
has size at most $(1+\sqrt{\eps})\omega/\eps\sim \omega/\eps$.

We now continue both explorations. At the start of step $i$, $i=0,1,2,\ldots$,
we have explored the neighbourhoods of $x$ out to distance $r_1+\ceil{i/2}$
and those of $y$ out to distance $r_2+\floor{i/2}$.
During step $i$, we first test whether any of the
$|\Ga_{r_1+\ceil{i/2}}(x)||\Ga_{r_2+\floor{i/2}}(y)|$ `cross-edges' between
these two neighbourhoods is present. If so, $d(x,y)=r_1+r_2+i+1$,
and we stop. Otherwise, we uncover the next neighbourhood of either $x$ or $y$
as appropriate and continue, stopping if we have found no cross-edge by step
$t_2+a$, in which case $d(x,y)>r_1+r_2+t_2+a$.

After $t_2+\amax$ steps as above, each extending either $x$'s or $y$'s neighbourhood, the
typical size of the neighbourhood of $x$ or $y$ reached is
$(\omega/\eps) \la^{(t_2+\amax)/2} = \omega^{1/4}\sqrt{\eps n}$.
In particular, this size is much less than the quantity $M$ defined
in Lemma~\ref{ggrow}. Hence, by Lemma~\ref{ggrow}, we may assume
that
\[
 |\Ga_{r_k+j}(x_k)| \sim \la^j|\Ga_{r_k}(x_k)| \sim \la^j\omega/\eps
\]
for $k=1,2$ and all $j\le (t_2+\amax)/2$, where $x_1=x$ and $x_2=y$.
Furthermore, the error terms, which are factors
of the form $(1+O(\sqrt{\eps})+O(\omega^{-1/3}))$,
are uniform in $j$.

It follows that at step $i$ we test $(1+o(1)) \la^i(\omega/\eps)^2$ potential
cross-edges, and that by any step $i\ge t_2/2$ we have tested in total
\[
 (1+o(1))\sum_{j=0}^i \la^j(\omega/\eps)^2 \sim (\omega/\eps)^2\sum_{j=-\infty}^i \la^j
 \sim (\omega/\eps)^2 \eps^{-1} \la^i
\]
potential cross-edges. (The bound $i\ge t_2/2$ is used for convenience only,
to allow us to approximate the sum from $j=0$ by the sum from $j=-\infty$.)

Since each cross-edge tested is present with its original unconditional probability
of $\la/n\sim 1/n$,
it follows that up to a $O(e^{-c_0\omega^{1/3}})$ error term (from the conclusion of Lemma~\ref{ggrow}
not holding, etc), the probability that the explorations do not meet by step $t_2+a$
is
\[
 p_{\ge a} = (1-\la/n)^{(1+o(1))\omega^2\eps^{-3}\la^{t_2+a}}.
\]
Since
\[
 \log(1/p_{\ge a})  \sim (1/n) \omega^2\eps^{-3}\la^{t_2}\la^a
 = \la^a,
\]
the proof is complete.
\end{proof}

Roughly speaking, Lemma~\ref{meet} tells us that once the neighbourhoods
of two vertices reach a decent size, $\omega/\eps$, then whp these neighbourhoods
then meet within $O(1/\eps)$ steps of `when they should', which is after an extra $t_2$ steps. To study the diameter
of $G(n,\la/n)$,
we shall need the full strength of the bound  actually proved.  Note
that there is variation of order $1/\eps$ in the actual time taken to meet,
as may be seen from \eqref{jb}, where 
any $a=O(1/\eps)$ gives a probability bounded away from $0$ and $1$.

In the light of Lemma~\ref{meet}, as in the case of $\la$ constant or $\la\to\infty$,
the key to understanding the diameter of $G(n,\la/n)$ is understanding
the distribution of the time taken until the neighbourhoods
of a vertex reach a reasonable size (in this case $\omega/\eps$);
this will be our aim in the next few subsections.
We shall take $\omega=\omega(n)=\La^{1/6}$, but there is in fact a wide flexibility in the choice of the function $\omega=\omega(n)$: 
the requirements in what follows are that $\omega$ is at least a certain power
of $\log\La$, and at most a certain power of $\La$. If $\La$ is large
enough to allow $\omega/\log n\to\infty$, then many arguments simplify;
we shall not assume this, however.

In the remainder of this section we explain why Lemma~\ref{meet} already gives us the {\em typical} distance
between vertices in the giant component,
if $\La=\eps^3n$ is at least $(\log n)^{20}$, say.
Indeed, the neighbourhoods of a random 
vertex of $G=G(n,\la/n)$ behave much like the branching process $\bp_\la=(X_t)_{t\ge 0}$,
at least to start with. Roughly speaking, a vertex is in the giant component
if and only if the corresponding branching process survives, which it
does with probability $s\sim 2\eps$. So we will be interested in 
the expected size  of $|X_t|$  conditioned on the process surviving.
\begin{lemma}\label{l:X}
Let $\surv$ be the event that $\bp_\la$ survives. Then
\[
 \E(|X_t| \mid \surv) = \frac{\la^t-(1-s)\las^t}{s},
\]
which is asymptotically $\la^t/s\sim \la^t/(2\eps)$ if $\eps\to0$ and $\eps t\to\infty$.
\end{lemma}
\begin{proof}
Writing $\ind{A}$ for the indicator function of an event $A$, we have
\[
 \E(|X_t|\ind{\surv}) = \E(|X_t|) - \E(|X_t|\ind{\surv^\cc})
 = \la^t - (1-s)\E(|X_t|\mid \surv^\cc) = \la^t-(1-s)\las^t,
\]
since the distribution of $\bp_\la$ conditioned on $\surv^\cc$ is that of $\bp_{\las}$.
The result follows.
\end{proof}
  
It is not hard to see that the `typical' size
of $|X_t|$ given $\surv$
is of the same order as the expected size; we shall give some precise results on this later.
Hence, for most vertices in the giant component, their neighbourhoods
take time $\log\omega/\log\la$ to reach size $\omega/\eps$,
so the typical distance is $2\log\omega/\log\la +t_2=\log(\eps^3 n)/\log\la$.
More precisely, one can check that the distance between two random vertices
of the giant component is $\log(\eps^3 n)/\log\la +\Op(1/\eps)$; we shall
not give the details. The rest of the proof of Theorem~\ref{th_eps} essentially shows that the other term in the formula in that theorem accounts for vertices whose neighbourhoods take an abnormally long time to start growing large.

\subsection{Branching process to graph}\label{ss_bptog}

At some point, we need to compare the probabilities of events defined
in terms of our random graph $G=G(n,\la/n)$ with events in the branching
process. It turns out that we have to consider events involving
trees of height $\Theta(\log\La/\log \la)=\Theta(\eps^{-1}\log\La)$,
recalling that $\La=\eps^3 n$,
with (it will turn out) up to around $1/\eps$ vertices at each distance from the root.
For this reason, we need to consider trees with at least $\Theta(\eps^{-2})$
vertices. If $\eps$ is smaller than $n^{-1/4}$, then we cannot simply 
extend Lemma~\ref{cpl} to cover such trees using the same proof,
since the error terms $|T|^2/n$ would be too large.

Fortunately, it is easy to prove a result that applies for the trees 
we need. Although this is in some sense a coupling result, the obvious
coupling between $\Gatm(x)$ and $\bp_\la$ fails here. This obvious coupling
is based on the fact that a $\Po(\la)$ and a $\Po(\la(1-\delta))$ distribution
can naturally be coupled to agree with probability at least $1-\la\delta$. In fact,
much better couplings are possible.
Recall that $X_{\le t}$ denotes the first $t$ generations of $\bp_\la$, seen
as a rooted tree.

\begin{lemma}\label{spcpl}
Let $\la=1+\eps$, where $\eps=\eps(n)=O(1)$.
Let $\delta(n)$ be any function with $\delta>0$ and $\delta\to 0$ as $n\to\infty$.
Let $t=t(n)\ge 0$ and let $T=T(n)$ be a rooted tree of height $t$ with $\eps|T|^2\le \delta n$, each generation of size at most $n^{1/3}$,
and $|T|\le \delta n^{2/3}$. Then
\[
 \Pr\bb{\Gatm(x)\isom T} \sim \Pr\bb{X_{\le t}\isom T}
\]
and
\[
 \Pr\bb{\Gat(x)\isom T} \sim \Pr\bb{X_{\le t}\isom T},
\]
where the asymptotics is uniform over all such sequences $T(n)$.
\end{lemma}
\begin{proof}
Rather than couple, we simply calculate directly; it is convenient to order
the vertices first. When constructing $\bp_\la$ starting from $X_0$,
let us number the particles $1,2,3,\ldots$ in the order they appear, so the initial particle
is particle $1$, and test particles in numerical order to see how many children they have.
We number the vertices uncovered in the neighbourhood exploration process
by which we find $\Gatm(x)$ analogously, this time using any (deterministic or random) rule
to decide in which order to number the children of a vertex.

For each numbering $\Ts$ of $T$ that can arise in such an exploration,
let $E_1(\Ts)$ be the event that $X_{\le t}$ is isomorphic to
$\Ts$ with the labels matching. Then $\{X_{\le T}\isom T\}$ is the disjoint union
of the events $E_1(\Ts)$, where $\Ts$ runs over all numberings of $T$; note that these events are equiprobable. Similarly, let $E_2(\Ts)$
be the event that $\Gatm(x)\isom \Ts$ with labels matching, so $\{\Gatm(x)\isom T\}$
is the disjoint union of the $E_2(\Ts)$. Fix one particular numbering $\Ts$.
Since the probabilities of $E_1(\Ts)$ and
$E_2(\Ts)$ do not depend on the numbering, it suffices to show that
$\Pr(E_1(\Ts))\sim \Pr(E_2(\Ts))$.

Let $r$ be the number of vertices of $T$
at distance $t$ from the root, and $m=|T|$ the total number of vertices.
For $1\le i\le m-r$, let $d_i$ denote the number
of children in $T$ of the $i$th vertex.
Now $E_1(\Ts)$ is simply the event that for $i=1,\ldots,m-r$, the $i$th particle
of the branching process has exactly $d_i$ children. Thus,
\[
 \Pr(E_1(\Ts)) = \prod_{i=1}^{m-r} \frac{\la^{d_i}}{d_i!}e^{-\la}.
\]
Similarly, $E_2(\Ts)$ is the event that for every $i$, when exploring the neighbours
of the $i$th vertex reached, we find exactly $d_i$ new neighbours.
Let $u_i=1+\sum_{j<i} d_j$ denote the number of vertices already `used' (reached)
at the point that we look for new neighbours of the $i$th vertex.
Then 
\[
 \Pr(E_2(\Ts)) = \prod_{i=1}^{m-r} \Pr\bb{\Bi(n-u_i,\la/n)=d_i}
 = \prod_{i=1}^{m-r} \frac{(n-u_i)_{(d_i)}}{d_i!}(\la/n)^{d_i}(1-\la/n)^{n-u_i-d_i}.
\]
Hence,
\[
 \rho = \frac{\Pr(E_2(\Ts))}{\Pr(E_1(\Ts))} =
  \prod_{i=1}^{m-r} \frac{(n-u_i)_{(d_i)}}{n^{d_i}} \frac{(1-\la/n)^{n-u_i-d_i}}{e^{-\la}}.
\]
Since $n-u_i\ge n/2$ and $d_i$ is bounded by $n^{1/3}$, we have
$n-u_i-j=(n-u_i)e^{O(n^{-2/3})}$ for $0\le j\le d_i$,
so
\begin{eqnarray*}
 \log\left(\frac{(n-u_i)_{(d_i)}}{n^{d_i}}\right) 
 &=& \log\left(\frac{(n-u_i)^{d_i}}{n^{d_i}}\right) +O(n^{-2/3}d_i) \\
 &=&  -\frac{u_id_i}{n} +O(d_i(u_i/n)^2) + O(n^{-2/3}d_i)
 =  -\frac{u_id_i}{n} + O(n^{-2/3}d_i),
\end{eqnarray*}
using $u_i\le |T| \le n^{2/3}$ in the last step.
Also,
\begin{eqnarray*}
 (n-u_i-d_i)\log(1-\la/n) &=& (n-u_i-d_i)(-\la/n +O(1/n^2)) \\ 
 &=&  -\la +u_i\la/n +O(d_i/n)+O(n^{-1}).
\end{eqnarray*}
Hence,
\[
 \log\rho = \sum_{i=1}^{m-r} \left(u_i\frac{\la-d_i}{n} +O(n^{-2/3}d_i)+O(n^{-1})\right)
 = o(1)+\sum_{i=1}^{m-r} u_i\frac{\la-d_i}{n},
\]
using $\sum_i d_i=m-1=o(n^{2/3})$.

Now $\la=1+\eps$, and $\sum u_i\le m^2$. By assumption $\eps m^2=o(n)$, so
\[
 \log\rho = o(1)+\sum_{i=1}^{m-r} u_i\frac{1-d_i}{n} = 
 o(1) + \sum_{i=1}^{m-r} u_i\frac{1}{n} -  \sum_{i=1}^{m-r} u_i\frac{d_i}{n}.
\]
We can rewrite the final sum as $\sum_{i=1}^{m-r} \sum_j u_i/n$,
where $j$ runs over the children of $i$.
Each $j$ in the range $2$ up to $m$ appears exactly once
in the double sum,
so the sum is equal to $\sum_{j=2}^m u_{j'}/n$, where $j'$ is the parent
of $j$.
For any vertex $j$, the vertex $j'$ is in the generation before $j$,
so $u_j-u_{j'}$ is at most twice the maximum number of vertices in a generation.
We have assumed this maximum is at most $n^{1/3}$, so $|u_j-u_{j'}|\le 2n^{1/3}$.
Hence,
\begin{eqnarray*}
 \log\rho 
 &=& o(1) + \sum_{i=1}^{m-r}\frac{u_i}{n} - \sum_{i=2}^m \frac{u_{i'}}n \\
&=& o(1)+ \frac{u_1}{n} + \sum_{i=2}^{m-r}\frac{u_i-u_{i'}}{n} - \sum_{i=m-r+1}^m \frac{u_{i'}}{n} \\
 &=& o(1) + o(1) +\sum_i O(n^{-2/3}) -O(rm/n) = o(1)+O(r m /n) = o(1),
\end{eqnarray*}
and the first statement follows.

For the second, it suffices to prove that $\Pr\bb{\Gat(x)\isom T \bigm| \Gatm(x)\isom T}\sim 1$.
But this is immediate since there are at most $2n^{1/3}|T|=o(n)$ extra edges that we must test.
\end{proof}

For any fixed $k$, Lemma~\ref{spcpl} extends to $k$ starting vertices and $k$ trees,
with virtually the same proof.
\begin{lemma}\label{spcplk}
Fix $k\ge 2$. Let $\la=1+\eps$, where $\eps=\eps(n)=O(1)$.
Let $\delta(n)$ be any function with $\delta>0$ and $\delta\to 0$ as $n\to\infty$.
Let $T_1,\ldots,T_k$ be rooted trees, with $\eps|T_i|^2\le \delta n$,
each generation of $T_i$ of size at most $n^{1/3}$, and $|T_i|\le \delta n^{2/3}$.
Given distinct vertices $x_1,\ldots,x_k$ of $G=G(n,\la/n)$, let
$E=E(x_1,\ldots,x_k,T_1,\ldots,T_k)$ denote the event that
$\Gl{t_i}(x_i)\isom T_i$ for $1\le i\le k$, and $d(x_i,x_j)> t_i+t_j$ for $1\le i<j\le k$,
where $t_i$ is the height of $T$.
Then
\[
 \Pr(E)  \sim \prod_{i=1}^k \Pr\bb{X_{\le t_i}\isom T_i},
\]
where the asymptotics is uniform over all choices of $T_1,\ldots,T_k$.\noproof
\end{lemma}
In other words, the event that the $t_i$-neighbourhood of each $x_i$ is isomorphic
to $T_i$, and these neighbourhoods are disjoint, has asymptotically the probability
suggested by independent branching processes. One can prove Lemma~\ref{spcplk} 
by adapting the proof of Lemma~\ref{spcpl} in the obvious ways. Alternatively,
it follows from Lemma~\ref{spcpl} by simple calculations.

\subsection{Slow initial growth: the branching process}\label{ss_bpslow}

In this subsection we study the probability that the branching process $\bp_\la$
survives, but takes much longer than usual to reach generations
of some large size. One might expect the results we need to be in the branching
process literature, and perhaps they are. However, we have not found them.
The key point is that here $\la$ is variable, tending down to $1$ from above,
so results for fixed $\la$ are not of much use. Furthermore, although there
is a natural scaling limit as $\la\to 1$ from above (described below),
results about this limit are not directly applicable either: we
wish to consider events of probability around $1/n$, and this probability
tends to $0$ as $\la\to 1$. In other words, we need explicit bounds on the
rate of convergence of some properties of the branching process as $\la\to 1$.
Fortunately, as is often the case, the branching process results we need
are not hard to prove directly.

With $\la=1+\eps>1$ fixed for the moment, let $\bp_\la=(X_t)_{t\ge 0}$
and $\bp_\la^+=(X_t^+)_{t\ge 0}$ be defined as before,
so $X_t^+\subset X_t$ is the set of particles in $X_t$ which have descendants
in all future generations.
Recall that $\bp_\la^+$ is again a Galton--Watson branching process,
with $|X_0^+|=1$ or $0$ depending on whether $\bp_\la$ survives,
and with offspring distribution $Z_\la$. Here, as before, $Z_\la$
denotes the distribution of a Poisson $\Po(s\la)$ random variable
conditioned to be at least $1$, where $s=s(\la)$ is the probability that
$\bp_\la$ survives forever.

{From} standard results (see Athreya and Ney~\cite{AN}, for example),
we have $|X_t|/\la^t\to Y=Y_\la$ a.s.,
and $|X_t^+|/\la^t\to Y^+=Y_\la^+$ a.s., for some random variables
$Y_\la$ and $Y_\la^+$. Our first (standard, trivial) observation
is that $Y_\la$ and $Y_\la^+$ coincide up to a constant factor.

\begin{lemma}\label{YYp}
We have $Y_\la^+=sY_\la$ a.s.
\end{lemma}
\begin{proof}
Fix $\delta>0$, and let $N=N(\delta)$ be a suitably chosen large integer.
{From} standard results, with probability 1, either $\bp_\la$ dies out,
or there is some minimal $t$ with $|X_t|\ge N$.
Choosing $N$ large enough, at this time $t$ the inequalities 
$||X_t|/\la^t-Y|\le \delta$ and
$||X_t^+|/\la^t-Y^+|\le \delta$ hold with probability at least $1-\delta$.
But $t$ is a stopping time, so given $t$ and $|X_t|$, each particle in $X_t$ survives
independently with probability $s$, and from the Chernoff bounds, provided $N$
was chosen large enough,
the ratio between $|X_t^+|$ and $|X_t|$ is within a factor $1\pm \delta$ of $s$
with probability at least $1-\delta$.

Hence, with probability at least $1-3\delta$ either $Y_\la^+=Y_\la=0$, or
$Y_\la^+ = (1+O(\delta))s Y_\la+O(\delta)$.
The result follows by letting $\delta\to 0$.
\end{proof}

Our ultimate aim is to estimate $\Pr(0<|X_t|<\omega)$ in the range of parameters
where this probability is very small (around $1/(\eps^2n)$, it will turn out).
Essentially, this reduces
to estimating the lower tail of $Y$; in the light of Lemma~\ref{YYp}, we may
study $Y^+$ instead. This turns out to be easier, since $(X_t^+)$ is in
some sense `better behaved' than $(X_t)$ when $\eps\to 0$.

When studying $\bp_\la^+=(X_t^+)$ it makes sense to condition on the event that $X_0^+$ is
non-empty, i.e., that $\bp_\la$ survives. Let us write
$\bpP_\la = (\XP_t)_{t\ge 0}$ for the conditioned process, i.e., a Galton--Watson
process with offspring distribution $Z_\la$ started with a single particle.
Let $\YP$ denote  $\lim_{t\to\infty} |\XP_t|/\la^t$, which exists a.s.
Thus $\YP$ is simply $Y^+$ conditioned on $Y^+>0$, up to a set of measure 0.
By standard results, $\YP$ is a continuous random variable
with strictly positive density on $(0,\infty)$.

It turns out that we will need both upper and lower tail bounds on $\YP=\YP_\la$.
The dependence of these bounds on $\eps=\la-1$ is very important. We start with
the upper tail.

\begin{lemma}\label{etail}
There is an absolute constant $c>0$ such that for any $1<\la<2$ and any $x>0$ we have
$\Pr(\YP_\la>x)\le 2e^{-cx}$.
\end{lemma}
\begin{proof}
Recall that $Z_\la$ denotes a Poisson distribution with mean $s\la$ conditioned to be at least
$1$, where $s=s(\la)$ is the positive solution to $1-s=e^{-\la s}$.
Set
\[
 f_\la(x) = \E(x^{Z_\la}) = \sum_{k\ge 1} x^k \frac{(s\la)^ke^{-s\la}}{k!(1-e^{-s\la})}
 = \frac{(e^{xs\la}-1)e^{-s\la}}{s} = \frac{e^{(x-1)s\la}-e^{-s\la}}{s}.
\]
Note that $f_\la(1)=1$, and, expanding about $x=1$, we have
\[
 f_\la(x) = 1 +\la(x-1)+\frac{s\la^2}{2}(x-1)^2+\cdots = 1+\la(x-1) +O(s\la^2(x-1)^2),
\]
provided $s\la(x-1)$ is bounded.
More precisely, recalling that $s\sim 2\eps$ as $\la\to 1$, and that $\la\le 2$,
it is easy to check that if $0\le x\le 2$, say, then we have
\begin{equation}\label{fexp}
 f_\la(x) \le 1+\la(x-1) + C_1\eps(x-1)^2
\end{equation}
for some absolute constant $C_1$, which we shall take to be at least $1$.

Suppressing the dependence on $\la$ in the notation,
for $t\ge 0$ let $g_t(\theta)=\E(e^{\theta |\XP_t|/\la^t})-1$.
Since $|\XP_0|$ is always $1$, we have $g_0(\theta)=e^{\theta}-1=\theta+O(\theta^2)$
for $\theta$ bounded; in particular,
\begin{equation}\label{m0}
 g_0(\theta) \le \theta + \theta^2
\end{equation}
if $\theta\le 1$.

Given $N=|\XP_1|$,
the conditional distribution of $|\XP_{t+1}|$ is simply the sum of $N$
independent copies of $|\XP_t|$, so
\begin{multline*}
 g_{t+1}(\theta) = \E\bb{ \E(e^{\theta |\XP_{t+1}|/\la^{t+1}} \mid N) } -1
  = \E\bb{ \E(e^{\theta |\XP_t|/\la^{t+1}})^N } -1 \\
 = \E\bb{ (1+g_t(\theta/\la))^N }-1 = f_\la( 1+g_t(\theta/\la) ) -1,
\end{multline*}
since $N=|\XP_1|\sim Z_\la$.
With $\theta$ and $t$ fixed,
set $y_r=g_r(\theta/\la^{t-r})$, so $y_t=g_t(\theta)$,
$y_0=g_0(\theta/\la^t)$, and
$y_{r+1} = f_\la(1+y_r)-1$ for $0\le r\le t-1$.
{F}rom \eqref{fexp}, if $y_r\le 1$, then
\begin{equation}\label{yrat}
 y_{r+1} \le \la y_r + C_1\eps y_r^2 \le \la y_r(1+C_1\eps y_r).
\end{equation}
Suppose $\theta\le 1/(100C_1)\le 1/100$. Then we claim that
\begin{equation}\label{yrih}
 y_r\le 2\theta/\la^{t-r}
\end{equation}
holds for $r=0,1,\ldots,t$.
This is certainly true for $r=0$, since $y_0=g_0(\theta/\la^t)
\le (\theta/\la^t) (1+\theta/\la^t) \le 2\theta/\la^t$.
If \eqref{yrih} holds for $r=0,1,\ldots,s-1$, then in particular
$y_r\le 1$ for $r< s$, so from
\eqref{m0} and \eqref{yrat} we have
\[
 y_s = y_0\prod_{r<s}\frac{y_{r+1}}{y_r} \le \frac{\theta}{\la^t} (1+\theta/\la^t)
 \la^s\prod_{r<s} (1+C_1\eps y_r)
 \le \frac{\theta}{\la^{t-s}} \exp\left(\frac{\theta}{\la^t}+C_1\eps\sum_{r<s} y_r\right).
\]
Using \eqref{yrih} for $r<s$, we have
$\sum_{r<s} y_r \le \sum_{0\le r\le t} 2\theta/\la^{t-r} \le 2\theta\sum_{r\ge0}\la^{-r}
= 2\theta(1+\eps)/\eps$.
Since $\theta\le 1/(100C_1)$, \eqref{yrih} for $r=s$ follows, completing the
proof of \eqref{yrih} by induction.

Setting $r=t$ in \eqref{yrih}, we have in particular that
$y_t=g_t(\theta)\le 1/(50C_1)\le 1$ for $\theta\le \theta_0=1/(100C_1)$.
Hence the moment generating functions $\E(e^{\theta|\XP_t|/\la^t})$
are uniformly bounded by $2$ for all $1<\la\le 2$ and all $\theta\le \theta_0$.

With $\la$ fixed, we have $|\XP_t|/\la^t\to \YP$ a.s. By Fatou's Lemma,
it follows that
\[
 \E\bb{ e^{\theta \YP} } \le \liminf_{t\to\infty} \E\bb{ e^{\theta|\XP_t|/\la_t} }
 = \liminf_{t\to\infty} g_t(\theta) \le 2.
\]
Applying Markov's inequality, it follows that for any $x$ we have
$\Pr(\YP\ge x) \le 2e^{-\theta x}$, completing the proof of the lemma.
\end{proof}

Our main application of the upper tail bound above is to show that the sum
of many independent copies of $\YP$ is tightly concentrated.

\begin{lemma}\label{econc}
Let $c$, $A$ and $\delta$ be positive constants. There is a constant
$\alpha=\alpha(c,A,\delta)>0$
such that, if $Z$ is any random variable satisfying the tail
bound $\Pr(|Z-\E Z|>x)\le Ae^{-cx}$ for all $x>0$, and $S_n$
is the sum of $n$ independent copies of $Z$, then
\begin{equation*}
 \Pr\bb{|S_n/n -\E Z| \ge \delta} \le e^{-\alpha n}
\end{equation*}
for all $n\ge 1$.
\end{lemma}

\begin{proof}
For $|\theta|<c$,
let $\phi(\theta)=\ex(e^{\theta(Z-\mu -\delta)})$ where $\mu=\ex Z$;
the tail bound on $Z$ implies that $\phi(\theta)$ is finite.
Then $\phi(0)=1$,  $\phi'(0)=-\delta$, and 
$$
\phi''(\theta) = \ex\bb{(Z-\mu-\delta)^2e^{\theta(Z-\mu-\delta)}},
$$
which is bounded by a constant due to the tail bound. Hence
there are positive constants $c''$ (which we may take smaller than $c/\delta$) and $c'$ such that
$\phi(\theta)<1-c'\delta^2$ when $\theta=c''\delta$. Now with $Z_1,\ldots,Z_n$ independent
copies of $Z$ and $S_n=\sum Z_i$, we have
\begin{eqnarray*}
\pr(S_n \ge  n(\mu +\delta)) 
&\le&  \ex e^{\theta  S_n} e^{-\theta(\mu +\delta)n}\\
&=&  \ex e^{\sum_i \theta (Z_i  -\mu -\delta) }\\
&=&\phi(\theta)^n\quad\mbox{by independence of the $Z_i$}\\
&<& (1-c'\delta^2)^n < e^{-c'\delta^2n}.
\end{eqnarray*}
An exponential upper bound on $\pr(S_n \le n(\mu -\delta))$ is obtained   by considering $\hat\phi(\theta) =\ex(e^{\theta( \mu -\delta-Z)})$ similarly.
\end{proof}
 
Using Lemma~\ref{econc} it is easy to show that up to an error probability
that is exponentially small in $\omega$, the martingale $|\XP_t|/\la^t$
has essentially converged to its (almost sure) limit $\YP$ by the time 
that $|\XP_t|$ first reaches size $\omega$. As before, it is crucial
that the concentration we obtain is uniform in $\la$ as $\la\downto 1$.

\begin{lemma}\label{converged+}
Let $0<\delta<1$, $1<\la \le 2$ and $\omega\ge 1$ be given,
let $t_\omega=\min\{t:|\XP_t|\ge \omega\}$,
whenever this is defined,
and let $E$ be the event that $|\XP_t|/\la^t$ is within
a factor $1\pm \delta$ of $\YP$
for all $t\ge t_\omega$.

Then $t_\omega$ is defined with probability 1, and
$\Pr(E)=1-e^{-\Omega(\omega)}$, where the implicit constant depends on
$\delta$ but not on $\la$.
\end{lemma}
\begin{proof}
The sequence $|\XP_t|$ is non-decreasing, and increases with probability
bounded away from zero (at least $\Pr(Z_\la>1)$) at each step, so $\XP_t\to\infty$
a.s., and $t_\omega$ is indeed defined with probability $1$.

Let $A$ be the event
\[
 A= \left\{ (1-\delta/10) |\XP_{t_\omega}|/\la^{t_\omega} 
 \le \YP
 \le (1+\delta/10) |\XP_{t_\omega}|/\la^{t_\omega}
 \right\}.
\]
Our first aim is to show that $A$ is very likely to hold.
Let us condition on the event $t_\omega=t$, where $t\ge 0$,
and also on $|\XP_t|$. Since $t_\omega$ is a stopping time, given
that $t_\omega=t$ and $|\XP_t|=m$, the descendants of the $m\ge \omega$
particles in $\XP_t$ form independent copies of the original process.
Let $n_{t',i}$ denote the number of descendants in generation $t'$
of the $i$th particle in $\XP_t$. Then for each $i$ we have
$n_{t',i}/\la^{t'-t}\to \YP_i$ a.s., where the $\YP_i$ are independent
and have the distribution of $\YP$.
It follows that $\YP=\sum_{i=1}^m \YP_i / \la^{t_\omega}$ a.s.

Now $\YP$ has mean $1$, and $m\ge \omega$.
Applying  Lemmas~\ref{etail} and~\ref{econc}, we see that
\[
\Pr\left( \left|\frac{1}{m}\sum_{i=1}^m\YP_i -1\right|\ge \delta/10 \right) 
=e^{-\Omega(m)}= e^{-\Omega(\omega)},
\]
so $\Pr(A)=1-e^{-\Omega(\omega)}$.

Let $B_-$ be the event that $t_\omega$ is defined, and there is some
$t>t_\omega$ for which $|\XP_t|/|\XP_{t_\omega}| \le (1-\delta/2) \la^{t-t_\omega}$.
If $B_-$ holds, let $t_1$ be the first such time.
Then $t_1$ (which is not always defined) is again a stopping time so, arguing as above,
given that $t_1=t$ and $|\XP_{t_1}|=m$, we have
$\YP=\sum_{i=1}^m \YP_i/ \la^{t_1}$, where the $\YP_i$ are iid with the distribution
of $\YP$. This also holds if we condition on the entire history up to time
$t_1$, and in particular on $t_\omega$ and $r=|\XP_{t_\omega}|$.

By definition of $B_-$ we have $m\le m_0=(1-\delta/2)\la^{t-t_\omega} r$,
so recalling that the $\YP_i$ are independent copies of $\YP$,
\[
 \Pr\left( \sum_{i=1}^m \YP_i \ge (1+\delta/10)m_0 \right) 
 \le 
 \Pr\left( \sum_{i=1}^{\floor{m_0}} \YP_i \ge (1+\delta/10)m_0 \right) 
 = e^{-\Omega(m_0)} = e^{-\Omega(\omega)},
\]
using Lemma~\ref{econc} and the fact that $\la^{t-t_\omega} r\ge r\ge \omega$
for the last step. If $B_-$ and $A$ both hold, then the event appearing on the left
above also holds, so we have shown that $\Pr(A\mid B_-)=e^{-\Omega(\omega)}$.
Hence, $\Pr(A\cap B_-)\le \Pr(A\mid B_-)=e^{-\Omega(\omega)}$.

Define $B_+$ to be the event that there is some $t>t_\omega$
for which  $|\XP_t|/|\XP_{t_\omega}| \ge (1+\delta/2) \la^{t-t_\omega}$.
A similar but simpler argument shows that $\Pr(A\cap B_+)=e^{-\Omega(\omega)}$.
Hence with probability $1-e^{-\Omega(\omega)}$  the event $A$ holds, while
neither $B_-$ nor $B_+$ does, and the lemma follows.
\end{proof}

Returning to the original branching process $\bp_\la=(X_t)$, recall that
this survives with probability $s=s(\la)=\Theta(\eps)$, where $\la=1+\eps$.
Recall also that $X_t/\la^t\to Y$ a.s., where by standard
results $\E(Y)=1$, and $Y=0$ if and only if the process
dies out, so $\Pr(Y\ne 0)=s$. Also,  recalling that $\YP$ has the distribution
of $Y^+$ conditioned on $Y^+>0$,   Lemma~\ref{YYp} implies that
the distribution of $sY$ given that $Y\ne 0$ is exactly the distribution
of $\YP$. 

The next lemma will be similar to Lemma~\ref{converged+}, but concerning $\bp_\la=(X_t)$.
This will lead us to consider the sum
$S_N$ of $N$ independent copies $Y_i$ of $Y$, for $\omega$ large and $N\ge \omega/\eps$.
Given $0<\delta<1$,
from concentration of the binomial distribution,
with probability $1-e^{-\Omega(\omega)}$
the number $M$ of $i$ with $Y_i\ne 0$ is within a factor $1\pm \delta$ of its
mean $sN=\Omega(\omega)$.
Conditional on $M$, the variable $sS_N$ is the sum of $M$ independent 
copies of $sY$ each conditioned to be positive,
or equivalently of $M$ independent copies of $\YP$,
so by Lemma~\ref{econc}, with probability $1-e^{-\Omega(M)}$
this sum is within a factor $1\pm \delta$ of
its mean $M$. It follows that with probability $1-e^{-\Omega(\omega)}$
we have $|S_N/N-1|\le 3\delta$, say.
Using this fact in place of concentration of the sum of $\omega$
copies of $\YP$, the proof of Lemma~\ref{converged+} gives the following result, which is more or less a sharpening of Lemma~\ref{l:X}. 
Recall that $Y=\lim_{t\to\infty}|X_t|/\la^t$.

\begin{lemma}\label{converged}
Let $\delta>0$, $1<\la \le 2$ and $\omega\ge 1$ be given, and set $\eps=\la-1$.
Let $t_{\omega/\eps}=\min\{t:|X_t|\ge \omega/\eps\}$,
whenever this is defined, let $\tdef$ be the event that $t_{\omega/\eps}$
is defined, 
let $E\subset \tdef$ be the event that $|X_t|/\la^t$ is within a factor
$1\pm \delta$ of $Y$ for all $t\ge t_{\omega/\eps}$, and
let $\surv=\{\forall t:|X_t|>0\}$
be the event that the process survives.

Then $\Pr(E\mid \tdef)=1-e^{-\Omega(\omega)}$, where the implicit constant depends on
$\delta$ but not on $\la$.
Furthermore, $\Pr(\surv\mid\tdef)=1-e^{-\Omega(\omega)}$
and $\Pr(\tdef\setminus E)=O(\eps e^{-\Omega(\omega)})$.
\end{lemma}
\begin{proof}
The first statement follows by modifying the proof of Lemma~\ref{converged+} as described
above. The second is an immediate consequence (and also easy to verify directly).
It implies in particular that $\Pr(\surv)/\Pr(\tdef)$ is bounded below,
so $\Pr(\tdef)=O(\Pr(\surv))=O(\eps)$. The final statement then follows from the first.
\end{proof}

Lemma~\ref{converged} tells us that for $\omega$ large enough,
the probability that the branching process takes much longer than
expected to reach size $\omega/\eps$ is essentially determined
by the tail of the distribution of $Y=Y_\la$ near $0$. 
Lemma~\ref{converged+} will be useful in studying this tail indirectly.

Writing $\RP$ for the set of non-negative reals,
let $(\Y_t)_{t\in \RP}$ be a standard Yule process. Thus
$\Y_0$ consists of a single particle, and each particle in the process
survives forever and gives rise to children according to a Poisson process
with rate 1, independently of the other particles and of the history.
Note that $|\Y_t|$ is a (random) non-decreasing function of $t$,
and that $\E(|\Y_t|)=e^t$.
It is well known that $\lim_{t\to\infty} |\Y_t|/e^t$
exists with probability~1 (see, for example,~\cite[Section III.7]{AN});
we denote this (random) limit by $W$.

It is not hard to see that as $\la$ decreases to $1$, the suitably
rescaled process $\bpP_\la$ converges in some sense to $(\Y_t)$.
All we shall need is a very weak result of this form.

\begin{lemma}\label{yule}
Let $T>0$ be fixed. As $\la=1+\eps$ tends to 1 from above, the distribution
of $|\XP_{\floor{T/\eps}}|$ converges to that of $|\Y_T|$.
\end{lemma}
\begin{proof}
We take snapshots of $(\Y_t)$ at times separated by $\eps$, i.e.,
consider $Y_n=\Y_{n\eps}$, $n=0,1,\ldots,T$.
Each particle $x$ in $Y_n$ always survives to $Y_{n+1}$, has no children
in $Y_{n+1}$ with probability $\Pr(\Po(\eps)=0)=e^{-\eps}=1-\eps+O(\eps^2)$,
and has exactly
one child in $Y_{n+1}$ with probability $\Pr(\Po(\eps)=1)=\eps+O(\eps^2)$.
Furthermore, the probability that this child (if it exists) has children
of its own by time $(n+1)\eps$ is $O(\eps)$. Hence, the number $Z'$ of descendants
of $x$ in $Y_{n+1}$ is $1$ with probability $1-\eps+O(\eps^2)$, two with probability
$\eps+O(\eps^2)$ and three or more with probability $O(\eps^2)$.
Hence $Z'$ and $Z_\la$, the offspring distribution in $\bpP_\la$, can be
coupled to agree with probability $1-O(\eps^2)$.

Using the independence properties of $\bpP_\la$ and of $(\Y_t)$, it follows
that these processes can be coupled so that the event
$E=\{ |\XP_n|=|Y_n|,\, n=0,1,\ldots,\floor{T/\eps}\}$
{\em fails} to hold with probability at most
\[
 O(\eps^2)\sum_{n\le T/\eps}\E(|Y_n|)= O(\eps^2) \sum_{n\le T/\eps} e^{\eps n}
 = O(\eps) e^T = O(\eps).
\]
Since $\Y_{\eps \floor{T/\eps}} = \Y_T$ with probability $1-O(\eps)$, the result follows.
\end{proof}

\begin{corollary}\label{Ycor}
As $\la=1+\eps$ tends to $1$ from above, $\YP_\la$ converges in distribution to $W$.
\end{corollary}
\begin{proof}
Fix $\delta>0$. It suffices to show that for $\eps$ sufficiently small
we can couple $\YP_\la$ and $W$ so that they agree within a factor of $1+O(\delta)$
with probability $1-O(\delta)$.

Let $\omega$ be a constant to be chosen below, depending on $\delta$ but not on $\eps$.
Since $|\Y_t|\to\infty$ with probability $1$, there is some $T$
such that $\Pr(|\Y_T|<\omega) \le \delta$.
{From} Lemma~\ref{yule}, if $\eps$ is sufficiently small,
then we may couple $\bpP_\la$ and $(\Y_t)$ so
that with probability at least $1-\delta$ we have $|\Y_T|=|\XP_{\floor{T/\eps}}|$.
Then with probability at least $1-2\delta$ we have $|\Y_T|=|\XP_{\floor{T/\eps}}|\ge \omega$.

Let $n=\floor{T/\eps}$. Applying Lemma~\ref{converged+}, it follows that
if $\omega$ is chosen large enough (depending only on $\delta$, not on $\eps$),
then with probability at least $1-3\delta$ the limit $\YP_\la$
is within a factor $1\pm \delta$ of $|\XP_n|/\la^n$.
A similar result holds for $(\Y_t)$. (Indeed,
since $|\Y_t|/e^t\to W$ a.s., there must be some constant $T'$ 
such that with probability $1-\delta$ we have $|\Y_t|/e^t$ within a factor
$1\pm \delta$ of $W$ for all $t\ge T'$.)
In particular, if $\omega$ is large enough, then with probability $1-\delta$ the ratio
$|\Y_T|/e^T$ is within a factor of $(1\pm\delta)$ of $W$.

Putting the pieces together, and noting that $\la^n=(1+\eps)^{\floor{T/\eps}}=e^T+O(\eps)$,
for $\eps$ small enough the quantities $\YP_\la$, $|\XP_n|/\la^n$, $|\XP_n|/e^T$,
$|\Y_T|/e^T$ and $W$ agree up to factors of $1+O(\delta)$ with probability $1-O(\delta)$,
completing the proof.
\end{proof}

It is well known, and not hard to check, that the (positive)
random variable $W$ associated
to the Yule process has an exponential distribution with mean $1$;
this is an exercise in~\cite{AN}, for example.
In particular,
\begin{equation}\label{Yt}
 \Pr(W\le x) = 1-e^{-x} \sim x
\end{equation}
as $x\to 0$ from above.
We are now ready to prove our bound on the lower tail of $\YP_\la$.

\begin{theorem}\label{YPt}
Let $\la=1+\eps$. As $\eps$ and $x$ tend to $0$ from above we have
\[ 
 \Pr(\YP_\la \le x) \sim x^{\log(1/\las)/\log\la}.
\]
\end{theorem}
Note that we make no assumption on the relative rates at which $\eps$ and $x$
tend to zero. With $x$ fixed, the result would be immediate from \eqref{Yt} and
Corollary~\ref{Ycor}.

\begin{proof}
Let $\delta>0$ be given. We must show that there are constants $x_0=x_0(\delta)$
and $\eps_0=\eps_0(\delta)$ such that for all $0<\eps<\eps_0$ and $0<x<x_0$
we have
$\Pr(\YP_\la \le x) = e^{O(\delta)} x^{\log(1/\las)/\log\la}$,
where the implicit constant is absolute.

By \eqref{Yt}, there is an $x_1>0$ such that for all $x\le x_1$ we have
\begin{equation}\label{Yconv}
 e^{-\delta} \le \Pr(W\le x)/x \le e^{\delta}.
\end{equation}
Fix such an $x_1$, and set $x_0=\min\{x_1,\delta\}$.

Trivially, for $(1-\delta)x_0\le x\le x_0$ and any $\la$, we have
%
\[
 \Pr\bb{\YP_\la\le(1-\delta)x_0} \le \Pr(\YP_\la\le x)\le \Pr(\YP_\la\le x_0).
\]
As $\eps\to 0$, from Corollary~\ref{Ycor}, for any constant $a$
we have $\Pr(\YP_\la \le a)\to \Pr(W\le a)$. Applying this with $a=x_0$ and $a=(1-\delta)x_0$,
it follows that there is an $\eps_0$ such that
\begin{equation}\label{inint}
 e^{-\delta}\Pr\bb{W\le (1-\delta)x_0} \le \Pr(\YP_\la \le x) \le e^\delta \Pr(W\le x_0)
\end{equation}
for all $\eps\le\eps_0$ and all $x$ in the interval $I=[(1-\delta)x_0,x_0]$.
We may and shall assume that $\eps_0<1/10$, say.
Since $\log(1/\las)/\log\la\to 1$ as $\eps\to 0$ (see \eqref{lllas}),
reducing $\eps_0$ if necessary, we have $x^{\log(1/\las)/\log\la}
 =e^{O(\delta)}x$
uniformly in $x\in I$ and $\eps\le\eps_0$.
Using \eqref{Yconv} and \eqref{inint}, it follows that
\begin{equation}\label{inI}
 \Pr(\YP_\la \le x) = e^{O(\delta)}  x^{\log(1/\las)/\log\la}
\end{equation}
for all $x\in I$ and $\eps\le \eps_0$, where the implicit constant is absolute.

At this point we return to the definition of $\YP_\la$ in terms of $(\XP_t)$.
Recall that $Z=Z_\la$ is a Poisson distribution with parameter $s\la$ conditioned
on being non-zero, and that $\E(Z)=\la$ and $\Pr(Z=1)=\las$.
Now $\YP_\la$ has the distribution of the sum of $Z$ independent copies $Y_1,\ldots,Y_Z$
of $\YP_\la/\la$. Hence, for
any $x$,
\[
 \Pr(\YP_\la\le x) \ge \Pr(Z=1,\, \YP_\la\le x) = \Pr(Z=1)\Pr(Y_1 \le x)
 = \las \Pr(\YP_\la \le \la x).
\]
Given any $x\le x_0$, there is some non-negative integer $r$ such that $x'=x \la^r$ lies in $I$.
{F}rom the inequality above it follows that $\Pr(\YP_\la \le x) \ge \las^r \Pr(\YP_\la\le x')$.
Applying \eqref{inI} to bound the second probability, 
and noting that $\las^r = \exp((-r)\log(1/\las))=\exp(\log(1/\las)\log(x/x')/\log\la)$,
it follows
that
\[
 \Pr(\YP_\la \le x)\ge (x/x')^{\log(1/\las)/\log\la} e^{O(\delta)} (x')^{\log(1/\las)/\log\la}
 = e^{O(\delta)}  x^{\log(1/\las)/\log\la},
\]
completing the proof of the lower bound.

For the upper bound we use the inequality
\begin{eqnarray*}
 \Pr(\YP_\la \le x) &=& \Pr\bb{Z=1,\, \YP_\la\le x} + \Pr\bb{Z\ge 2,\, \YP_\la\le x} \\
 &\le& \Pr(Z=1)\Pr(Y_1\le x) +\Pr(Z\ge 2)\Pr(Y_1+Y_2\le x) \\
 &\le& \Pr(Z=1)\Pr(\YP_\la\le \la x) +\Pr(Z\ge 2)\Pr(\YP_\la\le \la x)^2 \\
 &=& \las \Pr(\YP_\la \le \la x) \left(1+\frac{1-\las}{\las}\Pr(\YP_\la\le \la x)\right).
\end{eqnarray*}
Given $x\le x_0$, as before there is a non-negative integer $r$ such
that $x'=x\la^r\in I$. For $0\le i\le r$ let $x_i=x'/\la^i$,
so $x_0=x'$ and $x_r=x$. Let $p_i=\Pr(\YP_\la\le x_i)$, so
\[
 p_{i+1}\le \las p_i(1+\las^{-1}(1-\las)p_i)
\]
and hence, by induction,
\begin{equation}\label{pip0}
 p_i \le \las^i p_0 \prod_{j<i}(1+(\las^{-1}-1)p_j) \le \las^ip_0\exp\Biggl((\las^{-1}-1)\sum_{j<i}p_j\Biggr).
\end{equation}
Now $p_0=\Pr(\YP_\la<x')$ and $x'\in I$, so,
recalling that $x_0\le\delta$, we have $p_0=O(\delta)$.
Thus $p_0\le 1/10$, say, if we assume $\delta$ is small, which we may.
It follows by induction on $i$ that $p_i\le 2\las^i p_0\le \las^i/5$.
Indeed, if this holds for $j<i$, then the term inside the exponential
in \eqref{pip0} is at most
\[ 
 (\las^{-1}-1) \sum_{j<i} \las^j/5 \le \las^{-1}(1-\las)\sum_{j=0}^\infty \las^j/5 = \las^{-1}/5 \le 1/4,
\]
and $e^{1/4}<2$.
Plugging $p_j\le 2\las^j p_0$ back into \eqref{pip0}, we see that
$p_i\le \las^ip_0\exp(3p_0)=\las^i p_0 e^{O(\delta)}$.
Calculating as for the lower bound as above, this establishes the required upper bound.
\end{proof}

\begin{remark}  
The method used above shows that for $\la$ and $x$ bounded above,
$\Pr(\YP_\la\le x)$ is within a factor $C$ of $x^{\log(1/\las)/\log\la}$, where
$C$ depends only on the bounds we assume on $\la$ and $x$.
For $\la$ constant and $x\to 0$, this is a standard result.
Perhaps surprisingly, the conclusion of Theorem~\ref{YPt} does not hold in this case:
as pointed out to us by Svante Janson,
the limiting behaviour of $\Pr(\YP_\la \le x)/x^{\log(1/\las)/\log\la}$
as $x\to 0$ is oscillatory. One period corresponds to changing $x$
by a factor of $\la$, and the tail probability by a factor of $1/\las$.
\end{remark}

In the light of Lemma~\ref{YYp} and the fact that $\YP_\la$ is just $Y_\la^+$ conditioned
on being non-zero, an event of probability $s=s(\la)\sim 2\eps$ by~\eqn{sasympt}, Theorem~\ref{YPt}
has the following corollary.

\begin{corollary}\label{cYt}
Let $\la=1+\eps$. As $\eps$ and $x$ tend to $0$ from above we have
\[ 
 \Pr\bb{0<Y_\la \le x/\eps} \sim 4\eps x^{\log(1/\las)/\log\la}.
\]
\end{corollary}
\begin{proof}
Let $s=s(\la)$ denote the survival probability of $\bp_\la$.
Then
\begin{multline*}
 \Pr\bb{ 0<Y_\la \le x/\eps} = \Pr\bb{ 0<Y_\la^+ \le sx/\eps} \\
  = \Pr(Y_\la^+>0)\Pr\bb{Y_\la^+\le sx/\eps\mid Y_\la^+>0}
 = s \Pr\bb{\YP_\la \le sx/\eps},
\end{multline*}
where the first step is from Lemma~\ref{YYp} and the rest are from the definitions.
Applying Theorem~\ref{YPt},
and using once again $s\sim 2\eps$ and $\log(1/\las)/\log\la\sim 1$,
it follows that
\begin{multline*}
 \Pr(0<Y_\la \le x/\eps) \sim s (sx/\eps)^{\log(1/\las)/\log\la}\\
 \sim 2\eps\bb{(2+o(1))x}^{\log(1/\las)/\log\la}
 \sim 4\eps x^{\log(1/\las)/\log\la},
\end{multline*}
as claimed.
\end{proof}

In turn, Corollary~\ref{cYt} and Lemma~\ref{converged} will give us the required estimate
on the probability that the branching process $\bp_\la=(X_t)_{t\ge 0}$ takes
a long time to begin to have a large population.

Before turning to our tail bound, let us make a simple observation; the proof 
is analogous to, but simpler than, that of Lemma~\ref{nojg}, so we omit it.
\begin{lemma}\label{noj}
Let $\la=1+\eps$, $M\ge 1$ and $\delta>0$, 
and let $t_M=\min\{t:|X_t|\ge M\}$, whenever this is defined.
Given that $t_M$ is defined, the probability that $|X_{t_M}|$ exceeds $(1+\delta)(1+\eps)M$
is $e^{-\Omega(\delta^2 M)}$,
where the implicit constant is absolute.\noproof
\end{lemma}

In the following result, $t_{\omega/\eps}$ denotes $\min\{t: |X_t|\ge \omega/\eps\}$,
whenever this is defined.
\begin{theorem}\label{th19}
Let $\la=1+\eps$, and suppose that $\eps=\eps(n)\to 0$,
$\omega=\omega(n)\to\infty$,
and $t=t(n)$ satisfy $t\le 100\log\omega/\log \la$
and $\eps t\to\infty$. Then, with $t_1=\log\omega/\log\la$, we have
$$
 \Pr(t_{\omega/\eps} > t_1+t) \sim 4\eps\las^t,
$$
$$
 \Pr\bb{ 0<|X_r|< \omega/\eps,\, 0\le r\le t_1+t} \sim 4\eps\las^t,
$$
$$
 \Pr\bb{ 0<|X_{t_1+t}|< \omega/\eps} \sim 4\eps\las^t 
$$
and
$$
\Pr\bb{ (X_r)\mbox{ survives and }0<|X_{t_1+t}|< \omega/\eps }\sim 4\eps\las^t.
$$
\end{theorem}

\begin{proof}
We first show for any fixed $\delta>0$ we have
\begin{equation}\label{bda}
 \Pr\bb{t_{\omega/\eps} > t_1+t} =(1+O(\delta))4\eps\las^t + O(\eps e^{-\Omega(\omega)}),
\end{equation} 
\begin{equation}\label{bdb}
 \Pr\bb{ 0<|X_r|< \omega/\eps,\,0\le r\le t_1+t} = (1+O(\delta))4\eps\las^t + O(\eps e^{-\Omega(\omega)}),
\end{equation}
\begin{equation}\label{neweq0}
\Pr\bb{ 0<|X_{t_1+t}|< \omega/\eps }=(1+O(\delta)) 4\eps\las^t + O(\eps e^{-\Omega(\omega)}) 
 \end{equation}
 and
\begin{equation}\label{neweq}
\Pr\bb{ (X_r)\mbox{ survives and }0<|X_{t_1+t}|< \omega/\eps }=(1+O(\delta)) 4\eps\las^t + O(\eps e^{-\Omega(\omega)}).
 \end{equation}

Note that the events considered in \eqref{bda} and \eqref{bdb} are not quite the same: by the event
$A=\{t_{\omega/\eps}> t_1+t\}$
we mean the event that $t_{\omega/\eps}$ is defined and greater than $t_1+t$;
this certainly implies the event considered in \eqref{bdb}, but the latter may also hold
with $T=t_{\omega/\eps}$ undefined. 
Let $\surv$ be the event that the process survives,
noting that 
\begin{equation}\label{smid}
  \Pr(\surv\mid \{t_{\omega/\eps}\hbox{ is defined}\}) \ge 1-(1-s)^{\omega/\eps} =1-e^{-\Omega(\omega)}.
\end{equation}
In particular, for large $n$ this conditional probability is at least $1/2$, so the probability
that $T$ is defined is at most $2s=O(\eps)$.

Let $B_1$ be the `bad' event that $T=t_{\omega/\eps}$ is defined and
there is an $r\ge T$ with $|X_r|/\la^r$ outside the interval $(1\pm \delta)Y_\la$.
By Lemma~\ref{converged}, we have $\Pr(B_1\mid T\hbox{ defined})=e^{-\Omega(\omega)}$,
so $\Pr(B_1)=O(\eps e^{-\Omega(\omega)})$.
Let $B_2$ be the event that $T$ is defined and $|X_T|\ge (1+\delta)\omega/\eps$.
If $\eps$ is small enough, which we may assume, then $(1+\delta)\ge (1+\eps)(1+\delta/2)$,
and from Lemma~\ref{noj} we have $\Pr(B_2)=e^{-\Omega(\omega/\eps)} =O(\eps e^{-\Omega(\omega)})$.

Suppose that $B_1$ does not hold. Then if $t_{\omega/\eps}$ is defined, the process survives.
Thus, off $B_1$, the event that $T$ is
defined coincides with $\surv$ and hence with the event $Y_\la>0$.
Moreover, off $B_1\cup B_2$, whenever $Y_\la>0$
we have $Y_\la= (1+O(\delta)) |X_T|/\la^T= (1+O(\delta))(\omega/\eps)/\la^T$.
Thus, off $B_1\cup B_2$, for all sufficiently large constants $a$, $b>0$, 
 \begin{description}
 \item{(i)} $  Y_\la >  (1+a\delta)(\omega/\eps)/\la^{t_1+t}=(1+a\delta)/(\eps\la^t)$ implies    $T\le t_1+t$, and
 \item{(ii)}  $0<Y_\la \le (1-b\delta)/(\eps\la^t)$ implies $T> t_1+t$.
 \end{description}
Since $\eps t\to\infty$ we have $1/\la^t\to 0$.
Thus using Corollary~\ref{cYt} to bound the probabilities
 of the events on the left in (i) and (ii) above,
and recalling
that $\Pr(B_1\cup B_2)=O(\eps e^{-\Omega(\omega)})$, we obtain the 
bound \eqref{bda}.

To deduce \eqref{bdb}, it suffices to show that the probability that the indicated
event holds but $T$ is undefined is $o(\eps\las^t)$. Recall that up to probability 0 events,
if $\surv$ holds, then $T$ is defined.
So it suffices to bound the probability that $|X_{t_1+t}|>0$ but $\surv$ does not hold.
Now 
\[
 \Pr( |X_{t_1+t}|>0,\, \surv^\cc) = \Pr(\surv^\cc) \Pr(|X_{t_1+t}|>0\mid \surv^\cc) 
= (1-s) \Pr(|X_{t_1+t}^-|>0) \sim \Pr(|X_{t_1+t}^-|>0),
\]
where $(X_r^-)$ is the process conditioned on dying out, which has the distribution of $\bp_{\las}$.
As we shall see shortly (see Lemma~\ref{lsurv}),
$\Pr(|X_a^-|>0) = \Theta( \eps \las^a)$ as $\eps\to 0$ and $a\to\infty$ with $a=\Omega(1/\eps)$.
Since $\log\la=\Theta(\eps)$, we have $t_1+t\ge t_1=\Omega(1/\eps)$, so
\[
 \Pr( |X_{t_1+t}|>0,\, \surv^\cc) = \Theta( \eps \las^{t_1+t}) = \Theta(\eps\las^t(1/\omega)^{\log(1/\las)/\log\la})
 =o(\eps\las^t),
\]
using $\log(1/\las)/\log\la\sim 1$ (see \eqref{lllas})
and $\omega\to\infty$ for the last step. So~\eqn{bdb} follows.

To see that~\eqn{neweq0} holds, note that we may extend the
implication (i) above to imply $|X_{t_1+t}|>\omega/\eps$.
%
Also (ii) can trivially be extended
to imply $|X_{t_1+t}|<\omega/\eps$. Again applying
Corollary~\ref{cYt} gives~\eqn{neweq0}. Now~\eqn{neweq} also follows
since survival coincides with $Y_\lambda>0$.

To deduce that the various statements in the theorem hold, observe that under the
assumptions given on $\eps$, $\omega$ and $t$ but with
$\delta>0$ fixed, we have $\las^t=\omega^{-O(1)}$, while any function
that is $e^{-\Omega(\omega)}$ decreases
faster than any power of $\omega$, so the probabilities
in \eqref{bda}--\eqref{neweq} are all asymptotically
 $(1+O(\delta))4\eps\las^t$. Since $\delta>0$ is arbitrary, the same conclusion follows for $\delta\to 0$ slowly enough. 
\end{proof}
 
Theorem~\ref{th19} is the analogue of Lemma~\ref{l1}, giving (in the relevant range)
the distribution of the time the branching process takes to grow to a certain size.
It will turn out, however, that we need several further results about the branching process.

\subsection{Further branching process lemmas}\label{ss_bp2}

Theorem~\ref{th19} gives good bounds on the probability that the branching process
grows more slowly than expected. It will turn out that we also need a bound on the probability
that it grows faster. Such a bound is immediate from
Lemmas~\ref{YYp}, \ref{etail} and \ref{converged}.
However (to handle the case where $\eps^3 n$ grows slowly),  when $\eps t$ is small
we shall need a bound that is stronger than the one obtained this way.
This is easy to obtain directly using moment generating functions as in the proof of Lemma~\ref{etail}. Note that we study $(|X_t|)$ here rather than $(|\XP_t|)$.

\begin{lemma}\label{fast}
Suppose that $0<\eps<1/10$ and $\eps t<1/10$. Then
for all $N\ge 20t$ we have
\[
 \Pr(|X_t|\ge N)\le t^{-1}e^{-N/(20t)}.
\]
\end{lemma}
\begin{proof}
Let $m_r(\theta)=\E e^{\theta |X_r|}$ be the moment generating function of $|X_r|$,
so $m_0(\theta)=e^{\theta}\le 1+2\theta$ for $\theta<1/2$.
We have 
\[
 m_{r+1}(\theta) = \E(m_r(\theta)^{|X_1|}) = e^{\la(m_r(\theta)-1)} \le 1+ \la(m_r(\theta)-1)+ 2\la^2(m_r(\theta)-1)^2
\]
as long as $\la(m_r(\theta)-1)\le 3/2$. Let $g_r=m_r(1/(20t))-1$, noting that $g_0\le 2/(20t)=1/(10t)$.
Then, as long as $g_r\le 2/5$, we have
\[
 g_{r+1} \le \la g_r+2\la^2g_r^2 \le g_r(1+\eps+3g_r) \le g_r\exp(\eps+3g_r).
\]
We claim that
for $r\le t\le \eps^{-1}/10$ we have
\[
 g_r \le g_0 \exp\Bigl(\eps r+3\sum_{i<r} g_i\Bigr) \le g_0\exp\bb{1/10+3/10} < 2g_0 \le 1/(10t).
\]
The proof is by induction using the final bound $g_i<1/(10t)$ for $i<r$ to establish the second inequality.
Hence,
\[
 \E \bb{e^{|X_t|/(20t)}} =1+g_t \le 1+1/(10t).
\]
Applying Markov's inequality to $e^{|X_t|/(20t)}-1$, which is always
non-negative, it follows that
\[
 \Pr( |X_t| \ge N ) \le \frac{1}{10t}\bb{e^{N/(20t)}-1}^{-1} \le 
 \frac{1}{5t}e^{-N/(20t)}
\]
whenever $N\ge 20t$, as required.
\end{proof}

We next turn to various events associated to the subcritical branching process $\bp_{\las}=(X_t^-)$.
We start by estimating the probability that the process
survives to time $t$, as well as a derived quantity associated to the wedge condition.
If our aim is just to prove Theorem~\ref{th_eps}, then a considerably simpler
form of the following lemma will do. However, we shall prove
a more precise result useful also when it comes to studying the distribution.

\begin{lemma}\label{lsurv}
Let $\eps\to 0$ and set $s_t=s_t(\eps)=\Pr(|X_t^-|>0)$.
Then for $t=o(1/\eps)$ we have $s_t\sim 2/t$, while for $t\ge \eps^{-2/3}$
we have $s_t\sim \frac{2\eps}{\las^{-t}-1}$.
In particular, if $t=\Omega(1/\eps)$, then $s_t=\Theta(\eps \las^t)$,
and if $\eps t\to\infty$, then $s_t\sim 2\eps \las^t$.
Furthermore,
\[
 \prod_{t=1}^\infty (1-s_t) \sim \gamma_0\eps^2
\]
for some constant $\gamma_0>0$.
\end{lemma}

\begin{proof}
Let $\cs_t$ be the probability that a critical Poisson Galton--Watson branching process
survives to time $t$, so $\cs_0=1$ and $\cs_{t+1}=1-\exp(-\cs_t)$.
It is well known (see~\cite{Kol} or~\cite[Section I.9, Thm 1]{AN}) that $\cs_t\sim 2/t$ as $t\to\infty$, and indeed one can check that
\begin{equation}\label{cst}
 \cs_t = 2t^{-1} +O(t^{-2}).
\end{equation}
Moreover, $t\cs_t$ approaches $2$ from below. Clearly, $s_t<\cs_t$, so we have
\begin{equation}\label{uup} 
 s_t < \cs_t < 2/t
\end{equation}
for all $\eps>0$ and $t\ge 1$.

On the other hand, we may construct $\bp_{\las}$ by first
constructing a critical process, and then deleting each edge of the resulting
tree with probability $1-\las\sim \eps$, independently of all other edges.
If the critical process survives to time $t$, then there is at least one path of length
$t$ witnessing this, and it follows that $s_t\ge \cs_t(1-\las)^t=\cs_t(1-O(\eps t))$.
If $t=o(1/\eps)$,
then $(1-\las)^t\sim 1$, so $s_t\sim \cs_t$ and the first statement of the lemma follows.
Note also for later that
\begin{equation}\label{close}
 s_t = \cs_t -O(\eps t)\cs_t = \cs_t - O(\eps).
\end{equation}

For larger $t$ we use the following iterative formula, obtained by 
considering the number of particles in $X_1^-$ with descendants in generation $t+1$:
\begin{equation}\nonumber 
 s_{t+1} = \Pr\bb{ \Po( \las s_t)>0 } = 1-e^{-\las s_t} = \las s_t - \las^2 s_t^2/2 +O(s_t^3),
\end{equation}
where, since $\las\le 1$ and $s_t\le 1$, the implicit constant is absolute.
Note also that $s_{t+1}\le \las s_t$.
Rewriting the formula above,
\begin{equation}\label{sit}
 s_{t+1} = \las s_t \exp\bb{-\las s_t/2 +O(s_t^2)}.
\end{equation}
We now simply `guess' an approximate form for $s_t$ (obtained by solving a differential
equation, although things are not quite that simple):
for $t\ge 1$, set
\[
 r_t = \frac{2(1-\las)}{\las(\las^{-t}-1)} \sim \frac{2\eps}{\las^{-t}-1}.
\]
Since $a\ge 0$ implies $(1+a)^t-1\ge at$, we have $r_t\le r_1/t=2/t$ for all $t$.
In particular, $r_t\le 1/2$ for $t\ge 4$.
Also,
\[
 \frac{\las r_t}{r_{t+1}} = \frac{\las(\las^{-t-1}-1)}{\las^{-t}-1} \\
 = \frac{\las^{-t}-\las}{\las^{-t}-1} = 1+\frac{1-\las}{\las^{-t}-1}
 = 1+\las r_t/2.
\]
In particular, $r_{t+1}\le \las r_t$.
Furthermore, for $t\ge 4$, which implies $r_t\le 1/2$, we have
\begin{equation}\label{rit}
 r_{t+1} = \las r_t (1+\las r_t/2)^{-1}
 = \las r_t\exp\bb{-\las r_t/2 +O(r_t^2)}.
\end{equation}

Using \eqref{sit} and \eqref{rit},
it is now not hard to show that $s_t$ and $r_t$ remain close for all
large $t$.
Set $T=\floor{\eps^{-2/3}}$, noting that $T\to \infty$ and $T=o(1/\eps)$.
Note that
\[
 \las^{-T}=(1/\las)^T=(1+\eps+O(\eps^2))^T=1+T\eps+O(T^2\eps^2+T\eps^2)
=1+T\eps(1+O(\eps^{1/3})),
\]
so
\begin{equation}\label{start}
 r_T, s_T = (1+O(\eps^{1/3})) 2/T,
\end{equation}
using \eqref{cst} and \eqref{close} for $s_T$.

Let $\rho_t=s_t/r_t-1$, noting that $\rho_T=O(\eps^{1/3})$.
Then, from \eqref{sit} and \eqref{rit},
\begin{eqnarray*}
 1+\rho_{t+1} &=& (1+\rho_t) \exp\bb{-\las(s_t-r_t)/2 +O(r_t^2+s_t^2)} \\
  &=& (1+\rho_t) \exp\bb{-\las \rho_t r_t/2 +O(r_t^2+s_t^2)}.
\end{eqnarray*}
Since $r_t$ and $s_t$ are bounded, we have $\exp(O(r_t^2+s_t^2))\le M(r_t^2+s_t^2)$
for some absolute constant $M$.
For $\eps$ small and $t\ge T$ we have $r_t\le r_T\le 1/10$, say. It follows
that whatever the sign of $\rho_t$,
the $\exp(-\las \rho_t r_t/2)$ term `pulls $(1+\rho_t)$ towards $1$'
without overshooting, and hence that
\[
 |\rho_{t+1}| \le |\rho_t| + (1+|\rho_t|)M(r_t^2+s_t^2).
\]
Using $r_t\le \las^{t-T}r_T$ and $s_t\le \las^{t-T}s_T$, it follows that
\[
 |\rho_t|\le |\rho_T| + 2M\sum_{0\le s\le t-T} \las^{2s} (r_T^2+s_T^2),
\]
provided $|\rho_s|<1$ for $T\le s<t$.
Since $r_T^2\sim s_T^2\sim (4/T)^2 =\Theta(\eps^{4/3})$,
while $\sum_{s\ge 0} \las^{2s}=O(1/\eps)$,
it follows
easily that $|\rho_t|$ does remain bounded by $1$, and in fact
that $|\rho_t|=O(\eps^{1/3})$ uniformly in $t\ge T$.
In particular, $s_t\sim r_t$ for
$t\ge T$, proving the second statement of the lemma.
The next two statements follow.

Finally, we turn to the estimate on $\prod_{t\ge 1}(1-s_t)$.
{F}rom \eqref{cst} we see that $\sum_t |\cs_t-2/t|=\sum_t O(t^{-2})$ is bounded.
It follows that $\sum_{t\ge 3} \log\left(\frac{1-\cs_t}{1-2/t}\right)$ converges;
let us write $c$ for the value of this sum, which does not involve $\eps$.
Since $T\to \infty$, the sum truncated at $T$ converges to $c$ as $\eps\to 0$.
Hence, from \eqref{close},
\begin{multline*}
 \prod_{t<T} (1-s_t) = \prod_{t<T} (1-\cs_t+O(\eps)) = e^{O(\eps T)}\prod_{t<T}(1-\cs_t)\\
 \sim (1-\cs_1)(1-\cs_2)e^{c}\prod_{3\le t<T}(1-2/t) \sim \gamma_0 T^{-2},
\end{multline*}
for some constant $\gamma_0>0$.
On the other hand, comparison with an integral shows that
\[
 \sum_{t\ge T} r_t = -2\log(\eps T)+O(\eps T) = -2\log(\eps T) +o(1).
\]
We have already seen that $\sum_{t\ge T}r_t^2=o(1)$, and the same for $s_t$,
so, using the bound on $|\rho_t|$ established above, it follows that
\begin{multline*}
\log\prod_{t\ge T}(1-s_t) =o(1) -\sum_{t\ge T}s_t
 = o(1) -(1+O(\eps^{1/3})) \sum_{t\ge T}r_t                 \\
= 2\log(\eps T)+o(1)+O(\eps^{1/3}\log(\eps T))
 = 2\log(\eps T)+o(1).
\end{multline*}
Thus $\prod_{t\ge 1}(1-s_t) \sim \gamma_0 T^{-2} (\eps T)^2 =\gamma_0 \eps^2$, as claimed.
\end{proof}

The final statement of Lemma~\ref{lsurv} shows that if we start one copy of $\bp_{\las}$ at each
time $t\ge 1$, the probability that for every $t$ the $t$th copy dies within $t$ generations
is asymptotically $\gamma_0\eps^2$.

\begin{remark}
Constructing $\bp_\la$ first by constructing $\bp_\la^+$, and then adding the subcritical
trees, we see that
$p_r=\prod_{t=1}^r (1-s_t)$ is exactly the probability that $|X_r|=1$ given that $|X_r^+|=1$.
We have
\begin{multline*}
 \Pr\bb{|X_r|=1 \bigm| |X_r^+|=1}  = \Pr\bb{|X_r^+|=1\bigm| |X_r|=1}\Pr(|X_r|=1)/\Pr(|X_r^+|=1) \\
 = s\Pr(|X_r|=1)/\Pr(|X_r^+|=1) = s\Pr(|X_r|=1)/(s\las^r) = \Pr(|X_r|=1)/\las^r.
\end{multline*}
So the final statement of Lemma~\ref{lsurv}
is equivalent to the statement that for large $r$, $\Pr(|X_r|=1)\sim \gamma_0\eps^2\las^r$
for some constant $\gamma_0$, which can presumably be seen more directly somehow.
\end{remark}

Before turning to our next real lemma, let us get a simple observation out of the way.
Trivially, $\E(|X_t^-|)=\las^t$; a simple inductive calculation gives the standard
formula $\E(|X_t^-|^2)=\las^t(1+\las+\cdots+\las^t)$.
Since $\las>1-\eps$ (see \eqref{las1e}), this gives $\E(|X_t^-|^2) \le \eps^{-1}\las^{2t}$,
so $\Var(|X_t^-|)\le \eps^{-1}(\E |X_t^-|)^2$.
If we start $\bp_{\las}$ with $N\ge 10/\eps$ particles in generation $0$,
and $r\le 1/\eps$, then the size of generation $r$ has expectation $\mu \ge N(1-\eps)^{1/\eps}\ge N/3$,
and, using independence of the offspring of different particles,
variance at most $\eps^{-1}\mu^2/N\le \mu^2/10$. It follows by Chebyshev's inequality
that
\begin{equation}\label{slowdie}
 \Pr\bb{|X_r^-|\ge N/6  \bigm| |X_0^-|=N} \ge 1/2
\end{equation}
whenever $N\ge 10/\eps$ and $r\le1/\eps$.

Let $(D_t)_{t\ge 0}$ denote the union of countably many independent copies of $\bp_{\las}$,
where the $i$th process starts with a single particle in generation $i$.
Thus $|D_0|=1$, while given $|D_t|$, the distribution of $|D_{t+1}|$ has the form
$1+\Po(\las |D_t|)$.

\begin{lemma}\label{dthin}
Let $0<\eps<1/10$ be given, and define $\la=1+\eps$ and $\las=\la(1-s(\la))$ as usual.
For $\omega\ge 20$ and $t\ge 0$
we have $\Pr(|D_t|\ge \omega/\eps)=e^{-\Omega(\omega)}$, where the implied
constant is absolute.
Furthermore, for $T\ge 1/\eps$,
\[
 \Pr\bb{ \exists t:0\le t\le T,\, |D_t|\ge \omega/\eps} = O(\eps T e^{-\Omega(\omega)}).
\]
\end{lemma}
\begin{proof}
Let $f_t(x)=\E x^{|D_t|}$ be the probability generating function of $|D_t|$.
Then $f_0(x)=x$, while from the relationship between $D_{t+1}$ and $D_t$ above we have
\[
 f_{t+1}(x)= x f_t\bb{e^{\las(x-1)}}
\]
for all $t\ge 0$ and all $x$.
Fix $t\ge 0$ and let $x_0=1+\eps/10$, say. Inductively defining
$x_r$ by $x_{r+1}=e^{\las(x_r-1)}>1$, note that
\begin{equation}\label{ftx0}
 f_t(x_0)=\prod_{r=0}^t x_r \le \prod_{r=0}^\infty x_r.
\end{equation}
We claim that for every $r$ we have
\begin{equation}\label{xrind}
 x_r\le 1+(1-\eps/3)^r \eps/10,
\end{equation}
say. This certainly holds for $r=0$. Suppose then that \eqref{xrind} holds for some particular $r$.
Since $\las<(1-\eps/2)$, it follows that $x_{r+1}\le \exp( (1-\eps/2) (1-\eps/3)^r\eps/10)$.
Using $\exp(y)\le 1+y+y^2$ for $y\le 1$, we thus have
\[
 x_{r+1} \le 1+  (1-\eps/2)(1-\eps/3)^r\eps/10 + (1-\eps/3)^r\eps^2/100 \le 1+(1-\eps/3)^{r+1}\eps/10,
\]
and \eqref{xrind} follows by induction.

Combining \eqref{ftx0} and \eqref{xrind} we have, crudely, $\log(f_t(x_0))\le 2\sum_r (1-\eps/3)^r\eps/10
=6/10$, so $f_t(x_0)\le 2$. Recalling that $x_0=1+\eps/10$, we thus have
$\Pr(|D_t|\ge \omega/\eps)\le f_t(x_0)/x_0^{\omega/\eps} \le 2e^{-\Omega(\omega)}$, and the first statement
of the lemma follows, for all $\omega\ge 2$, say.

For the second statement, suppose now that $\omega\ge 20$.
Using \eqref{slowdie}, and simply ignoring the one new particle added
in each generation, for $0\le r\le 1/\eps$,
conditional on $|D_t|=N\ge 10/\eps$,
the probability that $D_{t+r}\ge N/6$ is at least $1/2$.
Let $k=\floor{1/\eps}$.
Examining $D_t$, $D_{t+1},\ldots$, one by one, stopping the first time any of
these sets has size more than $\omega/\eps$, it follows
that
\[
 \Pr\bb{ |D_{t+k}|\ge \omega/(6\eps) \bigm| \exists t': t\le t'\le t+k,\, |D_{t'}|\ge \omega/\eps } \ge 1/2,
\]
so
\[
  \Pr\bb{ \exists t': t\le t'\le t+k,\, |D_{t'}|\ge \omega/\eps } \le 2 \Pr\bb{|D_{t+k}|\ge \omega/(6\eps)}
 = e^{-\Omega(\omega)},
\]
using the first part for the final bound. Summing over $0\le t\le T$ in steps of $k=\floor{1/\eps}$,
the second statement follows.
\end{proof}

Next we shall show that conditioning $\bp_{\las}$ to survive to (at least)
a certain time does not increase its expected total size too much.

\begin{lemma}\label{tall}
Suppose that $\eps>0$ and $t\ge 1$. Let $N$ denote the total number of
particles in $\bp_{\las}$. Then $\E(N\mid X_t^-\ne\emptyset)\le (t+1)/\eps$.
\end{lemma}
\begin{proof}
We shall use repeatedly the observation that for any $\mu$,
the distribution of a Poisson $\Po(\mu)$ random variable conditioned
to be at least $1$ is stochastically dominated by $1+\Po(\mu)$.
(This may be seen by considering the first point, if any, of a Poisson process
in an interval.)

We may view the first generation of $\bp_{\las}$ as the union of two sets: the set $S_1$
consisting of those children of the root
that survive to time $t$, and the set $S_2$ of those that do not.
The full process is then obtained by taking a copy of the process conditioned to survive
for $t-1$ generations for each particle in $S_1$, and a copy conditioned to die
within $t-1$ generations for each in $S_2$.
The sets $S_1$ and $S_2$ have independent Poisson sizes. Conditioning on $X_t^-$
being non-empty is equivalent to conditioning on $|S_1|\ge 1$. Let us
instead simply add a new particle to $S_1$. By the observation at the start
of the proof, this gives a process whose distribution dominates that of $\bp_{\las}$.
Our new process consists exactly of the standard process $\bp_{\las}$,
together with a copy of $\bp_{\las}$ conditioned to survive at least $t-1$ generations
started at time $1$.

Applying the same procedure to the new copy (i.e., to the children of the extra
particle in $S_1$, but {\em not} to those of the other particles in $S_1$), and continuing,
it follows that the distribution of $\bp_{\las}$ conditioned to survive
to time $t$ is dominated by the distribution of the union of $t+1$ copies
of $\bp_{\las}$, one started at each time $r$, $0\le r\le t$. This has expected
total size $(t+1)/(1-\las)\le (t+1)/\eps$.
\end{proof}

Finally, we observe that if we condition on $\bp_\la$ surviving,
this process quickly realizes its {\em conditional} expected size, which
by Lemma~\ref{l:X} is a factor $(1+o(1))/s\sim 1/(2\eps)$ larger than the unconditioned size.
\begin{lemma}\label{livestart}
Let $\la=1+\eps$, where $\eps=\eps(n)\to 0$. Let $\omega(n)$ and $\omega'(n)$ satisfy
$\omega'\to\infty$ and $\omega/\omega'\to\infty$, and set
$t_1=\floor{\log\omega/\log\la}$. Then
$\Pr\bb{|X_{t_1}|\ge \omega'/\eps \mid (X_t)\hbox{ survives}}\to 1$
as $n\to\infty$.
\end{lemma}
\begin{proof}
This is a simple consequence of Lemma~\ref{converged} together with 
(a weak form of) our tail bound on the (a.s.\ defined) limit $Y=\lim_{t\to\infty}|X_t|/\la^t$.
Starting with the tail bound, set $x=2\omega'/\omega=o(1)$.
{From} Corollary~\ref{cYt} we have
\[
 \Pr(0<Y\le x/\eps)\sim 4\eps x^{\log(1/\las)/\log\la} =o(\eps),
\]
since $\log(1/\las)/\log\la\sim 1$.
Recalling that $Y>0$ if and only if the process survives, it follows
that $\Pr\bb{Y\le x/\eps \bigm| (X_t)\hbox{ survives}} = o(1)$.

Conditional on survival, there is some generation
with size at least $\omega'/\eps$ with probability $1$. By Lemma~\ref{converged},
with probability $1-o(1)$ the first such generation occurs at time
$\log(\omega'/\eps)/\log\la-\log Y/\log\la +O(1/\eps)$. By the tail
bound on $Y$ above, this is less than
$t_1$ with probability $1+o(1)$. Moreover, with probability $1-o(1)$,
from this point on $|X_t|$ is within a factor $2$, say, of
$\la^t Y$. At time $t_1$, $\la^{t_1}Y\ge 2\omega'/\eps$ unless
$Y\le 2\eps^{-1}\omega'/\omega=x/\eps$, an event of probability $o(1)$.
\end{proof}

\subsection{Typical distances in the 2-core}\label{ss_2c}

We are now almost ready to prove our lower bound on the diameter of $G(n,p)$.
It turns out that we need a result concerning typical
distances in the 2-core. Unfortunately, this does not seem to
follow easily from any published results, and our proof is a little painful.
We first need a result that essentially bounds
the $k$th moment of the size of the giant component.

Let $G=G(n,\la/n)$. We say that a $k$-tuple $(x_1,\ldots,x_k)$ of not necessarily
distinct vertices of $G$ is {\em useful} if for each $i$, either $x_i$
is in a component of $G$ containing a cycle, or it is joined by a path
in $G$ to some other $x_j$. It turns out that almost all useful
$k$-tuples arise as they should, i.e., from vertices in the giant component.
Recall that if $N$ is the number of vertices in the giant component of $G(n,\la/n)$,
then $N=(2+\littleop(1))\eps n$; see~\eqref{C1}.
An immediate consequence is that $\E N\ge (2-o(1))\eps n$. (In fact,
it is well known that $\E N\sim 2\eps n$.)

\begin{lemma}\label{lgc}
Let $\la=1+\eps$, where $\eps=\eps(n)>0$ satisfies $\eps\to 0$
and $\La=\eps^3 n\to\infty$, and let $k\ge 1$
be fixed. Then the expected number of useful $k$-tuples
in $G(n,\la/n)$ is $(1+o(1))(2\eps n)^k$.
\end{lemma}
\begin{proof}
The lower bound is immediate, since the number of useful $k$-tuples
is at least the $k$th power of the number $N$ of vertices in the largest
component of $G$, and $\E N^k\ge (\E N)^k\ge \bb{(2-o(1))\eps n}^k$.

Let $\psi=\psi(n)$ tend to infinity very slowly.

We first get a simple observation
out of the way. Let us say that a $k$-tuple of vertices is {\em close}
if each $x_j$, $j>1$, is within distance $\psi/\eps$ of $x_1$ in $G$.
Let $C_k$ denote the number of close $k$-tuples.
Set $t=\floor{\psi/\eps}$.
Then $\E C_k = n \E(|\Gat(x)|^{k-1})$, where $x=x_1$ is any fixed vertex of $G$.
Now $|\Gat(x)|$ is stochastically dominated by $|X_{\le t}|$, the union
of the first $t$ generations of $\bp_\la$.
(We are simplifying slightly here:
a binomial $\Bi(n,p)$ is  dominated by a Poisson with mean $-n\log(1-p)=np+O(np^2)$,
so we should consider the branching process with a parameter slightly larger
than $\la$; the difference is negligible.)
It is easy to check (for example by calculating inductively) that
with $r$ fixed,
$\E |X_{\le t}|^r = O(t^{2r-1}\la^{rt}) = O(\psi^{2r-1}e^{r\psi} \eps^{-(2r-1)})
=\Ot(\eps^{-(2r-1)})$, where we write $f=\Ot(g)$ if $f$ is bounded by a function
of $\psi$ times $g$.
It follows that
\begin{equation}\label{eCk}
 \E C_k = \Ot(n\eps^{-(2k-3)}) = o(\eps^k n^k),
\end{equation}
provided $\psi$ grows slowly enough, since $\eps^{3k-3}n^{k-1}\to\infty$.

Turning to useful $k$-tuples, we shall
proceed by induction on $k$. Let $U_k$ denote the number of $k$-tuples
of {\em distinct} vertices that are useful. Since $\eps n\to\infty$,
it suffices to prove that
$\E U_k\sim (2\eps n)^k$. We may then bound the total number
of useful $k$-tuples in terms of $U_1,\ldots,U_k$.

{From} now on we insist that $x_1,\ldots,x_k$ are distinct.
Let us say that a useful $k$-tuple is {\em reducible}
if it contains a non-empty subset $S$ which forms a close $r$-tuple
within a component of $G$ containing none of the remaining $x_i$.
If this holds, then there is some set of edges present
witnessing that $S$ is a close $r$-tuple, and a disjoint
set witnessing the event that the remaining set $S^\cc$
is useful. (We may have $k-r=0$; a 0-tuple is always useful.)
By the van den Berg--Kesten inequality~\cite{vdBK},
the probability of this event is at most the probability that $S$ is
close times the probability that $S^\cc$ is useful.
Using \eqref{eCk} and the induction hypothesis, this probability
is $o(\eps^r)O(\eps^{k-r})=o(\eps^k)$. Summing over $r$ and
over the $\binom{k}{r}$ sets $S$, we see
that the expected number of reducible useful $k$-tuples is $o(\eps^k n^k)$.

Finally, we estimate the number of irreducible useful $k$-tuples.
To do so, let us pick $x_1,\ldots,x_k$ one-by-one; we do {\em not}
fix them in advance. Each $x_i$ is chosen uniformly from the remaining
$n-i+1$ vertices.

Having chosen $x_i$, let us explore its neighbourhoods as follows.
First, if $x_i$ itself is in the set $R$ of vertices
previously reached by such explorations, we do not explore
at all, and declare $x_i$ to be `atypical for reason 1'.
Otherwise, we explore the neighbourhoods of $x_i$ as usual,
except that we do not (for the moment) test for edges
to $R$. Also, we stop as soon as either
(i) we reach generation $\psi/\eps$, or (ii) we find
$\psi/\eps$ vertices in one generation $t$. (We then stop
partway through this generation.)
Let $\Ga_i$ denote the set of vertices reached.
Our next step is to test all edges from $\Ga_i$ to $R$;
if such an edge is present, $x_i$ is `atypical for reason 2'.
We then test for non-tree edges within $\Ga_i$, i.e., for edges between two vertices in $\Ga_i$ at distance $t$ from $x_i$,
or for `redundant' edges between vertices at distances $t$ and $t+1$.
If we find such a non-tree edge, then $x_i$ is `atypical
for reason 3'.
Finally, if we have not yet labelled $x_i$ as atypical,
then we label $x_i$ as `good' if condition (i) or (ii) held,
and `bad' otherwise, i.e., if we ran out of vertices to explore.

Note that if any $x_i$ is bad, then $\Ga_i$ is its entire component,
this component is a tree, and every vertex of this
tree is within distance $\psi/\eps$ of $x_i$.
If $(x_1,\ldots,x_k)$ is useful and some $x_i$ is bad, then it follows
that at least one later $x_j$ lies in $\Ga_i$, so $(x_1,\ldots,x_k)$ is reducible.
Thus we may bound the expected number of irreducible useful $k$-tuples
by $n^k$ times the probability that no $x_i$ is bad. We do this by showing
that the conditional probability that $x_i$ is atypical or good
given $x_1,\ldots,x_{i-1}$ and the associated explorations
is at most $(1+o(1))2\eps$.

The definition of the exploration ensures that each $\Ga_i$
contains at most $\psi^2/\eps^2$ vertices, so $|R|\le k\psi^2\eps^{-2}$
and the probability that $x_i$ is atypical for reason $1$
is $\Ot(\eps^{-2}n^{-1})=o(\eps)$.
Suppose this does not happen. Then $|\Ga_i|$ is stochastically
dominated by $|X_{\le \psi/\eps}|$, which has expectation
$\sum_{r\le \psi/\eps}\la^r=O(\eps^{-1}\la^{\psi/\eps})=\Ot(\eps^{-1})$.
At the end of the previous exploration, we have already uncovered
all edges incident with all vertices of each $\Ga_j$, $j<i$,
except (possibly) for vertices in the last two generations.
(Two because we may have stopped part way through a generation.)
There are at most $2\psi k/\eps=\Ot(\eps^{-1})$ such vertices
in total. Hence, given $\Ga_i$, the conditional probability
that $x_i$ is atypical for reason $2$ is at most $|\Ga_i| \Ot(\eps^{-1}/n)$,
so the unconditional probability is at most $\Ot(\E|\Ga_i| \eps^{-1}/n) =\Ot(\eps^{-2}n^{-1})
=o(\eps)$.

Similarly, given $\Ga_i$, the conditional probability that $x_i$ is atypical
for reason $3$ is at most $|\Ga_i|(2\psi/\eps)\la/n$, since for each
vertex we have to test edges to the at most $2\psi/\eps$ other vertices
in the same generation or the previous generation. Hence the probability that $x_i$ is atypical
for this reason is also $o(\eps)$.

Finally, the exploration leading to $\Ga_i$ is dominated by $\bp_\la$,
so the probability that $x_i$ is good is bounded by the probability
that the branching process $\bp_\la$ either reaches size $\psi/\eps$,
or lasts for at least $\psi/\eps$ generations. It is easy to check
that the probability of this event is $(1+o(1))s\sim 2\eps$;
indeed, from Lemma~\ref{converged} (say), the event that $\bp_\la$
reaches size $\psi/\eps$ coincides up to probability $o(\eps)$
with the event that $\bp_\la$ survives, and Lemma~\ref{lsurv}
and the fact that $\bp_\la$ conditioned on dying is $(X_t^-)$
show that the events that $\bp_\la$ survives for $\psi/\eps$ generations
and that it survives forever agree up to probability $o(\eps)$.
\end{proof}

\begin{lemma}\label{l2c}
Let $\la=1+\eps$, where $\eps=\eps(n)>0$ satisfies $\eps\to 0$
and $\La=\eps^3 n\to\infty$, and let $C$ denote
the $2$-core of $G=G(n,\la/n)$. Then $N=|C|$ satisfies
\begin{equation}\label{Ck}
 \E (N^k) \sim (2\eps^2 n)^k
\end{equation}
for each fixed $k$. Furthermore, if $d=\log\La/\log\la-\omega/\eps$ with $\omega=\omega(n)\to\infty$,
then
\begin{equation}\label{md}
 \E M_d^{(k)} = o(\eps^{2k}n^k),
\end{equation}
where $M_d^{(k)}$ is the number of $k$-tuples of vertices of $C$ some pair of which 
are within distance~$d$.
\end{lemma}
One might expect the first statement to be known. Indeed,
Pittel and Wormald~\cite{PWio} have shown that the distribution of the size of the 2-core
is asymptotically normal, with mean $(2+o(1))\eps^2 n$ and variance
$(12+o(1))\eps n=o(\eps^4n^2)$. Unfortunately, convergence in distribution
does not imply convergence of the relevant moments, so we cannot simply
deduce \eqref{Ck}. We shall prove \eqref{Ck} using Lemma~\ref{lgc}; it is then
easy to deduce \eqref{md}.
\begin{proof}
Fix $k$ distinct vertices $x_1,\ldots,x_k$, and let $A$ be the event
that $x_1,\ldots,x_k$ are all in the 2-core. It suffices to show
that $\Pr(A)\le (1+o(1))(2\eps^2)^k$.

Let $G'=G-\{x_1,\ldots,x_k\}$, so $G'$ has the distribution of
$G(n',\la'/n')$ where $n'=n-k$ and $\la'-1\sim \la-1$. Let $U_r$ denote
the number of useful $r$-tuples of not necessarily distinct vertices
of $G'$. By Lemma~\ref{lgc}, we have
\begin{equation}\label{eu}
 \E U_r\le (1+o(1))(2\eps n)^r
\end{equation}
for any fixed $r$.

Suppose that $A$ holds, and let $E$ be a minimal set of edges witnessing
$A$. Note that every vertex of $S=\{x_1,\ldots,x_k\}$ meets at least
two edges of $E$. Also, since a vertex is in the 2-core if and only
if it is on a cycle or on a path joining two cycles,
$E$ may be written as the union of $k$ graphs with maximum degree at most $3$,
so at most $3k^2$ edges of $E$ meet $S$.

Let $E_0$, $E_1$ and $E_2$ denote respectively the sets of edges
of $E$ with both ends in $S$, one end in $S$, and neither end in $S$.
List the edges of $E_1$ as $a_ib_i$, $1\le i\le r\le 3k^2$,
where each $a_i$ is in $S$ and each $b_i$ in $G'$.
{From} the minimality of $E$, each $b_i$ is either joined to some other $b_j$
by a path in $E_2$ (which may have length $0$ if $b_i=b_j$), or is joined
by a path in $E_2$ to a cycle in $E_2$. (Otherwise, removing pendant
edges from $E$, we obtain a smaller witness to $A$.)
It follows that the $r$-tuple $(b_1,\ldots,b_r)$ is useful {\em in the graph $G'$}.
Let $t=|E_0|$.

Suppose first that $t=0$ and $|E_1|=2k$.
(More precisely, suppose there is a (minimal) witness $E$
with these properties.) Since each $x_i$ meets at least two edges
of $E_1$, it meets exactly two. Hence there is a
$2k$-tuple $(b_1,\ldots,b_{2k})$ that is useful in $G'$, with $x_i$ joined
to $b_{2i-1}$ and $b_{2i}$. But from \eqref{eu} and the independence
of $G'$ and the edges between $S$ and $G'$,
the expected number of such $2k$-tuples
is at most $(1+o(1)) (2\eps n)^{2k} (\la/n)^{2k} \sim 2^{2k}\eps^{2k}$.
Since the $2k$-tuple is ordered, whenever there is one there are at least $2^k$
(swapping $b_1$ and $b_2$, etc), so the probability that
a witness $E$ exists with $t=0$ and $|E_1|=2k$ is at most $(1+o(1))(2\eps^2)^k$.

It remains to show that the probability that there is a witness $E$ with $t>0$
or $t=0$ and $|E_1|>2k$ is $o(\eps^{2k})$, for which we simply bound the expected number
of such witnesses. Since each
vertex of $S$ meets at least two edges of $E_1$, we have $r=|E_1|\ge 2k-2t$,
while, as noted above, $r\le 3k^2$. Hence, setting $\Delta=0$
if $t>0$ and $\Delta=1$ if $t=0$, the expectation is bounded by
\[
 \sum_{t=0}^{\binom{k}{2}} \binom{k}{2}^t (\la/n)^t \sum_{r=2k-2t+\Delta}^{3k^2} k^r (\E U_r) (\la/n)^r,
\]
since there are most $\binom{k}{2}$ choices for each of the $t$ edges inside $S$,
and, given $r$, at most $k^r$ possibilities for which of the $x_j$ each $a_i$ is.
(Some $b_i$ may coincide, but we do not care.)
By \eqref{eu}, each term in the sum may be bounded by a constant times
\[
 n^{-t}(2\eps n)^r n^{-r} = O(n^{-t} \eps^r) = O(n^{-t} \eps^{2k-2t+\Delta}).
\]
For $t=0$ this is $O(\eps^{2k+1})=o(\eps^{2k})$. For $t\ge 1$,
since $\eps^2n\to\infty$, the final bound is $o(\eps^{2k})$. It
follows that $\Pr(A)\sim (2\eps^2)^k$, completing the proof of \eqref{Ck}.

Finally, as noted above,
it is relatively easy to deduce \eqref{md} from \eqref{Ck}.
Let $M$ be the number of $k$-tuples of vertices of $C$ in which every
pair is at distance larger than $d$. Then it suffices to show that
$\E M\ge (1+o(1))(2\eps^2n)^k$. 
In proving such a lower bound, we may consider $k$-tuples with additional
properties that make the analysis easier.

Let $\psi=\psi(n)=o(\omega)$ tend to infinity very slowly, let $E$ be the branching process 
event that at least two particles in generation $1$ survive to generation
$t=\psi/\eps$, that these particles each have at least $\psi/\eps$ descendants
in $X_t$, and that
$|X_{t'}| \le \psi^{10} \eps^{-1} \la^{\psi/\eps} = e^{O(\psi)}\eps^{-1}$
for $0\le t'\le t$. Recalling that, conditioned on survival, the branching process
typically has size of order $\eps^{-1}\la^{t'}$ in generations $t'$ where $t'$
is significantly larger than $1/\eps$ (see Lemmas~\ref{livestart} and~\ref{converged}),
it is easy to check that $\Pr(E)\sim s^2/2\sim 2\eps^2$, the asymptotic
probability that two particles in generation $1$ survive.
Also, Lemma~\ref{spcpl} applies to all trees consistent with $E$.

Given distinct vertices $x_1,\ldots,x_k$ of $G$, let $E_k'$ denote the event
that for every $i$ the $t$-neighbourhood of $x_i$ has the property corresponding
to $E$, and these $t$-neighbourhoods are disjoint.
Also, let $E_k$ be the event that $E_k'$ holds, every $x_i$ is in the 2-core,
and $d(x_i,x_j)>d$ for all $i$ and $j$.
By Lemma~\ref{spcplk} we have $\Pr(E_k')\sim \Pr(E)^k \sim (2\eps^2)^k$.
Since $\E M\ge (1+o(1))n^k\Pr(E_k)$, it thus suffices to show that $\Pr(E_k\mid E_k')=1-o(1)$.

But after testing whether $E_k'$ holds, we have
not looked at any edges outside the relevant neighbourhoods.
The expected number of paths of length at most $d$ joining one pair of vertices
in the last generation of these neighbourhoods
is bounded by
\[
 \sum_{1\le i\le d} n^{i-1}(\la/n)^i = n^{-1} \sum_{i\le d} \la^i \sim \eps^{-1}n^{-1}\la ^d.
\]
There are at most $\binom{k}{2}e^{O(\psi)}\eps^{-2}$ pairs to consider, so
the probability of finding any such path is at most
\[
 e^{O(\psi)} \eps^{-3}n^{-1}\la^d =  e^{O(\psi)}\La^{-1}\La \la^{-\omega/\eps} = e^{O(\psi)}e^{-(1+o(1))\omega} =o(1).
\]
Also, since for each of $x_1,\ldots,x_k$ we have two neighbours
with many (at least $\psi/\eps$) descendants in generation $t$, given $E_k'$
it is very likely that these neighbourhoods continue to expand and eventually meet,
so whp each $x_i$ is in $C$. Thus $\Pr(E_k\mid E_k')=1-o(1)$, as required.
\end{proof}
In fact, one can easily bound the expected number of pairs of vertices
of $C$ at distance significantly larger than $\log\La/\log\la$, noting that
all but at most $o(\eps^4n^2)$ such pairs also have the property $E_2'$.
Using Lemma~\ref{meet}
it is then easy to extend the argument above to show that if $x$ and $y$
are chosen uniformly at random from $C$, then
\[
 d(x,y) = \log\La/\log\la +\Op(1/\eps).
\]
Furthermore, one can obtain the limiting distribution of the correction
term without too much difficulty.
We omit the details as this is not our focus, and Lemma~\ref{l2c} is
all we shall need to know about the 2-core.

With the simple preliminaries of the last few subsections behind us,
we are now ready to begin the proof of Theorem~\ref{th_eps}.

\subsection{The lower bound on the diameter}\label{ss_low}

In this section we shall prove the lower bound on the diameter in Theorem~\ref{th_eps}.
As noted in Section~\ref{sec_intro}, we may assume that $\eps\to 0$.
The argument we present will be rather complicated. It is difficult to explain
why this is the case, other than to say that we have tried many promising simple
approaches, and while several are extremely plausible, we could not make the details
rigorous. Of course, a much simpler proof may nevertheless exist.

We must show that with high probability vertices $x$ and $y$ at large
distance exist. In doing so we may focus on vertices $x$ and $y$ whose
neighbourhoods satisfy certain restrictions, although if we are too
restrictive, we will not get a good bound.  Before turning to the
graph, let us describe the corresponding restrictions on the branching
process. 
Overall, our aim is to consider the event that a certain `wedge' condition
holds, and $t_{\omega/\eps}>t$, for $t$ near $t_0+t_1$, but
to make our arguments work we need some additional technical conditions.
We start by insisting that the process
$(X_t^+)$ consisting of those particles with infinitely many
descendants has size 1 for a large number of generations, then
bifurcates, and the non-surviving descendants of all the particles up
to this point have died out before very long, in a way to be made
precise. This condition will include an analogue of the weak wedge
condition described in \refSS{lower_fixed}.

For the rest of this section let $\eps=\eps(n)>0$
satisfy $\eps\to 0$ and $\La=\eps^3n\to\infty$.
Set
\[ 
 \omega=\La^{1/6},
\]
and let
\[
 t_1=\floor{\log\omega/\log\la}
\]
and
\[
 t_0=\floor{\log(\eps^3 n)/\log(1/\las)},
\]
as before. (The rounding to integers will always be irrelevant in calculations.)
Later, we shall also consider
\[
t_2=\log(\eps^3 n/\omega^2)/\log \la.
\]

For $r,q=O(1/\eps)$, set $T_0=t_0+r$ and $T_1=t_0+t_1+q$.
Recalling from \eqref{lllas}
that $\log\la, \log(1/\las)\sim\eps$, note that $T_0,T_1=O(\eps^{-1}\log\La)$.
We shall assume that $|r|\le t_0/2$
and that $|r|,|q|\le t_1/10$; these conditions hold for $n$ sufficiently large.

Let $A=A_{r}$ be the event that $|X_{T_0}^+|=1$ and $|X_{T_0+1}^+|=2$.
Then $\Pr(A)=s\Pr(Z_\la=1)^{T_0}\Pr(Z_\la=2)$, where, as before, $Z_\la$ is a Poisson with mean $s\la$ conditioned to be at least $1$. 
{From} \eqref{plus} we have $\Pr(Z_\la=1)=\las$, while from the definition
of $Z_\la$ we have $\Pr(Z_\la=2)/\Pr(Z_\la=1)=(s\la)/2\sim \eps$.
Hence,
\begin{equation}\label{PrA}
 \Pr(A) \sim 2\eps\las^{T_0} \eps\las \sim 2\eps^2\las^{T_0} \sim 2\eps^{-1}n^{-1}\las^r = \Theta(\eps^{-1}n^{-1}).
\end{equation}
When $A$ holds, let $x_i$ denote the unique
particle in $X_i^+$ for $0\le i\le T_0$, and let $y,y'$ be the two particles
in $X_{T_0+1}^+$.

Let $B=B_r$  be the event that $A=A_r$ holds, and the following conditions are satisfied:

(i) \label{Bi}
(the {\em strong wedge condition}) $x_0$ has no children other than $x_1$ and,
for $1\le i<T_0$, no children of $x_i$
other than $x_{i+1}$ or $y,y'$ have descendants in generation $2i$.

(ii) no particles in $X_{T_0+1}$ other than $y$ and $y'$ have descendants
in $X_{T_1'}$, where $T_1'=t_0+\floor{t_1/2}$.

Note that $T_0<T_1'<T_1$. Also,
since $T_1'-T_0=\floor{t_1}/2-r = t_1/2+O(\eps^{-1})$, we have $\eps(T_1'-T_0)\to\infty$.
For the moment we could simply write $T_1$ in place of $T_1'$ in condition (ii), but for the distribution
result in Section~\ref{sec_dist} it is convenient that $T_1'$ does not depend on $q$.

Unfortunately, it takes some effort to examine the effect that condition (i)
has upon the distribution of $t_{\omega/\eps}$, the time the branching
process takes to reach size $\omega/\eps$. (Condition (ii) presents no problems.)
Constructing $\bp_\la$ from $\bp_\la^+$ by adding independent copies of the subcritical process $\bp_{\las}$
starting at each particle, condition (i) says that
for $i<T_0$ the subcritical process started at $x_i$
dies by time $\max\{i,1\}$ (measured from its starting time), and condition (ii) that
for $i\le T_0$ the process
started from $x_i$ dies by time $T_1'-i$. Writing $d_t=1-s_t=\Pr(|X_t^-|=0)$
 for the probability 
that $\bp_{\las}$ dies by time $t$, we thus have
\[
 \Pr(B\mid A) = d_1 \prod_{i=1}^{T_0} d_{\min\{i,T_1'-i\}},
\]
so
\begin{equation}\label{BA}
 d_1 \prod_{i=1}^{\min\{T_0,T_1'/2\}} d_i \ge  \Pr(B\mid A)
 \ge d_1 \prod_{i=1}^\infty d_i \prod_{i=T_1'-T_0}^\infty d_i.
\end{equation}
By Lemma~\ref{lsurv}, as $\eps i\to\infty$ we have $s_i\sim 2\eps\las^i$,
and so $\log(1-s_i)\sim -2\eps\las^i$. 
Since $\eps(T_1'-T_0)\to\infty$, it follows that
\[
 \sum_{i\ge T_1'-T_0} \log(1-s_i) \sim -2\eps\sum_{i\ge T_1'-T_0}\las^i=O\bb{\las^{T_1'-T_0}}=o(1).
\]
Hence, $\prod_{i\ge T_1'-T_0}d_i\sim 1$. 
Similarly, since $\eps\min\{T_0,T_1'/2\}=\eps T_1'/2\to\infty$, we have $\prod_{i\ge \min\{T_0,T_1'/2\}}d_i\sim 1$. 
{From} \eqref{BA} it then follows that
\begin{equation}\label{PrB}
 \Pr(B\mid A) \sim d_1\prod_{i=1}^{T_1'/2}d_i \sim d_1\prod_{i=1}^\infty d_i
\sim d_1 \gamma_0\eps^2 = e^{-\las}\gamma_0\eps^2
  \sim \gamma_0e^{-1}\eps^2,
\end{equation}
using Lemma~\ref{lsurv} to estimate the infinite product.

Let $C$ be the event that $A$ holds, and the particles $y$ and $y'$ each
have at least $\omega'/\eps$ descendants in $X_{T_1}$, where
$\omega'=\sqrt{\omega}=\La^{1/12}$.
(Later we shall need to know that vertices corresponding to $y$ and $y'$ have many `descendants'
at distance $T_1$ from $x_0$; this will ensure that $x_{T_0}$ is in the 2-core.)
By Lemma~\ref{livestart}, applied
with $T_1-(T_0+1)=t_1+q-r-1=t_1+O(1/\eps)$ in place of $t_1$, i.e., with $\la^{t_1+q-r-1}=\Theta(\omega)$
in place of $\omega$, we have 
$$
\Pr(C\mid A)=1-o(1).
$$

We would like to impose the condition that $|X_t|<\omega/\eps$ for $0\le t\le T_1$;
however, for technical reasons we must consider the descendants of $x_{T_0}$
separately from the remaining particles.

Let $D_1$ be the event that $A$ holds, and between them the particles
$y$ and $y'$ have fewer than $(\omega-2\omega')/\eps$ descendants in each set $X_t$, $T_0+1\le t\le T_1$,
noting that $\omega-2\omega'\sim \omega$.
Conditioning on $A$, the trees of descendants of the two particles $y$, $y'$
form independent copies of $\bp_\la$, each conditioned on the event that it survives.
By Lemma~\ref{converged}, \whp\ as soon as the number of descendants of $y$ in $X_{T_0+1+r}$ is large compared to $\eps^{-1}$, it then remains
close to $\tY\la^r$, where $\tY$ has the distribution
of $Y=Y_\la$ conditioned to be positive.
Let $\tY_2$ have the distribution of the sum of
two independent copies of $Y$ each conditioned to be positive. Then it follows that
\[
 \Pr(D_1\mid A) = o(1)+ \Pr\bb{ \tY_2\la^{T_1-T_0-1} < (\omega-2\omega')/\eps}.
\]
Now $\la^{T_1-T_0-1}=\la^{t_1+q-r-1}=\omega\la^{q-r+O(1)}\sim\omega\la^{q-r}$,
and $\omega-2\omega'\sim\omega$, so
\[
 \Pr(D_1\mid A) = o(1)+ \Pr\bb{ \tY_2 < (1+o(1))\la^{r-q}/\eps}
 = o(1)+ \Pr\bb{ s\tY_2 < (2+o(1))e^{\eps(r-q)}},
\]
recalling that $s\sim 2\eps$ and noting that,
since $\eps(r-q)$ is bounded and $\la=1+\eps$, we have $\la^{r-q}\sim \exp(\eps(r-q))$.
In a moment we shall sum over $r$; 
we can evaluate the sum of the corresponding terms above by relating
it to a certain disjoint union of events and using Theorem~\ref{th19}.
While this is aesthetically pleasing, we in fact know the asymptotic distribution of $\tY_2$, 
so we shall just use it.

Recall from Lemma~\ref{YYp} and Corollary~\ref{Ycor} that $sY$
conditioned on $Y>0$ has the distribution of $\YP=\YP_\la$, which 
converges in distribution to an exponential with parameter $1$ as $\eps\to 0$.
It follows that $s\tY_2$ converges in distribution to the sum of two
independent such exponentials,
which has distribution function $\Psi(x)=\int_{y=0}^x e^{-y}(1-e^{-(x-y)})\dd y = 1-(x+1)e^{-x}$.
Thus
\[
 \Pr(D_1\mid A) = \Psi(2e^{r'-q'}) + o(1),
\]
where $r'=\eps r$ and $q'=\eps q$ and we use uniform continuity
to remove the $(1+o(1))$ factor in the argument of $\Psi$.

Since $r'-q'=\Theta(1)$, we thus have
$\Pr(D_1\mid A)=\Theta(1)$, and hence the above equation can be written as
$\Pr(D_1\mid A) \sim \Psi(2e^{r'-q'})$. Since $\Pr(C\mid A)=1-o(1)$, it follows that
$\Pr(C\cap D_1\mid A)\sim \Psi(2e^{r'-q'})$.
Given $A$, the events $B$ and $C\cap D_1$ are independent,
so
\begin{equation}\label{CD1}
 \Pr(C\cap D_1\mid A\cap B)\sim \Psi(2e^{r'-q'}).
\end{equation}

Turning to particles other than the descendants of $y$, $y'$, first let $D_1'$ 
be the event that $A$ holds and, for $T_0\le t\le T_1$, the set $X_t$ contains
at most $\omega'/\eps$ particles that are descendants of $x_{T_0}$ but
not of $y$ or $y'$. Given $A\cap B$, these particles form a copy of $\bp_\la$
starting at $x_{T_0}$ and conditioned to die within $T_1'-T_0$ generations.
This process may be viewed as $\bp_{\las}$ conditioned to die by a certain
time, so its distribution is dominated by that of $\bp_{\las}$.
Since the total expected size of $\bp_{\las}$ is $O(\eps^{-1})$, it follows
that $\Pr((D_1')^\cc\mid A\cap B)=o(1)$.

Let $D_2$ be the event that $A$ holds and, for $0\le t\le T_1$, the set $X_t$ contains
at most $\omega'/\eps$ particles that are not descendants of $x_{T_0}$.
Given $A\cap B$, the tree of particles that are not descendants of $x_{T_0}$
has the distribution of one copy of $\bp_{\las}$ started at each time $t$, $0\le t< T_0$,
conditioned on the various copies of $\bp_{\las}$ dying by various times.
This distribution is dominated by that studied in Lemma~\ref{dthin},
so by Lemma~\ref{dthin} we have $\Pr(D_2^\cc\mid A\cap B)=O(\eps T_1 e^{-\Omega(\omega')})
=o(1)$, recalling that $T_1=O(\eps^{-1}\log\La)$.

Let $D=D_1\cap D_1'\cap D_2$. Since $\Pr((D_1')^\cc\cup D_2^\cc\mid A\cap B)=o(1)$,
from \eqref{CD1}, we have
\begin{equation}\label{CD11}
 \Pr(C\cap D_1\cap D_1'\mid A\cap B)\sim \Psi(2e^{r'-q'})
\end{equation}
and
\[
 \Pr(C\cap D\mid A\cap B)\sim \Psi(2e^{r'-q'}).
\]
Note for later that if $D$ holds, then $|X_t|<\omega/\eps$ for $t\le T_1$. 

Finally, setting $E_{r,q}=A\cap B\cap C\cap D$, and recalling \eqref{PrA} and \eqref{PrB}, 
we have
\[
 \Pr(E_{r,q}) \sim 2\eps^{-1}n^{-1} \las^r \gamma_0e^{-1}\eps^2 \Psi(2e^{r'-q'})
 \sim 2\gamma_0e^{-1}\eps e^{-r'} \Psi(2e^{r'-q'})/n.
\]
Since this estimate holds uniformly in $r$, $q$ with $r,q=O(1/\eps)$,
it also holds uniformly in $r,q$ with $|q|,|r|\le 2M/\eps$, say,
for some function $M=M(n)$ tending to infinity.
For $|q|\le M/\eps$, let $E_q=\bigcup_{-2M/\eps\le r\le 2M/\eps} E_{r,q}$.
For fixed $q$, the events $E_{r,q}$ are disjoint, so we have
\begin{eqnarray*}
 \Pr(E_q) &\sim& 2\gamma_0e^{-1}\eps n^{-1} \sum_{-2M/\eps\le r\le 2M/\eps} e^{-r'} \Psi(2e^{r'-q'}) \\
  &=& 2\gamma_0e^{-1}\eps n^{-1} e^{-q'} \sum_{-2M/\eps-q\le r-q\le 2M/\eps-q} e^{-(r'-q')} \Psi(2e^{r'-q'}).
\end{eqnarray*}

The sum above simplifies considerably, since it corresponds to splitting
a single event according to the time that $(X_t^+)$ first subdivides.
Rather than using this observation, we simply calculate.
Since $\Psi(x)=O(1)$ as $x\to\infty$ and $\Psi(x)=O(x^2)$ as $x\to 0$,
the sum above has exponentially decaying
tails. Recalling that $r'$ and $q'$ simply denote $\eps r$ and $\eps q$, it follows easily that
\[
 \Pr(E_q) \sim 2\gamma_0e^{-1}n^{-1} e^{-\eps q} \int_{-\infty}^\infty e^{-x} \Psi(2e^{x}) \dd x.
\]
A simple computation shows that the integral evaluates to $2$, so
\[
 \Pr(E_q)\sim   4\gamma_0e^{-1}n^{-1}e^{-\eps q} \sim 4\gamma_0e^{-1}n^{-1}\las^q,
\]
uniformly in $|q|\le M/\eps$, provided $M=M(n)$ tends to infinity sufficiently slowly.

Note that the event $E_q$ requires that $y,y'\in X_{T_0+1}^+$, an event depending on an infinite number of generations of the process $\bp_\la$.   To work with the graph, we seek an event depending on a finite number of
generations of $\bp_\la$. Let $F_q$ be the event corresponding to $E_q$ but
depending only on the first $T_1=t_0+t_1+q$ generations. More precisely,
$F_q$ is the event that there are exactly two particles, $y$ and $y'$, say,
in some generation $T_0+1=t_0+r+1$, $-2M/\eps\le r\le 2M/\eps$,
with descendants in generation $T_1'=t_0+\floor{t_1/2}$,
each of these particles
has at least $\omega'/\eps$ descendants in $X_{T_1}$,
$y$ and $y'$ have a common
parent $x_{T_0}$, the equivalent of the strong wedge condition (i) holds,
and $D=D_1\cap D_1'\cap D_2$ holds.
{From} the strong wedge condition,  if $F_q$ holds
then, in the tree obtained from $\bp_\la$ by deleting all descendants of $x_{T_0}$,
the initial particle is the unique particle at maximum distance from $x_{T_0}$.

If $E_q$ holds, then so does $F_q$. Furthermore,
$\Pr(E_q\mid F_q)=1+o(1)$, since for each of $y$ and $y'$, the probability that none
of its at least $\omega'/\eps$
descendants in generation $T_1$ goes on to survive forever is $O((1-s)^{\omega'/\eps}) =o(1)$. Hence,
\[
 \Pr(F_q)\sim \Pr(E_q)\sim 4\gamma_0e^{-1}n^{-1}\las^q.
\]

Let $T$ be a tree of height $t=T_1$ consistent with $F_q$.
Then $t=O(\eps^{-1}\log\La)$, while, since $D$ holds,
each generation contains at most $\omega\eps^{-1}=\omega\La^{-1/3}n^{1/3}=o(n^{1/3})$
vertices. Also,
the total size $|T|$ of $T$ is
\begin{equation}\label{Tsize}
 O(\omega\eps^{-2}\log\La) =O(\omega\La^{-2/3}n^{2/3}\log\La) = o(n^{2/3}),
\end{equation}
and
$\eps|T|^2 = O(\omega^2\eps^{-3}\log^2\La) = O(\omega^2\La^{-1}n\log^2\La)=o(n)$.
Lemma~\ref{spcpl} applies to all such trees, telling us that
\[
 \Pr\bb{\Gat(x)\isom T} \sim \Pr\bb{\Gatm(x)\isom T} \sim \Pr\bb{X_{\le t}\isom T}.
\]

Let $F_q(x)$ denote the event that $\Gl{T_1}(x)$ is a tree satisfying the property
$F_q$, where $T_1=t_0+t_1+q$. Summing over all such trees, we see that
\begin{equation}\label{FE}
 \Pr( F_q(x)) \sim \Pr(F_q) \sim 4\gamma_0e^{-1}n^{-1}\las^q
\end{equation}
uniformly in $q$ such that $|\eps q|\le M$, for some $M\to\infty$.

Let $q_0$ be chosen so that $\eps q_0$ tends to minus
infinity very slowly, and let $\FF(x)=F_{q_0}(x)$.
Let $N$ be the number of vertices $x$ for which $\FF(x)$ holds;
then
\[
 \E N= n\Pr(F_{q_0}(x))\sim 4\gamma_0e^{-1}\las^{q_0} \to\infty.
\]
We are now almost finished: it remains to
use a second moment argument to show that $N$ is whp large, and then to bound
the probability that two vertices satisfying the relevant condition are close.

Given distinct vertices $x$ and $y$ of $G=G(n,\la/n)$,
let $A(x,y)$ be the event that $\FF(x)$ and $\FF(y)$ both hold,
with the trees `witnessing' this being disjoint. For trees $T_1$ and $T_2$ consistent
with $F_{q_0}$, by Lemma~\ref{spcplk} the probability that the relevant
neighbourhoods of $x$ and $y$ are disjoint and isomorphic to $T_1$ and $T_2$
respectively is asymptotically the product of the individual
probabilities. It follows easily that
\begin{equation}\label{Axy}
 \Pr(A(x,y))\sim \Pr(\FF(x))\Pr(\FF(y))  = \Pr(\FF(x))^2.
\end{equation}

At this point, it seems that there should be a simple argument involving `pulling
the trees off the $2$-core and reattaching them randomly'. However, once again,
we did not manage to make such an argument precise in a simple way.

Our next aim is to show that it is very unlikely that $\FF(x)$ and $\FF(y)$ hold
and the trees witnessing these events overlap.
Recall that if $\FF(x)$ holds, then there is a unique `first' vertex in the neighbourhoods
of $x$ with two children with descendants in generation
$t_0+t_1+q_0$. Let $x'$ denote this
vertex. Since the two children of $x'$ each have at least $\omega'/\eps=\La^{1/12}/\eps$ descendants
in generation $t_0+t_1+q_0$, with probability at least $1-o(\La^{-100})$, say,
their neighbourhoods continue to grow, and eventually meet, in which case $x'$ is
in the 2-core.
Let $\tF(x)$ be the event that $\FF(x)$ holds and $x'$ is in the 2-core,
so $\Pr(\tF(x))\sim \Pr(\FF(x))$. Also,
let $\gba$ be the `global bad event' that there is some vertex $x$ such that $\FF(x)$
holds but $x'$ is not in the 2-core. Then
\begin{equation}\label{B1small}
 \Pr(\gba)\le n\Pr(\FF(x))o(\La^{-100}) =o( \las^{q_0}\La^{-100}) =o(1),
\end{equation}
assuming, as we may, that $\eps q_0\ge -\log\log\La$, say.

Similarly, if $\FF(x)$ and $\FF(y)$ hold, then it is very likely that
$x$ and $y$ are in the same component. Writing $\gbb$ for the event
that there are $x$ and $y$ in different components such that $\FF(x)$ and $\FF(y)$ hold,
we have
\begin{equation}\label{B2small}
 \Pr(\gbb)=o(1).
\end{equation}

For our second moment bound, we will study $\tN$, the number of vertices $x$
such that $\tF(x)$ holds. Note that whp $\tN$ is equal to $N$, since $\gba$ has
probability $o(1)$. Also,
\[
 \E \tN = n\Pr(\tF(x)) \sim n\Pr(\FF(x)) =\E N \sim 4\gamma_0e^{-1}\las^{q_0}\to\infty.
\]

Let $\tA(x,y)$ denote the event that $\tF(x)$ and $\tF(y)$ hold,
with the trees witnessing $\FF(x)$ and $\FF(y)$ disjoint. If $\tA(x,y)$
holds, then so does $A(x,y)$. On the other hand, 
continuing to explore as before, we see that given $A(x,y)$, the
vertices $x'$ and $y'$ are very likely to be in the 2-core, so
\begin{equation}\label{tAxy}
 \Pr( \tA(x,y) ) \sim \Pr(A(x,y)) \sim \Pr(\FF(x))^2 \sim \Pr(\tF(x))^2.
\end{equation}

It remains to consider the case of overlapping trees.

We defined $\FF(x)$ in such a way that if $\FF(x)$ holds,
then $x'$ together with the component of $G-x'$ containing $x$ forms a tree,
in which $x$ is the unique vertex
at maximal distance from $x'$. If $\tF(x)$ holds,
so $x'$ is in the 2-core, then $x$ is the unique vertex of this tree
at maximal distance from the $2$-core.
Let $T_x$ denote this tree, or, in general, the tree component containing $x$ if we
delete from $G$ all edges lying in the $2$-core.
If $\tF(x)$ and $\tF(y)$ both hold, then from this uniqueness property,
the trees $T_x$ and $T_y$ are disjoint, except possibly at $x'$ and $y'$:
they are two distinct trees attached to the 2-core.

Let $\tB(x,y)$ be the event that $\tF(x)\cap \tF(y)$ holds and
the trees $T_x$ and $T_y$ are disjoint (except possibly at $x'$ and $y'$), but
the trees witnessing $\FF(x)$ and $\FF(y)$ overlap.
{From} the remarks above, for $x\ne y$,
\begin{equation}\label{tAtB}
 \tF(x)\cap \tF(y) = \tA(x,y) \cup \tB(x,y).
\end{equation}
To bound $\Pr(\tB(x,y))$, we first test whether
$\FF(x)$ ({\em not} $\tF(x)$) holds, in a way that first uncovers the tree $T_x$.
Roughly speaking, we would like to show that the number of trees $T_x$ hanging
off the $2$-core is well behaved (i.e., its second moment is not too large). Then we could
say that the attachment points to the $2$-core are uniformly distributed, so it's
unlikely that there are two trees attached to close points. The problem is
that we need independence to get the second moment bound, and we do not have this,
as we can't tell in advance when we have reached the 2-core and should stop exploring the tree
from $x$. To get around this, we choose a stopping vertex in advance.

Given distinct vertices $x$ and $\bx$,
let $\FF(x;\bx)$ be the event that $\FF(x)$ holds, with the division vertex $x'$ equal to $\bx$.
Note that $\FF(x)$ is the disjoint union of the events $\FF(x;\bx)$,
$\bx\in V(G)\setminus\{x\}$, all of which are equally likely. Thus
\begin{equation}\label{FFxbx}
 \Pr(\FF(x;\bx)) = (n-1)^{-1} \Pr(\FF(x)) \sim n^{-1}\Pr(\FF(x)).
\end{equation}
Let $T(x;\bx)$ be the event that $\bx$ together
with the component of $G-\bx$ containing $x$ forms a tree
consistent with $\FF(x;\bx)$.
In other words,
$T(x;\bx)$ is the event that the part of $G$ that we can reach from $x$ if we do
not allow ourselves to pass through $\bx$ is one of a certain set of trees.
Note that we do not insist
that $\bx$ is in fact in the 2-core, and that if $\FF(x;\bx)$ holds then $T(x;\bx)$ must hold.

Crucially, we may test whether $T(x;\bx)$ holds by exploring the neighbourhoods
of $x$ in the usual way, except that if we reach $\bx$ at some point, we do not test
for edges from $\bx$ to unseen vertices. (Since we require the relevant neighbourhood
to be a tree, we do test for edges between all pairs of reached vertices.)
Also, given $T(x;\bx)$, we may test whether $\FF(x;\bx)$ holds by continuing to explore
from $\bx$; roughly speaking, the property required of this further exploration is captured
by $C\cap D_1\cap D_1'$ above (this was the reason for `splitting off' $D_1'$ from $D_2$),
and has probability essentially $\Theta(\eps^2)$.

More precisely, suppose that $T(x;\bx)$ holds and let us condition
on the particular tree $T_x$ revealed by the exploration so far. Let $V'=V(G)\setminus V(T_x)\cup\{\bx\}$.
Then we have not yet examined any edges inside $V'$, and the only edges outside
$V'$ are those of $T_x$. Since $T_x$ is required to be consistent with $F(x;\bx)$,
we know that $d(x,x_0)=t_0+r$ for some $r$ with $|r|\le 2M/\eps$,
and, from \eqref{Tsize}, that $T_x$ contains $o(n^{2/3})$ vertices. 

Now (recalling that $F(x)=F_{q_0}(x)$), the event
$F(x;\bx)$ holds if and only if the following conditions
are satisfied as we explore a further $t=T_1-(t_0+r)=t_1+q_0-r$
steps from $\bx$ in $G[V']$:  (i) the graph we uncover is a tree,
(ii) there are exactly two vertices ($y$ and $y'$) in $\Gamma_1(\bx)$
with `descendants' in $\Gamma_{T_1'-(t_0+r)}(\bx)$, where $T_1'=t_0+\floor{t_1/2}$, (iii) these two vertices
each have at least $\omega'/\eps$ descendants in $\Gamma_t(\bx)$,
(iv) between them, $y$ and $y'$ have at most $(\omega-2\omega')/\eps$
descendants in each $\Gamma_{t'}(\bx)$, $t'\le t$,
and (v) the neighbours of $\bx$ other than $y$ and $y'$ have in total
at most $\omega'/\eps$ neighbours in each of these sets.
Indeed, (i) and (ii) together with the fact that $T_x$ is consistent
with $F(x;\bx)$ ensure that the event corresponding to $A\cap B\cap D_2$
in the definition of $F=F_{q_0}$ holds, (iii) ensures that $C$ holds,
(iv) that $D_1$ holds, and (v) $D_1'$.

Arguing as for \eqref{FE}, we can approximate the probability
of these conditions holding by that of the corresponding
branching process event (the conditions ensure that only $o(n^{2/3})$ vertices
are involved in total). Then we may consider the infinite version
of the branching process event, differing only in that we assume that
$y$ and $y'$ are in $X_1^+$. Now we require that $|X_1^+|=2$; since
$|X_1^+|\sim \Po(s\la)$, this has probability $\Theta(\eps^2)$. Given
this, in the branching process the remaining conditions corresponding
to (iii), (iv) and (v) are {\em exactly} the conditions $C$, $D_1$ and
$D_1'$ considered earlier, except that now $\bx$ plays the role
of the initial particle $x_0$, and all generation numbers
are offset by $t_0+r$. In particular, the conditional
probability of these events is exactly the probability
$\Pr(C\cap D_1\cap D_1' \mid A\cap B)$ evaluated in \eqref{CD11},
with $r'=\eps r$ and $q'=\eps q_0$.

Let $\psi=\psi(n)$ be a function tending infinity to arbitrarily slowly
(more slowly than the reciprocal  of the implicit function in the $o(\cdot)$ notation 
in \eqref{md}),
and let us write $f=\Tht(g)$ if $f/g=\psi^{O(1)}$.
Taking $M(n)$ to tend to infinity sufficiently slowly, from the comments
above and \eqref{CD11}, we see that
\begin{equation}\nonumber
 \Pr\bb{\FF(x;\bx)\mid T(x;\bx)} = \Tht(\eps^2)
\end{equation}
whenever $T(x;\bx)$ holds. From \eqref{FFxbx} it follows that for all $\bx\ne x$ we have
\[
 \Pr(T(x;\bx)) = \Tht\bb{\eps^{-2}n^{-1}\Pr(\FF(x))} = \Tht(\eps^{-2}n^{-2}),
\]
recalling (from \eqref{FE}) that $\Pr(\FF(x))=\Theta(n^{-1}\las^{q_0})=\Tht(n^{-1})$.

Given $x\ne y$ and $\bx$, $\by$, let $B'(x,y,\bx,\by)$ be the event
that $T(x,\bx)\cap T(y,\by)$ holds, with the trees $T_x$ and $T_y$ edge disjoint.
Note that if this event holds, then $\bx,\by\notin\{x,y\}$.
We may test whether $B'(x,y,\bx,\by)$ holds by exploring from $x$ and
$y$ respectively (with the explorations modified at $\bx$ and $\by$),
and the two explorations cannot `help' each other.
Arguing as for $\eqref{Axy}$ above, using Lemma~\ref{spcplk}, it follows that
\[
 \Pr(B'(x,y,\bx,\by))\sim \Pr(T(x,\bx))\Pr(T(y,\by)) =\Tht(\eps^{-4}n^{-4})
\]
for all $\bx$, $\by\notin\{x,y\}$; the probability is $0$ if $\bx$ or $\by\in \{x,y\}$.

Fix vertices $x\ne y$, and let $\rx$ and $\ry$ be chosen independently and uniformly
at random from $V(G)$. Note that
\begin{equation}\label{TT}
 \Pr(B'(x,y,\rx,\ry))=\Tht(\eps^{-4}n^{-4}).
\end{equation}

Let us condition on $B'(x,y,\rx,\ry)$. Moreover, we condition
on $V_x=V(T_x)\setminus \{\rx\}$, on $V_y=V(T_y)\setminus \{\ry\}$
and on the structure of the trees $T_x$ and $T_y$, but {\em not} on $\rx$
and $\ry$. Given this information, $\rx$ and $\ry$ are independent and
uniform from $V'=V(G)\setminus (V_x\cup V_y)$. Indeed, the given information
says that certain trees $T_x$ and $T_y$ are attached to $\rx,\ry\in V'$.
Each tree is equally likely to be attached to any vertex of $V'$,
so, given this, the attachment vertices are uniform on $V'$.

The event we have conditioned on does not depend on the edges in $V'$.
Hence, the conditional distribution of $G[V']$ is that of $G'=G(n',\la/n)$,
where $n'=n-|T_x|+1-|T_y|+1$. From the definition of $F_{q_0}$, we
have $|T_x|,|T_y|=o(n^{2/3})$, so $n'=n-o(n^{2/3})$.
The edge probabilities in $G'$ are thus $\la'/n'$ where
\[
 \la'=\la n'/n = (1+\eps)(n-o(n^{2/3}))/n = 1+\eps -o(n^{-1/3}) = 1+\eps',
\]
with $\eps'\sim \eps$.

Let $\tB'(x,y)$ be the event that $\tB(x,y)$ holds, and $\rx=x'$, $\ry=y'$,
so
\begin{equation}\label{tB'}
 \Pr(\tB'(x,y)) = n^{-2} \Pr(\tB(x,y)).
\end{equation}
If $\tB'(x,y)$ holds, then so does $B'(x,y,\rx,\ry)$.
Furthermore, $\rx$ and $\ry$ must be in the 2-core of $G$, which
is the same as the 2-core $U$ of $G'$. Also, $\rx$ and $\ry$ must be {\em close},
i.e., within distance $d=2t_1+4M(n)/\eps\sim 2t_1$.

{From} the remarks above, we may bound $\Pr\bb{\tB'(x,y) \mid B'(x,y,\rx,\ry)}$
by the conditional probability (given the trees $T_x$, $T_y$ etc but not $\rx$, $\ry$)
that $\rx$ and $\ry$ are close in $U$, and hence by
\[
 |G'|^{-2} \E M_d(G') \sim n^{-2} \E M_d(G'),
\]
where $M_d(G')$ is the number of close pairs in $U$, and the expectation
is over the random graph $G'$.

Now $G'$ has the distribution of $G(n',\la'/n')$, with $\la'=1+\eps'$ and
$\eps'\sim \eps$. Also, $d\sim 2t_1\sim 2\log\omega/\eps=\log(\eps^3 n)/(3\eps)
\sim 3^{-1} \log((\eps')^3n')/\eps'$.
By Lemma~\ref{l2c}, we thus have $\E M_d(G')=o(\eps^4n^2)$.
Taking our slowly growing function $\psi(n)$ small enough, 
the expectation is smaller than $\eps^4n^2$ by at least a factor $e^\psi$, say.
It follows that
\[
 \Pr\bb{\tB'(x,y) \mid B'(x,y,\rx,\ry)} \le \Theta(\eps^4e^{-\psi}).
\]
Using \eqref{TT} it follows that $\Pr(\tB'(x,y))=\Ot(n^{-4}e^{-\psi})=o(n^{-4})$,
and hence, from \eqref{tB'}, we have
\begin{equation}\label{tBsmall}
 \Pr(\tB(x,y))=o(n^{-2}).
\end{equation}
It follows that whp there are no pairs $(x,y)$ for which $\tB(x,y)$ holds.
Recalling \eqref{B1small}, \eqref{B2small} and \eqref{tAtB}, and noting
that $\tA(x,y)$ trivially implies $A(x,y)$, we see that whp every pair
of vertices $x\ne y$ for which $\FF(x)\cap\FF(y)$ holds has the properties
\begin{equation}\label{concl}
 d(x,y)<\infty \hbox{\quad and \quad} A(x,y).
\end{equation}

Using \eqref{tAxy}, \eqref{tAtB} and \eqref{tBsmall}, and recalling that $\tN$ denotes the number
of vertices $x$ such that $\tF(x)$ holds, we have
\[
 \E(\tN(\tN-1)) = \sum_x\sum_{y\ne x}\bb{ \Pr(\tA(x,y))+\Pr(\tB(x,y))}
  = (1+o(1))(\E\tN)^2 +o(1).
\]

Since $\E\tN\to\infty$, it follows that $\E \tN^2\sim (\E\tN)^2$,
and hence that $\tN$ is concentrated about its
mean. Since $\E N\sim \E\tN$, and $\tN$ and $N$ are whp equal,
we thus have $N$ concentrated about its mean also, where $N$ is the number
of $x$ such that $\FF(x)$ holds.

Finally, the end of the proof is as in Section~\ref{lower_fixed}.
Set $t_2=\log(\eps^3 n/\omega^2)/\log \la$,
let $K\eps \to \infty$ very slowly,
let $N$ be the number of vertices $x$ for which $\FF(x)$ holds,
and let $M$ be the number of pairs $x$, $y$ for which $A(x,y)$ holds (i.e., $\FF(x)$ and $\FF(y)$ hold
disjointly) but $d(x,y)\le d$, where
\begin{eqnarray*}
 d = 2(t_0+t_1+q_0)+t_2-K &=& 2\frac{\log(\eps^3n)}{\log(1/\las)} + 2\frac{\log\omega}{\log\la}
 + \frac{\log(\eps^3 n/\omega^2)}{\log \la} + O(1)+2q_0-K \\
 &=& \frac{\log(\eps^3 n)}{\log\la}+ 2\frac{\log(\eps^3n)}{\log(1/\las)} + O(1)+2q_0-K.
\end{eqnarray*}
Given that $\FF(x)$ and $\FF(y)$ hold disjointly, the $(t_0+t_1+q_0)$-neighbourhoods of $x$ and $y$
each contain at most $\omega/\eps$ vertices. Exploring from $x$ and $y$ in the obvious way, 
the rest of the graph is `unseen', and the expected number of paths of length at
most $t_2-K$ joining one neighbourhood to the other is at most
\begin{multline*}
 (\omega/\eps)^2 \sum_{k\le t_2-K} n^{k-1} (\la/n)^k = \omega^2\eps^{-2}n^{-1}\sum_{k\le t_2-K}\la^k \\
  = O(\omega^2\eps^{-3}n^{-1}\la^{t_2-K}) = O(\la^{-K})=o(1).
\end{multline*}
(Here it is important that we work with $\FF(x)$ and not $\tF(x)$.)
Hence, the conditional probability that $d(x,y)\le d$ is $o(1)$,
so $\E M=o\bb{n(n-1)\Pr(A(x,y))}=o((\E N)^2)$, using \eqref{Axy}.
It follows that whp there are at least $\E N/2\ge 2$ vertices $x$ for which $\FF(x)$
holds, but at most $(\E N)^2/5$ pairs of vertices with $A(x,y)$ holding but
$d(x,y)\le d$. Using \eqref{concl}, it follows
that $\diam(G)\le d$ whp. Recalling that both $-q_0\eps>0$ and $K\eps$ may be taken to tend
to infinity arbitrarily slowly, this completes the proof of the lower bound in Theorem~\ref{th_eps}.

\subsection{The upper bound}\label{ss_up}

Throughout we fix a function $\eps=\eps(n)>0$
satisfying $\eps\to 0$ and $\eps^3n\to\infty$.
As before we shall often
write $\La$ for $\eps^3 n$, and set%
\[
 \omega = \La^{1/6}.
\]
As before, let $t_0=\log(\eps^3 n)/\log(1/\las)$, $t_1=\log\omega/\log\la$, and
$t_2=\log(\eps^3 n/\omega^2)/\log\la$; we ignore rounding to integers,
which makes no essential difference in our calculations.

Let $K=K(n)$ be such that $K\eps\to \infty$, and let
\begin{equation}\label{d0def}
 d_0=\log(\eps^3n)/\log\la+2\log(\eps^3n)/\log(1/\las) = 2t_0+2t_1+t_2,
\end{equation}
so our aim is
to prove that $\diam(G)\le d_0+K$ holds whp, and we may assume if we
like that $K\eps$ grows slower than any given function of $n$ tending
to infinity.
The basic idea is to simply estimate the expected number of pairs $x$, $y$
with $d(x,y)\ge d_0+K$.
However, the calculations in the previous sections imply that on its own,
this will not work; the expectation turns out to be roughly
$\eps^{-4}$ if $K\eps$ grows slowly. The reason is that, given
that a tree hanging off the 2-core has height at least $h$,
the expected number of vertices it contains at distance at least
$h$ from the 2-core is of order $\eps^{-2}$. 

To get around this, we need to impose a version of the wedge condition; we
should like to consider only vertices $x$ that are at maximal distance
from the 2-core in their tree. (Note that we cannot insist that $x$
is the unique vertex at this distance in its tree, as we did before.)
This suggests the {\em weak wedge condition}: roughly speaking,
we should like any `side branches' starting
from $\Ga_t(x)$ to have height at most $t$, one more than the height allowed
in the strong wedge condition.
This is all very well if the neighbourhoods of $x$ out to the relevant
distance form a tree, but in the upper bound we must consider {\em all}
vertices $x$, so we must modify the condition. Unfortunately, most of the
work in this section will be needed to show that we can rule
out various unlikely cases (such as the diameter coming from a pair $x$, $y$
where $x$ is close to a short cycle).

Suppose that $x$ and $y$ are a pair of vertices at maximal distance,
pick any $t\le d(x,y)$,
and consider any shortest path $P$
from $x$ to $y$.
Then, tracing $P$ backwards from $y$ to $x$, we first
meet $\Gat(x)$ at some vertex $v_t\in \Ga_t(x)$. Since $P$ is shortest, $d(v_t,y)=d(x,y)-t$,
so continuing from $v_t$ to $x$ along the unique path in $\Gatm(x)$ joining these vertices,
we find another shortest path $P'$ from $x$ to $y$ that starts with $v_0v_1v_2\cdots v_t$,
a path in $\Gatm(x)$.
We shall split the tree $\Gatm(x)$ into the {\em trunk} $T$, consisting of all vertices
with descendants in $\Ga_t(x)$, plus one {\em side branch} $B_v$ for each $v\in T$.
Here $B_v$ consists of $v$ together with all its descendants in $\Gatm(x)$
that are not descendants of another trunk vertex. (This corresponds roughly
to the decomposition of $\bp_\la$ into $\bp_\la^+$ together with independent
copies of $\bp_{\las}$; the difference is that we only consider
finitely many generations, as we must in the graph.)

Of course each $v_i$ is a trunk vertex.
The key observation is that for $0\le i\le t$, the side branch
$B_{v_i}$ is either {\em short}, i.e., has height
at most $i$, or is {\em reattached}, i.e., $B_{v_i}-v_i$ meets an edge of
$\Gat(x)\setminus \Gatm(x)$. Otherwise, let $w$ be a vertex of $B_{v_i}$ at maximum
distance from $v_i$ in $B_{v_i}$. Since $B_{v_i}$ is not reattached,
any path from $w$ to $y$ must pass via
$v_i$. Since $B_{v_i}$ is not short,
the total length of such a path exceeds $d(x,y)$, contradicting the 
assumption that $x$ and $y$ are
at maximum distance.

Given $1\le d\le t$ and a vertex $x$,
let $S$ denote the set of vertices of $\Ga_d(x)$ that have one or more descendants
in $\Ga_t(x)$, in the tree $\Gatm(x)$.
We say that $x$ is {\em $(d,t)$-acceptable}
if there is a vertex $v\in S$
such that every side branch in $\Gatm(x)$ of the path $x=v_0v_1\cdots v_d=v$ is either
short or reattached. From the observation above,
if $x$ and $y$ are at maximal distance,
then $x$ and $y$ must be $(d,t)$-acceptable for any $1\le d<t<d(x,y)$.

Set $h=\eps^{-1}\log\log\La$, say. (Here $\eps^{-1}$ times any slowly-enough growing function will do.)
For $t>h$, let $A_t=A_t(x)$ be the event that $x$ is $(h,t)$-acceptable, and
let $B_t=B_t(x)$ be the event
that $0<|\Ga_r(x)|< \omega/\eps$ holds for $0\le r\le t$.
The following lemma will play a key role in our estimates.

\begin{lemma}\label{A2}
Under the assumptions of Theorem~\ref{th_eps} we have
\begin{equation}\label{ubaim}
 \Pr(A_t\cap B_t) \le (1+o(1)) 4\gamma_0 \eps^3\las^{t-t_1}
\end{equation}
uniformly in all $t$ in the range
$t_1+3h\le t\le 10\eps^{-1}\log\La$,
where $\gamma_0>0$ is the constant appearing in Lemma~\ref{lsurv}.
\end{lemma}
Recall that, by the second part of Theorem~\ref{th19},
in  the branching process $\bp_\la$ we have
\begin{equation}\label{Bsim}
 \Pr\bb{ 0<|X_r|< \omega/\eps,\, r=0\ldots t} \sim 4\eps\las^{t-t_1}
\end{equation}
for any $t\le 10\eps^{-1}\log \La$ such that $\eps(t-t_1)\to\infty$;
by Lemma~\ref{spcpl}, this carries over to the graph. Thus
Lemma~\ref{A2} says essentially that
the conditional probability that our modified wedge condition holds
is asymptotically $\gamma_0\eps^2$.
We postpone the proof of the lemma for the moment.

Unfortunately, to handle the case when $\La=\eps^3 n$ grows slowly, it turns out that we need
two further lemmas. The first is a very simple observation; once one thinks
of the lemma, it is very easy to prove. We thought of it after
seeing the preprint of Ding, Kim, Lubetzky and Peres~\cite{DKLP_diam}.
\begin{lemma}\label{noshort}
Let $L=L(n)$ be any function satisfying $L=o(1/\eps)$. Then, under the conditions
of Theorem~\ref{th_eps}, whp the giant component of $G(n,\la/n)$ contains
no cycle of length at most $L$.
\end{lemma}
\begin{proof}
Fix $3\le \ell\le L$ and a sequence $v_1,\ldots,v_\ell$ of distinct vertices
of $G=G(n,\la/n)$.
Let $E$ be the event that this sequence forms
a cycle, i.e., that the edges $v_1v_2$, $v_2v_3, \ldots, v_\ell v_1$ are all
present, so $\Pr(E)=\la^\ell/n^\ell\sim n^{-\ell}$.
Let $F$ be the event that $E$ holds and this cycle is in the giant
component. First testing whether $E$ holds, and then exploring outwards
from this cycle, by comparison with the branching process as usual
we see that $\Pr(F\mid E)=O(\ell s)=O(\eps \ell)$, with the implicit constant universal.
Hence $\Pr(F)=O(\eps\ell n^{-\ell})$. Summing over all at most $n^{\ell}$ sequences,
and dividing by $2\ell$ to avoid overcounting, the expected number of $\ell$-cycles
in the giant component is thus $O(\eps)$. Finally summing over $\ell\le L$ and 
using Markov's inequality gives the result.
\end{proof}

\begin{lemma}\label{A1}
Let $\psi=\psi(n)$ be some function of $n$ tending to infinity slowly,
with $\psi=O(\La^{1/8})$ and $\psi=o(\eps^{-1/10})$.
Let $A^*(x)$ denote the event that $t_{\omega/\eps}(x)$ is defined,
$x$ is $(d,t)$-acceptable for all $1\le d<t\le t_{\omega/\eps}(x)$,
and $\Gat(x)$ is a tree for $t=\min\{t_{\omega/\eps}(x),\eps^{-1}/\psi\}$.
Under the assumptions of Theorem~\ref{th_eps} we have
\[
 \Pr(A^*(x))=O(\eps^3\psi^8) = O(\eps^3\La).
\]
\end{lemma}
(It is likely that the probability estimated above is $O(\eps^3)$, at least if the quantity
$\eps^{-1}/\psi$ in the definition of $t$ is replaced by a small constant times $\eps^{-1}$,
but even a bound such as $O(\La^{100}\eps^3)$ would be more than enough for us here.)

Assuming Lemmas~\ref{A2} and~\ref{A1} for the moment, it is not hard
to complete the proof of Theorem~\ref{th_eps},
calculating as in Section~\ref{sec_const}, by summing the
expected number of pairs $x$, $y$ with $t_{\omega/\eps}$ in certain ranges and 
both having acceptable neighbourhoods.

\begin{proof}[Proof of Theorem~\ref{th_eps}.]
Let $d_0$ be defined by \eqref{d0def}, and let $K=K(n)$ be such that $K\eps\to\infty$
and $K\le \eps^{-1}\log\log\La$, say. Our aim is to show that whp there is no
pair $(x,y)$ of vertices in the same component with $d(x,y)\ge d_0+K$.
In the light of \L uczak's bound~\eqref{Luczak} from~\cite{Luczak:diam},
and a standard duality argument, we need only consider the giant component. 
(In fact, {\L}uczak and Seierstad~\cite{LuczakSeierstad} have shown that in the random
graph process, whp, for all densities in the range considered here, the diameter is realized
by the giant component.)

Let us say that a vertex $x$ is {\em tree-like} if $\Gl{\eps^{-1}/\psi}(x)$ is a tree.
By Lemma~\ref{noshort}, whp {\em every} vertex in the giant component
is tree-like, so it suffices to consider pairs $(x,y)$ in which
both $x$ and $y$ have this property.

As noted above,
in any pair $(x,y)$ at maximal distance greater than $d_0$, both $x$ and $y$
must be $(d,t)$-acceptable for any $d<t<d_0$.
Set
\[
 t^+ = t_0+t_1+ K/3,
\]
noting that $t^+<d_0/2$ and $t^+>h=\eps^{-1}\log\log\La$.
By Lemma~\ref{A2}, for any vertex $x$ we have
\[
 \Pr(A_{t^+}(x)\cap B_{t^+}(x)) \le (4+o(1))\gamma_0\eps^3\las^{t^+-t_1}=O(n^{-1}\las^{K/3})=o(n^{-1}),
\]
so whp there is no vertex for which this event holds.
Let $A'(x)$ be the event that $t_{\omega/\eps}(x)$ is defined and at most $t^+$, and
$A^*(x)$ holds, where $A^*(x)$ is defined in Lemma~\ref{A1}.
Let us call $(x,y)$ a {\em regular far pair} if $d(x,y)>d_0+K$, and the events $A'(x)$ and $A'(y)$ hold.
Then from the comments above it suffices to prove that whp there are no regular far pairs.

We may test whether $A'(x)$ holds by uncovering successive neighbourhoods of $x$, stopping
at the first (if there is one) with at least $\omega/\eps$ vertices, and then
testing for acceptability and the tree condition, or stopping after $t^+$ steps if there
is no such neighbourhood (in which case $A'(x)$ does not hold).
By definition, each neighbourhood other than the last has at most $\omega/\eps$
vertices. By Lemma~\ref{nojg}, the probability that we
find more than $2\omega/\eps$ vertices in the last neighbourhood is at most
$\exp(-\Omega(\omega/\eps))=\exp(-\Omega(\eps^{-1/2}n^{1/6}))=o(n^{-100})$. Ignoring this event,
testing $A'(x)$ involves uncovering
\[
 O(t^+\omega/\eps)=O(\omega\log\La/\eps^2)=O(\La^{1/3}\eps^{-2})=O(\La^{1/3}\La^{-2/3}n^{2/3}) =o(n^{2/3})
\]
vertices.
Also, we uncover $O(\omega/\eps)=o(n^{1/3})$ vertices in each generation.
Noting that $\eps(t^+\omega/\eps)^2=O(\La^{2/3}\eps^{-3})=O(n\La^{-1/3})=o(n)$,
Lemmas~\ref{spcpl} and~\ref{spcplk} apply to the corresponding trees.
By Lemma~\ref{spcplk} it follows that for $x$ and $y$ distinct,
\begin{multline}
 \Pr\bb{A'(x)\cap A'(y)\cap \{d(x,y)>t_{\omega/\eps}(x)+t_{\omega/\eps}(y)\}} \\
 = (1+o(1))\Pr(A'(x))\Pr(A'(y))+o(n^{-100}) = O(\La^2\eps^6),\label{A'xy}
\end{multline}
using $\Pr(A'(x))\le \Pr(A^*(x))$ and Lemma~\ref{A1} for the final bound. (In fact, we have glossed
over something here: using Lemma~\ref{spcplk} shows that the events that the
explorations from $x$ and $y$ give certain trees consistent with $A'(x)$ and $A'(y)$ are asymptotically
independent.
However, the events $A'(z)$, $z=x,y$, depend not just on the trees,
but also on any additional edges between the trees' vertices.
Since these are present independently with probability $\la/n$,
asymptotic independence of the trees gives asymptotic
independence of the entire neighbourhoods.)

Suppose we have explored the neighbourhoods of $x$ and $y$ and found that the event described
above holds, i.e., $A'(x)$ and $A'(y)$ hold disjointly.
Then Lemma~\ref{meet} applies, and the conditional probability
that the explorations do not meet within $t_2+2\eps^{-1}\log\log\La$ further steps
is $\exp\bb{-(1+o(1))(\log \La)^{2+o(1)}}+O(\La^{-10})=O(\La^{-10})$.
Summing over choices for $x$ and $y$, we see that the expected number of regular far
pairs with $d(x,y)\ge t_{\omega/\eps}(x)+t_{\omega/\eps}(y)+t_2+2\eps^{-1}\log\log\La$
is $O(n^2\La^2\eps^6\La^{-10})=O(\La^{-6})=o(1)$. Hence, whp there are no such pairs.

Set
\[
 t^- = t_0+t_1-2\eps^{-1}\log\log\La,
\]
noting that whp every vertex $x$ in a regular far pair satisfies
\begin{equation}\label{rf2}
 t_{\omega/\eps}(x)\ge d_0+K-(t_2+2\eps^{-1}\log\log\La) -t^+ \ge t_0+t_1-2\eps^{-1}\log\log\La = t^-.
\end{equation}
This value is large enough that Lemma~\ref{A2} applies.

(Let us remark that if $\La\ge (\log n)^{20}$, say, then the argument
above simplifies: we may replace $2\eps^{-1}\log\log\La$ by $2\eps^{-1}\log\log n$, and the error
probability given by Lemma~\ref{meet} is then $o(n^{-100})$ (using the
middle expression in~\eqref{jb}), so there is no
need to check acceptability to conclude the equivalent of \eqref{rf2}. In particular,
there is no need for Lemma~\ref{A1} in this case at all.)

For distinct vertices $x$ and $y$ and integers $t^-\le t,t'\le t^+$, let
\[
 E_{x,y,t,t'} = A_{[t]}(x)\cap \{ t_{\omega/\eps}(x)=t \} \cap A_{[t']}(y) \cap \{t_{\omega/\eps}(y) =t' \}
 \cap \{ d(x,y)\ge d_0+K \},
\]
where $[t]$ denotes the largest multiple of $\floor{1/\eps}$ that is strictly smaller than $t$.
{From} the comments above, to prove that $\diam(G)\le d_0+K$ holds whp it suffices
to prove that whp none of the events $E_{x,y,t,t'}$ holds. (Here we may impose whatever
acceptability conditions we like: the reason for choosing exactly $A_{[t]}(x)$ will become
clear in a moment.)

Using Lemma~\ref{spcplk} as above,
the probability that $E_1=A_{[t]}(x)\cap \{t_{\omega/\eps}(x)=t\}$ and
$E_2=A_{[t']}(y)\cap \{t_{\omega/\eps}(y)=t'\}$ hold with disjoint witnesses
is asymptotically $\Pr(E_1)\Pr(E_2)$.
Noting that $d_0+K-t-t'-t_2$ is within $5\eps^{-1}\log\log\La=o(t_2)$ of
$0$ and is hence at least $-t_2/2$,
given that $E_1$ and $E_2$ hold disjointly, Lemma~\ref{meet} tells us
that the probability that $d(x,y)\ge d_0+K$ is $\exp(-(1+o(1))\la^{d_0+K-t-t'-t_2})+O(\La^{-10})$.

Let $U=\bigcup_{x\ne y,\, t^-\le t,t'\le t^+} E_{x,y,t,t'}$. Then, writing $A\lesssim B$ for $A\le (1+o(1))B$, 
\begin{multline*}
 \Pr(U) \lesssim n^2 \sum_{t=t^-}^{t^+} \sum_{t'=t^-}^{t^+} \Pr\bb{A_{[t]}(x)\cap \{t_{\omega/\eps}(x)=t\} }
 \Pr\bb{A_{[t']}(x)\cap \{t_{\omega/\eps}(x)=t'\} } \\
 \left(\exp\bb{-(1+o(1))\la^{d_0+K-t-t'-t_2}} +O(\La^{-10})\right).
\end{multline*}

Grouping the sums into blocks of size $k=\floor{1/\eps}$, and noting that if $r$ is a multiple of $k$ then
\[
 \sum_{t=r+1}^{r+k}\Pr\bb{A_{[t]}(x)\cap \{t_{\omega/\eps}(x)=t\} } =
  \Pr\bb{A_r(x)\cap \{r< t_{\omega/\eps}(x)\le r+k\} } \le \Pr(A_r(x)\cap B_r(x)),
\]
we have
\begin{multline*}
 \Pr(U) \lesssim n^2 \sum_{t^--k\le t\le t^+}'\  \sum_{t^--k\le t'\le t^+}'
  \Pr(A_t(x)\cap B_t(x)) \Pr(A_{t'}(x)\cap B_{t'}(x)) \\
 \left(\exp\bb{-(1+o(1))\la^{d_0+K-t-t'-t_2-2k}} +O(\La^{-10})\right),
\end{multline*}
where primes denote sums that run over multiples of $k$.
{From} Lemma~\ref{A2} we thus have
\begin{eqnarray*}
 \Pr(U) &\lesssim& n^2 \sum_{t,\,t'}'
  16\gamma_0^2 \eps^6 \las^{t+t'-2t_1} \left(\exp\bb{-(1+o(1))\la^{d_0+K-t-t'-t_2-2k}} +O(\La^{-10})\right), \\
 &=& o(1)+ n^2 \sum_{t,\,t'}'
  16\gamma_0^2 \eps^6 \las^{t+t'-2t_1} \exp\bb{-(1+o(1))\la^{d_0+K-t-t'-t_2-2k}},
\end{eqnarray*}
since there are at most $(\eps t^+)^2=O((\log\La)^2)$ terms in the double
sum (which has the same limits as before), so
the contribution of the $O(\La^{-10})$ term can be bounded by
$16n^2\gamma_0^2\eps^6(\log\La)^2 O(\La^{-10})=O(\La^{-8}(\log\La)^2)=o(1)$.

Taking the final term in the sums above, we have
$t$ and $t'$ at least $t^+-k$, so 
the exponent of $\la$ above
is at least $d_0+K-2t^+-t_2-4k=K/3-4k$, which is at least $K/4$
if $n$ is large. Hence the exponential term above is {\em always} at most
$\exp(-\la^{K/4}/2)$, say.
Taking the final term in the sum, the corresponding $\las^{\cdots}$ term
is at most
\[
 \las^{2t^+-2k-2t_1} = \las^{2t_0+2K/3-2k} \le \las^{2t_0} = \eps^{-6}n^{-2}.
\]
As $t+t'$ decreases from its maximum possible value in steps of $k$, the exponent
of $\la$ in the exponential increases by $k\sim 1/\eps \sim 1/\log\la$,
so the $\la^{\cdots}$ term increases by a factor that is asymptotically $e$
and certainly at least $2$. The $\las^{\cdots}$ term
increases by a factor of $\las^{-k}$ which is asymptotically $e$ and certainly at most
$3$. Also, after $r$ steps, there are at most $r+1$ ways of realizing
a given sum $t+t'$. It follows that
\[
 \Pr(U)\le o(1) + \sum_{r=0}^\infty 16\gamma_0^2(r+1)3^r\exp(-\la^{K/4}2^{r-1}),
\]
say. Since $\la^{K/4}\to\infty$, the exponential term in the final sum decreases
extremely rapidly, and the whole sum is dominated by its first term, which is $o(1)$.
This completes the proof of Theorem~\ref{th_eps}, assuming Lemmas~\ref{A2} and~\ref{A1}.
\end{proof}

Let us note for later, when we come to consider the distribution of the diameter,
that if we modify the definition of $E_{x,y,t,t'}$ by replacing
$d(x,y)\ge d_0+K$ by $d(x,y)\ge d_0-K$, then we obtain
\[
 \Pr(U) \le \sum_{r=0}^\infty 16\gamma_0^2(r+1)3^r\exp(-\la^{K/4-2K}2^{r-1}).
\]
Indeed, everything is as before except that the exponent of $\la$ has
decreased by $2K$. Now this new sum is large, but the contribution
from terms with $r\ge \log(\la^{3K})/\log 2 \sim 3K\eps/\log 2$
is still small. Hence, the sum from terms in which
one of $t,t'$ is smaller than $t^+$ by more than $3\floor{1/\eps}K\eps/\log 2\sim 3K/\log 2\le 5K$
is small. Since $\diam(G)\ge d_0-K$ whp, it follows that
whp the diameter is realized by vertices $x$ and $y$ which form
a regular far pair in which each vertex $z$ has
$t_0+t_1-5K\le t_{\omega/\eps}(z)\le t^+=t_0+t_1+K/3$. Since $K\eps$ may be taken
to tend to infinity arbitrarily slowly, this says that for a given error
probability, it suffices to consider regular far pairs
in which the vertices satisfy
\begin{equation}\label{bstar}
t_{\omega/\eps}(z)=t_0+t_1+O(1/\eps).
\end{equation}

It remains to prove Lemmas~\ref{A2} and~\ref{A1}.

\begin{proof}[Proof of Lemma~\ref{A2}.]
Recall that $t_1+3h\le t \le 10\eps^{-1}\log\La$, and $h=\eps^{-1}\log\log\La<t/2$.
Let $A=A(x)$ denote the event that $x$ is $(h,t)$-acceptable, and
$B_t=B_t(x)$ the event that $0<|\Ga_r(x)|<\omega/\eps$ holds for $0\le r\le t$.
Our aim is to bound the probability of $A\cap B_t$; note that this event
depends only on $\Gat(x)$.

To avoid dependence, we'd like to work with the branching process rather than the graph,
but we cannot assume that the relevant neighbourhoods of $x$ are trees.
So let us model the pair $(\Gatm(x),\Gat(x))$
by a pair $(T^\star,G^\star)$ as follows: first construct the branching process 
$(X_r)_{0\le r\le t}$, keeping track of the order in which the particles are
born, as in the proof of Lemma~\ref{spcpl}.
Let $T^\star$ be the corresponding labelled rooted tree of height at most~$t$.
Given $T^\star$, i.e., given $(X_r)$, form $G^\star$ by starting with $T^\star$ and adding
each of the following `potential extra edges' independently with probability $\la/n$:
all possible edges within $X_r$ and, for each $v\in X_r$, all possible
edges from $v$ to children (in $X_{r+1}$) of earlier particles $v'\in X_r$.
The potential extra  edges correspond to edges that would not have been tested in the graph
exploration, so the conditional distribution of $G^\star$ given $T^\star$ is the same
as that of $\Gat(x)$ given $\Gatm(x)$ (with an order on the vertices).
If $(T_0,G_0)$ is any possible value of $(\Gatm(x),\Gat(x))$ consistent with $A\cap B_t$, then
since $B_t$ holds, $T_0$ is a tree to which Lemma~\ref{spcpl} applies.
So $\Pr(\Gatm(x)\isom T_0) \sim \Pr(T^\star\isom T_0)$.
It follows that $\Pr(\Gat(x)\isom G_0)\sim \Pr(G^\star\isom G_0)$.
Hence, $\Pr((T^\star,G^\star)\in A\cap B_t)$ is asymptotically equal to the probability
that $(\Gatm(x),\Gat(x))\in A\cap B_t$. From now on we consider the model
$(T^\star,G^\star)$, forgetting about the graph $G(n,\la/n)$.

For technical reasons we modify $G^\star$ slightly as follows:
recalling that each set $X_r$ comes with an order, we only test
for possible extra edges $vw$ when both endvertices are among
the first $\omega/\eps$ vertices in the relevant set(s) $X_r$.
This does not affect the probability
of $A\cap B_t$, since when $B_t$ holds (which is determined by $T^\star$), 
the distribution of $G^\star$ given $T^\star$ is unchanged.

Let $S$ be the set of particles in $X_h$ with descendants in $X_t$.
To achieve independence between $A$ and $B_t$, let us weaken $B=B_t$ to $B'=B'_t$, the condition
that for every $v\in S$, the number of descendants of $v$ in each $X_r$, $h\le r\le t$,
is at most $\omega/\eps$. Our aim is to bound $\Pr(A\cap B)$ by $\Pr(A\cap B')$; to evaluate
the latter we estimate $\Pr(B')$ and $\Pr(A\mid B')$.

Our first aim is to show that
\begin{equation}\label{bsim1}
 \Pr(B')\sim \Pr(B)=p_0\sim 4\eps\las^{t-t_1},
\end{equation}
where the final estimate is from \eqref{Bsim}.
Note that $\Pr(B'\mid |S|=s)=p^s$, where $p$ is the
(unconditional) probability that $0<|X_r|< \omega/\eps$
holds for $0\le r\le t-h$. From Theorem~\ref{th19},
we have $p\sim 4\eps\las^{t-h-t_1}\sim p_0\las^{-h}$.
Also, since $t\ge t_1+3h$, we have $p\le (1+o(1))\las^{2h}$.
Since $\eps h\to\infty$ it follows that $p^2\le (1+o(1))\las^{h}p_0=o(p_0)$.

Since $B\subset B'$, we have
\[
 \Pr\bb{B\cap\{|S|\ge 2\}}\le \Pr\bb{B'\cap\{|S|\ge 2\}} \le \Pr\bb{B'\mid\{|S|\ge 2\}}\le p^2=o(\Pr(B)).
\]
Recalling that if $B$ or $B'$ holds then $|S|\ge 1$, to show that $\Pr(B')\sim \Pr(B)$
it suffices to show that $\Pr(B\cap\{|S|=1\})\sim \Pr(B'\cap\{|S|=1\})=p$.
Let $B''$ be a strengthened version of $B'$, where we replace the upper
bound $\omega/\eps$ by $\omega'/\eps$, with $\omega'=(1-1/\log\La)\omega\sim \omega$.
Applying Theorem~\ref{th19} again with this new value of $\omega'$,
we find that $\Pr(B''\mid |S|=1)\sim p$. But given that $|S|=1$ and $B''$ holds,
$B$ certainly holds as long as the tree $T$ formed by the descendants of the root
that are not descendants of the unique particle in $S$ contains at most
$\omega/(\eps\log\La)>\La^{1/10}/\eps$ particles in each generation.

The distribution of $T$ is dominated by that of the tree $T'$ formed
by starting one copy of $\bp_{\las}$ in each generation $0\le t<h$.
(In $T$ these copies are conditioned to die by a specific time.)
The first $h-1$ generations of $T'$ have exactly the distribution of
the process $(D_t)$ studied in Lemma~\ref{dthin}.
Hence, by the second part of that lemma, the probability that one of the first $h$
generations of $T$ exceeds size $\La^{1/20}/\eps$ is
$O(\eps h e^{-\Omega(\La^{1/20})})=o(1)$. From generation $h$ onwards, the tree
$T$ evolves as a subcritical branching process, and from a standard martingale argument
the probability that any later generation exceeds the size of generation $h$ by a factor
of $\La^{1/20}$ is at most $1/\La^{1/20}=o(1)$.
Thus we do indeed have $\Pr(B\mid B''\cap\{|S|=1\})\sim 1$, and it follows that
$\Pr(B')\sim \Pr(B)$, as claimed.

Recalling that $p^2=o(\Pr(B))$ and hence $p^2=o(\Pr(B'))$, for $r\ge 2$ we have
\[
 \Pr(|S|=r\mid B') \le \Pr(B'\mid |S|=r)/\Pr(B') = o(p^{r-2}).
\]
Summing, it follows that
\begin{equation}\label{esim1}
 \E(|S|\mid B')\sim 1.
\end{equation}

We claim that (in the modified $G^\star$ model)
\begin{equation}\label{cs1}
 \Pr( A \mid B'\cap\{|S|=N\}) \le (1+o(1)) \gamma_0N\eps^2.
\end{equation}
Using
$\Pr(A\mid B')=\sum_{N\ge 1} \Pr(|S|=N\mid B')\Pr( A \mid B'\cap\{|S|=N\})$,
and \eqref{esim1} and \eqref{bsim1},
the required bound \eqref{ubaim} on $\Pr(A\cap B)\le \Pr(A\cap B')$
then follows.

It remains only to prove \eqref{cs1}.
Recall that we are working with the model $(T^\star,G^\star)$.
Let us construct $T^\star$ (which is simply the first $t$ generations of $\bp_\la$)
by decomposing it into the trunk and side branches exactly as in the graph.
Thus the {\em trunk} consists of the subtree $T'$ of $T^\star$ consisting of all particles
with descendants in $X_t$.
Then $T^\star$ may be formed by adding for each $v$
in generation $r$, $0\le r<t$, of $T'$ a copy $W_v$ of the process $(X_{t'})_{0\le t'\le t-r}$
conditioned on $X_{t-r}$ being empty. We may think of $W_v$ as the subcritical
process $\bp_{\las}$ conditioned on dying out by time $t-r$.

Now whether $B'$ holds is determined by $T'$ together with the trees $W_v$ for
vertices $v$ in sets $X_r$, $r\ge h$. Let us condition on $T'$ and these trees $W_v$;
the only remaining randomness is in the $W_v$ for $v\in X_r$, $r<h$.

Let $v$ be one of the $N=|S|$ vertices in $S$, and let $x=v_0v_1\cdots v_h=v$ be the path to $v$.
Let $W_i=W_{v_i}$, for $0\le i\le h-1$. Let $A_v$ be the event
that every $W_i$ is either short or, when we come to $G^\star$, reattached.
Note that $A$ holds if and only if one of the events $A_v$ holds,
so it suffices to prove that the conditional probability of $A_v$
is $(1+o(1))\gamma_0\eps^2$.
Since the different $W_w$ are independent given $T'$, the conditional distribution
of each $W_i$ (given $T'$ and the $W_w$, $w\in X_r$, $r\ge h$)
is just the unconditioned distribution.
Writing, as before,
$s_i=\Pr(|X_i^-|\ge 0)$ and $d_i=1-s_i$, 
the probability that $W_i$ is {\em tall} (not short) is just
\begin{multline}\label{sidepi}
  p_i = \Pr\bb{ |X_{i+1}|>0 \bigm| |X_{t-i}|=0 } = \Pr\bb{ |X_{i+1}^-|>0 \bigm| |X_{t-i}^-|=0 } \\
  = \frac{d_{t-i}-d_{i+1}}{d_{t-i}} = \frac{s_{i+1}-s_{t-i}}{1-s_{t-i}}
 = s_{i+1}-O(s_{t-i}) = s_{i+1}-O(s_h),
\end{multline}
since $t\ge 2h$ and $i\le h$.

Now $h\ge 1/\eps$, so (by Lemma~\ref{lsurv}) $s_h=O(\eps\las^h)$.
Let $\wN$ be the number of tall $W_i$. Then from the estimate above and Lemma~\ref{lsurv},
\[
 \Pr(\wN=0) = \prod_{i=0}^{h-1} (1-p_i) =\exp(O(h\eps\las^h))\prod_{i=1}^h (1-s_i)
\sim \prod_{i=1}^\infty (1-s_i) \sim \gamma_0\eps^2,
\]
since $\las^h=\exp(-(1+o(1))\eps h)$ and $\eps h\to\infty$, so $h\eps\las^h=o(1)$.
It thus suffices to show that $\Pr(A_v)\le (1+o(1))\Pr(\wN=0)$;
then \eqref{cs1} follows by the union bound.
In other words, we must show that $A_v\cap \{\wN>0\}$ is much less likely than $\wN=0$.

Let $I$ be any subset of $\{0,1,2,\ldots,h-1\}$ with $|I|\ge1$, and let us condition
on precisely the corresponding trees $W_i: i\in I$ being tall.
Let $M_i$, $i\in I$, be the number of vertices in each tall tree $W_i$, noting
that these numbers are conditionally independent.
Given that a particular $W_i$ is tall, its average size is at most that of $\bp_{\las}$
conditioned to survive to height $i+1$ (we also condition on dying out by height $t-i$).
By Lemma~\ref{tall} this is at most $(i+2)/\eps$.

Let us now go through the tall trees in order, checking to see whether each is reattached.
(We will be forced to skip some; see below.)
Due to the way we modified $G^\star$, when checking if $W_i$ is reattached, for each
vertex $u$ of $W_i$ with $u\in X_r$ we need only check for edges of $G^\star\setminus T^\star$
between $u$ and up to $\omega/\eps$ vertices in each of $X_{r-1}$, $X_r$
and $X_{r+1}$.
For each $u$, the probability of finding such an edge is at most
$p=3(\omega/\eps)\la /n\le 4\omega\eps^{-1}n^{-1}$.
The probability that the tall tree $W_i$ reattaches is thus at most
$\E(M_ip)=\E(M_i)p \le (i+2)\eps^{-1}p \le 4(i+2)\omega\eps^{-2}n^{-1}$.

When testing whether the first tall tree does reattach,
we stop if we find one edge witnessing this. This edge may `spoil'
a later tall $W_j$ by going to a vertex of that $W_j$.
For $J\subset I$, let $E_J$ be the event that the tall trees $W_j:j\in J$ are reattached
by $|J|$ edges each with one end in the appropriate $W_j$ and the other outside $\bigcup_{i\in J}W_i$.
Given all the trees, the conditional probability of $E_J$ is at most
$\prod_{j\in J} M_jp$. Since, conditioning only on which
trees are tall but not their sizes, the $M_j$ are independent,
it follows that
\[
 \Pr(E_J\mid I) \le \prod_{j\in J}\frac{4(j+2)\omega}{\eps^2 n}.
\]
If all $W_i:i\in I$ are reattached, then the testing algorithm above shows
that $E_J$ must hold for some $J$ containing at least half of the
first $k$ elements of $I$ for every $k\le |I|$, corresponding to the fact
that we test trees in order, and each spoils at most one later one.
Hence,
\[
 \Pr(A_v\mid I) \le \sum_J \prod_{j\in J} \frac{4(j+2)\omega}{\eps^2 n},
\]
with the sum restricted as above.

Suppose that $|I|=2k-1$ or $|I|=2k$, and list the elements of $I$ as $i_1$, $i_2,\ldots$ in order.
There are at most $4^k$ terms in the sum, and the largest has $J=\{i_1,i_3,i_5,\ldots,i_{2k-1}\}$,
so given $I$, the probability of reattachment is at most
\[
 \frac{16(i_1+2)\omega}{\eps^2n}\frac{16(i_3+2)\omega}{\eps^2n}\cdots \frac{16(i_{2k-1}+2)\omega}{\eps^2n}.
\]

Now the probability that the tall trees are exactly those indexed by $I$ is
\[
 \Pr(\wN=0)\prod_{i\in I}\frac{p_i}{1-p_i} \le \Pr(\wN=0)\prod_{i\in I}3p_i \le \Pr(\wN=0)\prod_{i\in I}\frac{10}{i+2},
\]
say, noting that $p_i\le s_{i+1}$
and using the crude upper bound $3/(i+2)$ for $s_{i+1}$.
Summing over $I$ with $|I|\ge 1$ we find that
\begin{multline*}
 \Pr(A_v\cap\{\wN>0\})\\ \le \Pr(\wN=0) \sum_{r\ge 1}\, \sum_{0\le i_1<i_2<\cdots<i_r<h} \frac{10}{i_1+2}\frac{10}{i_2+2}\cdots\frac{16(i_1+2)\omega}{\eps^2n}\frac{16(i_3+2)\omega}{\eps^2n}\cdots.
\end{multline*}
The sum over even $r$, say $r=2k$, may be crudely bounded by $\sum_{k=1}^\infty S^k$, where 
\[
 S = \sum_{0\le a<b< h} \frac{10}{a+2}\frac{10}{b+2}\frac{16(a+2)\omega}{\eps^2n} \le 
\sum_{b< h} b\frac{1600}{b+2}\frac{\omega}{\eps^2 n} \le 1600h\omega\eps^{-2}n^{-1}.
\]
Since $h\le \eps^{-1}\log\log\La$, we have $S=o(1)$. Bounding the
sum over odd $r$ similarly, it follows
that $\Pr(A_v\cap\{\wN>0\})=o(\Pr(\wN=0))$, as required.
\end{proof}

Finally, we prove Lemma~\ref{A1}.
\begin{proof}[Proof of Lemma~\ref{A1}.]
Throughout this proof, let $K=\ceil{\log(1/\eps)}$ and, for $1\le k\le K$,
let $t_k=\eps^{-1}/(k\psi)$. (We ignore the irrelevant rounding to integers,
noting that $t_K\to\infty$.)

For $2\le k\le K$ let $E_k(x)$ denote the event that $t_k<t_{\omega/\eps}(x)\le t_{k-1}$,
the neighbourhoods of $x$ to distance $t_{\omega/\eps}(x)$ form a tree, and $x$
is $(t_k/2,t_k)$-acceptable.
Let $E_1(x)$ denote the event that $B_{t_1}$ holds (i.e., $0<|\Ga_t(x)|<\omega/\eps$
for $0\le t\le t_1$), that $\Gl{t_1}(x)$ is a tree, and $x$ is $(t_1/2,t_1)$-acceptable.
Finally, let $E_{\infty}(x)$ denote the event that $t_{\omega/\eps}(x)\le t_K$.
Splitting into cases according to the value of $t_{\omega/\eps}(x)$, we see
that if $A^*(x)$ holds, then so does one of the events $E_1(x)$, $E_\infty(x)$ or $E_k(x)$,
$2\le k\le K$.

Let us start with a simple branching process observation related to that in Lemma~\ref{dthin},
writing $t_{\omega/\eps}$ for $\min\{t: |X_t|\ge \omega/\eps\}$, as before, whenever this is defined.
Suppose we have chosen some $t\ge 1$ in advance. If we explore the branching process
step by step and find a generation $X_r$, $r\le t$, with size at least $\omega/\eps$,
then it is easy to see that the conditional probability that $|X_t|\ge \omega/\eps$ 
is at least $1/10$, say.
Thus $\Pr(|X_t|\ge \omega/\eps)\ge \Pr(t_{\omega/\eps}\le t)/10$,
and hence
\begin{equation}\label{bpbig}
 \Pr(t_{\omega/\eps} \le t) \le 10 \Pr(|X_t|\ge \omega/\eps).
\end{equation}
Using this observation and Lemma~\ref{fast}, we see that
\[
\Pr(t_{\omega/\eps}\le t_K)\le 10t_K^{-1}e^{-\omega\eps^{-1}t_K^{-1}/20} = 10t_K^{-1}e^{-\omega\psi K/20}
 = o(\eps^3),
\]
since $\omega\psi\to\infty$ while $K\ge \log(1/\eps)$. Comparing the graph
and branching process as usual, it follows that $\Pr(E_\infty(x))=o(\eps^3)$.

Turning to $E_k(x)$ for $1\le k\le K$, note that we may test whether this event
holds by exploring at most $t_1=O(1/\eps)$ steps from $x$, stopping if we reach
a neighbourhood of size $\omega/\eps$,
and then checking that the neighbourhoods so far form a tree,
and satisfy the relevant acceptability conditions.
Arguing as above~\eqref{A'xy},
Lemma~\ref{spcpl} thus gives $\Pr(E_k(x))=(1+o(1))\Pr(E_k)+O(n^{-100})$, where
$E_k$ is the branching process event corresponding to $E_k(x)$. It thus suffices to show that
\begin{equation}\label{eksum}
 \sum_{k=1}^K \Pr(E_k) = O(\eps^3\psi^8).
\end{equation}
This statement involves only the branching process $\bp_\la$, so from now on we work
with this rather than the graph.

Let $A_k$ be the event that the branching process satisfies the condition
corresponding to $(t_k/2,t_k)$-acceptability.
To simplify the arguments, for $2\le k\le K$ let $E_k'$
be the event that $A_k$ holds and
%
$|X_{t_{k-1}}|\ge \omega/\eps$.
Only the second condition involves generations beyond $t_k$, so arguing as for \eqref{bpbig}
we have $\Pr(E_k'\mid E_k)\ge 1/10$, and hence $\Pr(E_k)\le 10\Pr(E_k')$.
Also, let $E_1'$ be the event that $A_1$ holds and $|X_{t_1}|>0$.
Then $E_1'\supset E_1$. Hence
\begin{equation}\label{Ekk'}
 \Pr(E_k)\le 10\Pr(E_k')
\end{equation}
for all $1\le k\le K$.

For $k\le 2$ let $L_k$ be the event that $|X_{t_{k-1}}|\ge \omega/\eps$;
let $L_1$ be the event that $|X_{t_1}|>0$, so $E_k'=A_k\cap L_k$.
As before, let $T$ be the {\em trunk} of $\bp_\la$ defined up to generation
$t_k$, so $T$ is the random tree consisting of all particles with descendants
in $X_{t_k}$. If we condition on the first $t_k$ generations
of $\bp_\la$, then the conditional probability of $L_k$ depends only on $|X_{t_k}|$.
Since knowing the trunk $T$ determines $|X_{t_k}|$, we thus have
\[
 \Pr(E_k') = \Pr(A_k\cap L_k) = \sum_{T'} \Pr(T=T') \Pr(A_k\mid T=T')\Pr(L_k\mid T=T'),
\]
where the sum runs over all possible trunks $T'$. Note that we may assume $T'$
is non-empty, i.e., $|X_{t_k}|>0$, as otherwise $L_k$ cannot hold.

As before, given the trunk, we may reconstruct $X_{\le t_k}$ by adding independent
random branches to each trunk vertex, with each branch a copy of $\bp_{\las}$
conditioned to die by (absolute, not relative) time $t_k$.
Let $S$ be the set of trunk vertices in generation $t_{k/2}$, and $N=|S|$
the number of such vertices, so $N$ is random but depends only on $T$.
Since we are considering the branching process, which is by definition a tree,
the acceptability condition $A_k$ holds if and only if some
$v\in S$ has the property that the side branch
started at each $v_i$ has height at most $i$ for all $0\le i\le t_k/2$,
where $v_0v_1v_2\cdots v_{t_k/2}=v$ is the chain of ancestors of $v$.
For a given $v$, the probability of this event is exactly
$\prod_{i=0}^{t_k/2} (1-p_i)$, where $p_i$ is given by \eqref{sidepi}
with $t=t_k$ and $h$ replaced by $t_k/2$. (The argument is as for \eqref{sidepi}.)
It follows easily from the estimates in Lemma~\ref{lsurv} that $s_{t_k/2}=O(t_k^{-1})$ and that
\[
 \prod_{i=0}^{t_k/2} (1-p_i) =\Theta(1)\prod_{i=1}^{t_k/2} (1-s_i) =O(t_k^{-2}).
\]
So far we considered a single $v\in S$; by the union bound it follows that
$
 \Pr(A_k\mid T=T') \le C t_k^{-2} N(T')
$
for some absolute constant $C$. Hence,
\begin{equation}\label{Ek2}
 \Pr(E_k') \le Ct_k^{-2} \sum_{T'} \Pr(T=T') N(T') \Pr(L_k\mid T=T').
\end{equation}

Let $n_0=n_0(k)=(k\psi)^5$. Let $\mu_k^-$ and $\mu_k^+$ denote respectively
the contributions to the sum in \eqref{Ek2} from trunks $T'$ with $N(T')\le n_0$
and $N(T')>n_0$, so $\Pr(E_k')\le \mu_k^-+\mu_k^+$.
Trivially, we have
\begin{equation}\label{mumb}
 \mu_k^- \le Ct_k^{-2} n_0 \sum_{T'} \Pr(T=T')\Pr(L_k\mid T=T') = Ct_k^{-2} n_0 \Pr(L_k).
\end{equation}
For $k=1$ we have $\Pr(L_1)=\Pr(|X_{t_1}|>0)$.
Writing $\surv$ for the event that the whole process survives, we have
\[
 \Pr(|X_t|>0) = s+(1-s)\Pr(|X_t|>0\mid \surv^\cc) =s+(1-s)\Pr(|X_t^-|>0).
\]
By Lemma~\ref{lsurv}, it follows that for $t=o(1/\eps)$ we have
\begin{equation}\label{lastt}
 \Pr(|X_t|>0)\sim 2/t.
\end{equation}
In particular, $\Pr(L_1)=O(1/t_1)$, so
from \eqref{mumb}
\[
 \mu_1^- =O(t_1^{-3}n_0) = O(\eps^3\psi^3 \psi^5) = O(\eps^3\psi^8).
\]
For $k\ge 2$, from Lemma~\ref{fast} we have
\[
 \Pr(L_k) = \Pr\bb{|X_{t_{k-1}}|\ge \omega/\eps} \le t_{k-1}^{-1}e^{-\omega\psi(k-1)/20},
\]
so, from \eqref{mumb},
$\mu_k^-\le 10Ct_k^{-2}n_0t_{k-1}^{-1}e^{-\psi\omega (k-1)/20}$.
Recalling that $t_k=\eps^{-1}/(\psi k)$ and $n_0=n_0(k)=k^5\psi^5$, it follows that
\[
 \sum_{k=2}^K \mu_k^- \le \sum_{k\ge 2} 10C \eps^3 k^8\psi^8 e^{-\psi\omega(k-1)/20}.
\]
Since $\omega$ and $\psi$ are large for $n$ large, the first term dominates, and this
sum is $o(\eps^3)$. Together with the bound for $\mu_1^-$ above this gives
\begin{equation}\label{mkm}
 \sum_{k=1}^K \mu_k^- = O(\eps^3\psi^8).
\end{equation}

It remains to bound $\mu_k^+$. Noting that $N(N-1)\ge n_0N$ whenever $N>n_0$, 
and that $\Pr(L_k\mid T=T')\le 1$, from \eqref{Ek2} we have
\[
 \mu_k^+ \le Ct_k^{-2} \sum_{T'} \Pr(T=T') n_0^{-1} N(T')(N(T')-1)
 = Ct_k^{-2} n_0^{-1} \E\bb{N(N-1)},
\]
where the final expectation is unconditional.
Given $X_{t_k/2}$, each particle in this generation survives to generation
$t_k$ independently with probability $p=\Pr(|X_{t_k/2}|>0)=O(t_k^{-1})$,
from \eqref{lastt}. Hence
\[
 \E\bb{N(N-1)} = p^2 \E\bb{|X_{t_k/2}|(|X_{t_k/2}|-1)}
 = O(t_k^{-2}) \E\bb{|X_{t_k/2}|(|X_{t_k/2}|-1)}.
\]
A simple inductive formula, or a tree counting argument, gives
$\E\bb{|X_t|(|X_t|-1)} =\la^t(\la+\la^2+\cdots+\la^t) \le t\la^{2t}$.
With $t=t_k/2\le 1/\eps$, this is $O(t_k)$, so
$\E\bb{N(N-1)} = O(t_k^{-1})$.
Hence,
\[
 \mu_k^+=O(t_k^{-3}n_0^{-1}) = O\bb{\eps^3\psi^3k^3(k\psi)^{-5}} = O(\eps^3 k^{-2}).
\]
Thus $\sum_{k=1}^K\mu_k^+=O(\eps^3)$. Recalling that $\Pr(E_k')\le \mu_k^-+\mu_k^+$,
and using \eqref{mkm} and \eqref{Ekk'}, this establishes \eqref{eksum}. As
noted earlier, the lemma follows.
\end{proof}

\begin{remark}
As noted earlier, in the first draft of this paper we needed the condition
$\La\ge e^{(\log^* n)^4}$. The changes that allowed us to eliminate
this are the introduction of Lemma~\ref{noshort} (making checking
for acceptability in the case when $t_{\omega/\eps}(x)=o(1/\eps)$ much simpler),
the modification of Lemma~\ref{A1} to include the tree condition,
and the new proof of Lemma~\ref{A1} above.
\end{remark}

\section{The distribution of the correction term}\label{sec_dist}

In this section we shall describe the limiting distribution
of the correction term in Theorem~\ref{th_eps} and, very briefly,
that in Theorem~\ref{th_const}.
Surprisingly, although Theorem~\ref{th_eps} is much harder to prove
than Theorem~\ref{th_const}, the study of the correction term
is much easier in the former case. Indeed, with $p=\la/n$
and $\la$ constant, even the description of the correction term is rather complicated.
Let us start with the simpler case, assuming that $\la=1+\eps$ with $\eps=\eps(n)\to 0$.
It turns out that given the results of the previous section, not much
extra work is needed to obtain the distribution. Essentially, only one
natural extra idea is needed. Since the formal details would take
some time to write out, we shall only sketch the arguments.

In \refSS{ss_low}, we obtained a lower bound on the diameter by
considering vertices $x$ with a certain property $\FF=\FF_q$, $q=q_0$,
depending on the $t=t_0+t_1+q$ neighbourhoods, where $|q\eps|\le M$ was
essentially bounded. (We shall repeatedly use the observation
that if some probability is $o(1)$ uniformly in $|q\eps|\le M$ for any
constant $M$, then it is $o(1)$ uniformly in $|q\eps|\le M$ if $M=M(n)$ tends to
infinity slowly enough. It is often easier to think
of $M$ as constant, although in the end we need $M\to\infty$.)

One aspect of this property $\FF_q$, or rather of the related property $\tF_q$, was that
in the tree $T_x$ containing $x$ and attached to the 2-core,
$x$ is the unique vertex at maximal distance from the 2-core.
It turns out that a positive fraction of the trees attached
to the 2-core have more than one vertex at maximal distance,
and to obtain a precise result we must also consider such trees.
But we must only count each tree once.
The solution is very natural: we consider an auxiliary random
order $\prec$ on $V(G)$, and consider only vertices $x$
such that, writing $S_x$ for the set of vertices
of $T_x$ at maximal distance from the 2-core,
$x$ is the first vertex of $S_x$ in the order $\prec$.

More precisely, we modify the definition of
the branching process events $E_q$ and $F_q$, by weakening
the `strong wedge condition' $B$(i) on page~\pageref{Bi}: instead of insisting that the
`side branch' starting at generation $i$ dies within $i$
generations, we insist that it dies within $i+1$
generations (this is the {\em weak wedge condition}), and also,
writing $S_i$ for the set of particles
in the $i$th generation of the $i$th side branch, letting $S$ be
the union of the sets $S_i$ together with the initial particle,
and taking a random order on $S$, we insist that the initial particle
comes first in this order; we call this the {\em medium wedge condition}.

We showed that the probability of the strong wedge condition was 
asymptotically\break $d_1\prod_{i=1}^\infty d_i\sim d_1\gamma_0\eps^2 \sim e^{-1}\gamma_0\eps^2$,
where $d_i=\Pr(|X_i^-|=0)=1-s_i$ is the probability that the subcritical
process dies by time $i$.
Similarly, the probability of the weak wedge condition is asymptotically
$\prod_{i=1}^\infty d_i\sim \gamma_0\eps^2$.

If we condition on the weak wedge condition, then the distribution
of $S$ depends on $\eps$. However, the conditional probability that $S_i$
is non-empty is bounded by
\begin{multline*}
 \Pr\bb{|X_i^-|>0\bigm| |X_{i+1}^-|=0} = 1-\Pr\bb{|X_i^-|=0 \bigm| |X_{i+1}^-|=0} \\
  = 1-\frac{d_i}{d_{i+1}} =\frac{s_i-s_{i+1}}{1-s_{i+1}} \sim s_i-s_{i+1}.
\end{multline*}
{From} \eqref{uup} we have $s_i<2/i$ for all $i\ge 1$, so $\sum_i \Pr(S_i\ne\emptyset)$ converges
uniformly as $\eps\to 0$. Hence, for any $M(n)\to\infty$, the probability
that any $S_i$, $i>M$, is non-empty tends to $0$.
For fixed $i$, the distribution of $S_i$ converges as $\eps\to 0$,
in fact, to the distribution of the size of the $i$th generation of the exactly critical
process $\bp_1$ given that the $(i+1)$st generation is empty.
It follows that, in the branching process, $|S|$ converges in distribution
to some random variable $R$ not depending on $\eps$.
Modifying the arguments in \refSS{ss_low}, we find that when
we replace the strong wedge condition by the medium wedge condition,
in place of \eqref{FE} we obtain the estimate
\begin{equation}\label{Eqsim}
  \Pr(F_q(x))\sim \Pr(F_q)\sim 4\gamma_1 n^{-1}\las^q
\end{equation}
uniformly in $|q|\le M/\eps$,
where $\gamma_1=\E(1/R)\gamma_0$, and $\gamma_0$
is the constant in Lemma~\ref{lsurv}.

Turning to the upper bound, after much work mostly
involving ruling out pathological cases, we showed in \refSS{ss_up}
that for any function $M(n)$ tending to infinity,
whp any vertex $x$ that is part of a pair $(x,y)$ at maximal
distance satisfies the property $B^*(x)$, that $|t_{\omega/\eps}(x) -t_0-t_1|\le M/\eps$,
(see \eqref{bstar})
together with a certain unpleasant `acceptability' condition $A^*(x)$.
Moreover, Lemma~\ref{A2} shows that the expected number
of such vertices is bounded by some function of $M$.
Thinking of $M$ as constant for the moment, this expectation
is bounded. Now given that a vertex has property $B^*$, it is likely
that its relevant neighbourhood (up to $t_{\omega/\eps}$)
is a tree. (The expected number of edges within sets $\Ga_t(x)$
is bounded by $\delta=\la n^{-1}t_{\omega/\eps}
\binom{\omega/\eps}{2}=O(\omega^2(\log\La)\eps^{-3}n^{-1})=O(\La^{-2/3}\log\La)=o(1)$;
a similar
bound holds for the expected number of `redundant' edges between consecutive
$\Ga_t(x)$.)
We had to consider the non-tree case, because $\delta$ may go to zero only slowly,
but after reducing to vertices satisfying $B^*$, it is easy to check
from the proof of Lemma~\ref{A2} that the probability that $A^*\cap B^*$ holds
and the neighbourhood is not a tree is $o(\Pr(A^*\cap B^*))$.
It follows that (if $M$ increases slowly enough), the expected number
of vertices with $A^*\cap B^*$ holding and the neighbourhood not a tree is $o(1)$.

When considering tree neighbourhoods, acceptability becomes a much simpler
condition, closely related to the weak wedge condition. So far we considered
any vertex $x$ in a pair $(x,y)$ at maximal distance. Since we are only interested
in the existence of a pair at a certain distance, we may restrict
our attention to those $x$ that are first in their tree $T_x$ in our auxiliary
random order. For vertices satisfying $A^*\cap B^*$, the conditional probability
of this extra condition is asymptotically $\E(1/R)$, as above. Putting the pieces
together, we find that whp the diameter is realized by some pair of vertices
each of which satisfies a certain condition $F_q'$ depending
on its $t=t_0+t_1+q$ neighbourhood, where again $|q\eps|\le M$.
This condition is that $\Gat(x)$ is a tree,
and the event $A_t\cap B_t$ considered in Lemma~\ref{A2} 
modified to the medium wedge condition holds.
Also, modifying the proof of this lemma as indicated above, the
probability that a vertex satisfies this condition is
\[
 \Pr(F_q'(x))\sim \E(1/R) 4\gamma_0\eps^3\las^{t-t_1} \sim 4\gamma_1\las^q n^{-1}.
\]
Now the precise details of $F_q(x)$ and $F_q'(x)$ are rather different.
However, the definitions are such that $F_q(x)$ implies $F_q'(x)$.
(Firstly, in defining $F_q(x)$ we insisted that $\Gat(x)$ is a tree. Secondly,
via the condition $D=D_1\cap D_1'\cap D_2$, we ensured that $|\Ga_{t'}(x)|<\omega/\eps$
for $0\le t'\le t$. Thirdly, via $A$ we ensured that for all $t'$ up to
$t_0-r\ge t_0-2M/\eps$,
which is much larger than $h$,
there is a unique particle in each generation $t'$ with descendants
in $\Ga_t(x)$. Finally, we imposed the (there strong, but now medium)
wedge condition on all the side branches starting
up to time (at least) $t_0-2M/\eps$. This implies the (modified) form
of $(h,t)$-acceptability in $F_q'$.)

Since $\Pr(F_q'(x))\sim \Pr(F_q(x))$, and the expected number of vertices
with $F_q(x)$ is (for $M$ fixed) $\Theta(1)$, it follows that
for each $q$, whp {\em every} vertex with property $F_q'(x)$ also has $F_q(x)$.
We shall essentially consider only a bounded number of values of $q$
(again, a number that tends to infinity arbitrarily slowly),
so this holds whp for all such values.
Thus, whp, the diameter is equal to the maximum distance between
vertices with property $F_q(x)$ for suitable $q$. This also applies
if $M\to\infty$ slowly enough. We may thus forget
about $F_q'(x)$.

Now the condition $F_q(x)$ says that the (medium) wedge
condition holds, that $t(x)=t_{\omega/\eps}(x)>t_0+t_1+q$,
and that certain other technical conditions hold.
We shall need to know a little more, namely roughly how large $t(x)$
is. From the remarks above, we may ignore $x$ with $t(x)\ge t_0+t_1+M/\eps$.
For $-M^2\le i\le M^2$, let $q_i=i/(M\eps)$. Let us say that $x$
is {\em of type $i$} if $F_{q_i}(x)\setminus F_{q_{i+1}}(x)$ holds;
this corresponds roughly to the wedge condition plus $q_i< t(x)-t_0-t_1\le q_{i+1}$.
Let $N_i$ be the number of type $i$ vertices.
With $M$ constant, applying \eqref{Eqsim} twice shows that $\E N_i$
is asymptotically what it should be, and as usual this extends to $M\to\infty$
slowly enough, in which case
\[
 \E N_i \sim 4\gamma_1  e^{-q_i\eps}/M,
\]
since $\las^q=(1-\eps+O(\eps^2))^q \sim e^{-q\eps}$ if $q\eps$ does not grow too fast.

Let us say that $x$ is {\em plausible} if it is of type $i$ for some $-M^2\le i\le M^2$.
{From} the comments above, whp the diameter is realized by a pair of plausible vertices.

Now, the precise technical conditions in the definition of type $i$ vertices
are as in \refSS{ss_low}; as there, these allow us to calculate 2nd moments,
and indeed $r$th moments for any fixed $r$. More precisely, given a sequence
$\bi=(i_1,\ldots,i_r)$, let us say that a sequence $(x_1,\ldots,x_r)$ of distinct
vertices is an {\em $r$-tuple of type $\bi$} if each $x_j$ is of type $i_j$.
Such an $r$-tuple is {\em good} if the relevant trees witnessing this are disjoint,
and {\em bad} otherwise. Arguing as in \refSS{ss_low}, the expected
number of good $r$-tuples of type~$\bi$
is what it should be, namely $(1+o(1))\prod_{j=1}^r \E N_{i_j}$
(which is $\Theta(1)$ if $M$ is fixed),
and the expected number of bad $r$-tuples is $o(1)$.
This shows that all fixed mixed moments of
the sequence $(N_{-M^2},\ldots,N_{M^2})$
converge to what we expect, and thus that (for $M$ fixed) the sequence $(N_i)$ converges
in distribution to a sequence of independent Poisson random variables.

Turning to the diameter, let $P$ be the number of unordered {\em pairs} $(x,y)$ of plausible
vertices with $d(x,y)\ge d=d_0+c\eps^{-1}$, where $c$ is constant. We
aim to understand $\Pr(P>0)$ by evaluating the factorial moments
$\E_k(P) = \E(P(P-1)\ldots(P-k+1))$.
Now $\E_k(P)$ is the expected number of $k$-tuples of distinct pairs with the relevant property.
It may be that several pairs involve the same vertex; in general we can write
$\E_k(P)$ as a sum over integers $r\le 2k$ and graphs $H$ on $\{1,2,\ldots,r\}$ with $k$
edges of the expectation of the number of $r$-tuples of plausible
vertices in which certain specified pairs are at distance at least $d$ and the others are not.
We evaluate this by summing over the types of the relevant vertices.
Thus we must evaluate the expected number of $r$-tuples $(x_1,\ldots,x_r)$
of type $\bi$ in which $k$ specified pairs are at distance at least $d$ and the others are not.

Since there are $o(1)$ bad $r$-tuples, we consider only good $r$-tuples. Finally, we test
whether a particular sequence $(x_1,\ldots,x_r)$ has the required property by exploring
the neighbourhoods of each $x_j$ out to the relevant distance ($t_0+t_1+q_{i_j}$).
By Lemma~\ref{spcplk}, the probability that the explorations are disjoint and each $x_j$ is of the right
type is `what it should be', namely $n^{-r}$ times the expected number of good
$r$-tuples of type $\bi$. Suppose this happens. Then we have not so far
tested any edges outside these neighbourhoods.

Continuing to explore, the neighbourhoods grow at the expected rate whp. We explore
$t_2/2-O(1/\eps)$ further steps, by which time the neighbourhoods have size $\Theta(\sqrt{\eps n})$.
(Recall that this is the size at which they typically meet.)
By this time, there are very few (in expectation $O(1)$) vertices in two or more neighbourhoods,
and whp none in three or more. It follows that the times at which different pairs
of neighbourhoods meet are essentially independent, with distribution
given by Lemma~\ref{meet}. This allows us to calculate $\E_k(P)$,
and hence $\Pr(\diam(G(n,\la/n))\ge d)\sim \Pr(P>0)$.

Rather than give any further details, let us describe the limiting distribution
we obtain. It should then be clear that all expectations being `what they should be'
corresponds to convergence to the corresponding values for this limiting distribution.

Let $\PP$
be a Poisson process on $\R$ with density function $f(x)=4\gamma_1e^{-x}$.
Note that $\int_{x'\ge x} f(x') \dd x'= f(x)<\infty$ for any $x$, so with probability
1 we may list the points of $\PP$ as $z_1,z_2,\ldots$ in decreasing order.
For each $1\le i<j$, let $T_{ij}$ be a random variable
with $\Pr(T_{ij}>x)=\exp(-e^x)$, with these variables independent of each other and
of $\PP$.
Finally, let $D=\sup\{z_i+z_j+T_{ij}\}$. It is not hard to check
that with probability 1 $D$ is finite, and the supremum is attained.
Indeed, as $M\to\infty$, the probability that it is attained by some
$i$, $j$ with $z_i,z_j\ge -M$ tends to $1$.

\begin{theorem}\label{th:dist1}
Let $\eps=\eps(n)>0$ satisfy $\eps\to 0$ and $\eps^3n\to\infty$,
and let $\la=1+\eps$.
For any constant $c$ we have
\[
 \Pr\left(  \diam(G(n,\la/n)) \ge \frac{\log(\eps^3 n)}{\log\la} + 2\frac{\log(\eps^3 n)}{\log(1/\las)} +c/\eps \right) \to \Pr(D\ge c)
\]
as $n\to\infty$.\noproof
\end{theorem}
In other words, the $\Op(1/\eps)$ correction term in \eqref{deps}
converges in distribution to $D$ (after multiplication by $\eps$).

We have proved Theorem~\ref{th:dist1} in outline above. There are
a few further technical details (such as checking that the relevant
sequences of moments do not grow too fast, so convergence
of all fixed moments gives convergence in distribution), but we
shall not describe these any further.

The description of the random variable $D$
is somewhat complicated; however, it seems rather unlikely that this random variable
will have a simpler description. Given this description, the branching
process approach taken here seems with hindsight very natural: the description
of $D$ more or less forces us to consider the (exponentially distributed)
times that the vertices take for their neighbourhoods to reach certain
very large sizes, and then the time they take to meet after this.

Finally, let us comment very briefly on the case $p=\la/n$, $\la$ constant.
It is not that the proof is any harder in this case (it is much easier),
but the result is much harder to describe. Again we consider vertices
satisfying the medium wedge condition (which now has probability
bounded away from $0$), and, taking $\omega=(\log n)^6$, say,
we study the distribution of $t_{\omega}(x)$
for such $x$,
in the range where $\Pr(t_\omega(x)\ge t_0)$ is of order $1/n$.
{From} Lemma~\ref{l1} it is very easy to check that when $t_{\omega}(x)$
is very large, this is almost always because for many generations there is only
one neighbour whose descendants do not die quickly, and we easily
find asymptotic independence of the event $\{t_{\omega}(x)>t_0\}$ and the wedge condition.

Approximating by a branching process, it is easy to prove
an equivalent of Theorem~\ref{th19}, showing
that the distribution of $t_\omega(x)$
may be described (as in the $\la\to 1$ case) by the tail
of $Y=Y_\la$ near 0. But now the first complication appears:
this random variable no longer has a nice power-law tail,
but asymptotically follows a power law multiplied by a function
that oscillates periodically within a constant factor.
Also, when we explore neighbourhoods and reach size $\omega$,
the current neighbourhood may have any size between $\omega$ and $\la\omega$;
this constant factor affects the probability of joining up with another neighbourhood
within a certain time.
In the end it turns out that the distribution depends on the fractional
parts of both $\log n/\log \la$ and $\log n/\log\las$, as indeed
it must from the form of~\eqref{dform}. We omit the details, as
a precise statement of the result would be rather lengthy.

\begin{acknowledgements}
This research started during the program `Random Graphs and Large-Scale Real-World Networks'
at the Institute for Mathematical Sciences, National University of Singapore in summer 2006;
the authors are grateful to the Institute for its support.
The authors would like to thank the anonymous referee for a careful reading of the paper
and many helpful suggestions concerning the presentation.
\end{acknowledgements}

\end{document}